\documentclass[]{article}

\usepackage{graphicx}
\usepackage{amsmath}
\usepackage{amssymb}
\usepackage{amsthm} 
\usepackage{bm}
\usepackage{enumerate}
\usepackage{color}
\usepackage{mathdots}
\usepackage{sectsty}
\usepackage{hyperref}
\hypersetup{hidelinks}
\usepackage{tikz}
\usepackage{caption}
\usepackage{adjustbox}
\usepackage{yhmath}
\usepackage{tikz-cd}
\usepackage{mathrsfs}
\usepackage{quiver}
\sectionfont{\scshape\centering\fontsize{12}{14}\selectfont}
\subsectionfont{\scshape\fontsize{12}{14}\selectfont}
\usepackage{fancyhdr}

\newcommand\shorttitle{Relation between crystals and $\mathcal{C}$-complexes}
\newcommand\authors{Arun Soor}

\newtheorem{thm}{Theorem}[section]
\newtheorem{cor}[thm]{Corollary}
\newtheorem{lem}[thm]{Lemma}
\newtheorem{scholium}[thm]{Scholium} 
\newtheorem{defn}[thm]{Definition}
\newtheorem{rmk}[thm]{Remark}
\newtheorem{example}[thm]{Example}
\newtheorem{prop}[thm]{Proposition}

\newtheorem*{claim*}{Claim}

\newcommand*{\sheafhom}{\mathcal{H}\kern -.5pt om}

\DeclareMathOperator{\prolim}{``lim''}
\DeclareMathOperator{\indlim}{``colim''}

\title{\large \bf $\wideparen{\mathcal{D}}$-modules are quasi-coherent sheaves on an analytic stack}
\author{Arun Soor}
\date{\today}
\begin{document}
\maketitle
\abstract{We construct a fully-faithful functor of $\infty$-categories from complexes of $\wideparen{\mathcal{D}}$-modules with Fr\'echet cohomology to quasi-coherent sheaves on an analytic stack. We prove various descent results for $\infty$-categories of $\wideparen{\mathcal{D}}$-modules in the analytic topology.}
\tableofcontents
\section{Introduction}
In algebraic or analytic geometry we frequently consider a ``sheaf theory" attached to a category $\mathscr{C}$ of geometric objects. For the purposes of this introduction only, let us say that a \emph{sheaf theory} is a functor 
\begin{equation}
    \mathscr{C}^\mathsf{op} \to \mathsf{Cat}_\infty,
\end{equation}
where $\mathsf{Cat}_\infty$ is the $\infty$-category of $\infty$-categories. For instance $\mathscr{C}$ could be the category of schemes, and $\mathscr{C}^\mathsf{op} \to \mathsf{Cat}_\infty$ could be the functor which assigns to each scheme $X$ the derived category $\operatorname{QCoh}(X)$ of quasi-coherent sheaves on $X$. Or, $\mathscr{C}$ could be the category of smooth schemes, and $\mathscr{C}^\mathsf{op} \to \mathsf{Cat}_\infty$ could be the functor which assigns to each scheme $X$ the derived category of $\mathcal{D}$-modules on $X$. Or, $\mathscr{C}$ could be the category of locally compact Hausdorff topological spaces, and $\mathscr{C}^\mathsf{op} \to \mathsf{Cat}_\infty$ assigns to each $X$ the derived category $D(X, \mathbf{Z})$ of abelian sheaves on $X$.

As remarked in the introduction to \cite{ScholzeSixFunctors}, in many cases the sheaf theory can be factored through a functor $ \mathscr{C}^\mathsf{op} \to  \{\text{analytic stacks}\}$, so that it can be written as the composite
\begin{equation}\label{eq:QCohcomposite}
   \mathscr{C}^\mathsf{op} \to  \{\text{analytic stacks}\} \xrightarrow[]{\operatorname{QCoh}} \mathsf{Cat}_\infty.
\end{equation}
This hypothesis could be summarized in the slogan
\begin{equation}\label{eq:universalsheaftheory}
    \text{``$\operatorname{QCoh}$ is the universal sheaf theory"},
\end{equation}
where $\operatorname{QCoh}$ stands for the sheaf theory of quasi-coherent sheaves on analytic stacks. In the preceding examples, one has that the sheaf theory of algebraic $\mathcal{D}$-modules factors through the assignment $X \mapsto X_{\mathrm{dR}}$ which sends a scheme to its (algebraic) de Rham space. The sheaf theory of abelian sheaves on locally compact Hausdorff spaces factors through the assignment $X \mapsto X_{\mathrm{sch}}$ sending $X$ to its schematization \cite{ToenChampsAffines}. The purpose of this article is to explain precisely how the $\wideparen{\mathcal{D}}$-modules of Ardakov and Wadsley  \cite{Dcap1}, which we can think of as a sheaf theory for rigid-analytic varieties, fits into the ansatz provided by \eqref{eq:QCohcomposite} and \eqref{eq:universalsheaftheory}.

Let us specialize to the setting of the problem. We consider $\mathsf{dRig}$, the category of \emph{derived rigid spaces} introduced in \cite{soor_six-functor_2024}. In \emph{loc. cit.} we associated to each $X \in \mathsf{dRig}$ an (analytic) prestack $X_\mathrm{str}$ called the \emph{stratifying stack}. The corresponding sheaf theory is 
\begin{equation}
    \mathsf{dRig} \xrightarrow{X \mapsto X_{\mathrm{str}}} \mathsf{PStk} \xrightarrow[]{\operatorname{QCoh}} \mathsf{Cat}_\infty, 
\end{equation}
in which $\mathsf{PStk}$ stands for the category of (analytic) prestacks and the value of $\operatorname{QCoh}$ on such is defined as in \cite{soor_six-functor_2024}. We aim to relate $\operatorname{Strat}(X) := \operatorname{QCoh}(X_{\mathrm{str}})$ to Ardakov--Wadsley's $\wideparen{\mathcal{D}}$-modules.

In \cite{soor_six-functor_2024}, in the spirit of \cite{MR3264953}, we considered the canonical morphism of prestacks $p: X \to X_{\mathrm{str}}$. When this morphism is thought of as a covering or atlas, there is an expectation to describe objects of $\operatorname{QCoh}(X_{\mathrm{str}})$ in terms of objects of $\operatorname{QCoh}(X)$ equipped with some kind of ``descent datum". In order to get the relation to modules over an algebra, we are looking for the action of a descent monad (rather than a comonad) and so we are led to consider the $!$-push-pull adjunction 
\begin{equation}
    p_!: \operatorname{QCoh}(X) \leftrightarrows \operatorname{QCoh}(X_{\mathrm{str}}) : p^!,
\end{equation}
on quasi-coherent sheaves. In \cite{soor_six-functor_2024} we defined the monad $\mathcal{D}^\infty_X := p^!p_!$ and proved the following.
\begin{thm}\cite[Theorem 4.101]{soor_six-functor_2024}\label{thm:previousthm} 
Suppose that $X = \operatorname{dSp}(A)$ is a smooth classical affinoid rigid space which is \'etale over a polydisk. Then the adjunction $p_! \dashv p^!$ is monadic, so that there is an equivalence of $\infty$-categories 
\begin{equation}
    \operatorname{QCoh}(X_{\mathrm{str}}) \simeq \operatorname{Mod}_{\mathcal{D}^\infty_X}\operatorname{QCoh}(X).
\end{equation}
\end{thm}
In light of Theorem \ref{thm:previousthm}, the main task of this article is to better understand the monad $\mathcal{D}^\infty_X$. 
For functional-analytic reasons, the monad $\mathcal{D}^\infty_X$ is not always given by tensoring with $\mathcal{D}^\infty_X1_X$. Nonetheless, we have the following. 
\begin{thm}[= Lemma \ref{lem:naturaltransformlema}] Let $X = \operatorname{dSp}(A)$ be a derived affinoid algebra. Then the object $\mathcal{D}_X^\infty1_X$ acquires the canonical structure of an algebra object\footnote{In particular this gives a definition of the ring of infinite-order differential operators even when $X$ is not classical or smooth. In these situations we do not expect the object $\mathcal{D}_X^\infty1_X$ to be concentrated in a single degree. It would be interesting to compute the graded ring $H^\bullet\mathcal{D}_X^\infty1_X$ in these cases.} in $\operatorname{QCoh}(X \times X)$ (under convolution). With this algebra object structure there is a canonical morphism of monads 
\begin{equation}\label{eq:intronaturaltransform}
    (-) \widehat{\otimes}_X \mathcal{D}^\infty_X1_X \to \mathcal{D}^\infty_X(-).
\end{equation}
\end{thm}
The natural transformation \eqref{eq:intronaturaltransform} is produced via some general machinery developed in Appendix \ref{sec:monoidsfromMonads}, which could be of independent interest. In short, one gets a transformation of the form \eqref{eq:intronaturaltransform} whenever the underlying endofunctor of a monad is \emph{lax-linear}\footnote{See Appendix \ref{sec:monoidsfromMonads}.}. Now the problem of describing modules over the monad is reduced to two sub-problems:
\begin{itemize}
    \item[] \emph{Problem 1}: Identify \emph{fixed points}, that is, some class of objects of $\operatorname{QCoh}(X)$ on which the natural transformation \eqref{eq:intronaturaltransform} restricts to an equivalence;
    \item[] \emph{Problem 2}: Identify the algebra structure on $\mathcal{D}^\infty_X1_X$. 
\end{itemize}
As the below Theorem shows, Problem 1 should be thought of as a functional-analytic condition, and for Problem 2 we can identify $\mathcal{D}^\infty_X 1_X$ with Ardakov--Wadsley's $\wideparen{\mathcal{D}}_X(X)$, at least when $X = \operatorname{dSp}(A)$ is classical and \'etale over a polydisk. 
\begin{thm}[= Theorem \ref{thm:PNFtensor}, Proposition \ref{prop:algebraiso}]
Assume that $X = \operatorname{dSp}(A)$ is classical and \'etale over a polydisk. 
\begin{enumerate}[(i)]
    \item Let $\operatorname{Fr}(X) \subseteq \operatorname{QCoh}(X)$ be the full subcategory spanned by complexes whose cohomology groups are Fr\'echet spaces. Then $\operatorname{Fr}(X) \subseteq \operatorname{QCoh}(X)$ is preserved by the monad $\mathcal{D}^\infty_X$ and the natural transformation \eqref{eq:intronaturaltransform} restricts to an equivalence on $\operatorname{Fr}(X)$.
    \item There is an equivalence of algebra objects\footnote{In $\operatorname{QCoh}(X \times X)$ with the convolution monoidal structure.} $\mathcal{D}^\infty_X1_X \simeq \wideparen{\mathcal{D}}_X(X)$, where the latter is defined as in \cite{Dcap1} and viewed as an algebra object in degree $0$.  
\end{enumerate}
\end{thm}
As a Corollary (using Theorem \ref{thm:previousthm} for the second part), we obtain the following:
\begin{cor}
Assume that $X = \operatorname{dSp}(A)$ is classical and \'etale over a polydisk.  
\begin{enumerate}[(i)]
    \item There is an equivalence of $\infty$-categories $\operatorname{RMod}_{\wideparen{\mathcal{D}}_X(X)} \operatorname{Fr}(X) \simeq \operatorname{Mod}_{\mathcal{D}_X^\infty} \operatorname{Fr}(X)$,
    \item There is a fully-faithful functor $\operatorname{RMod}_{\wideparen{\mathcal{D}}_X(X)} \operatorname{Fr}(X) \hookrightarrow \operatorname{QCoh}(X_{\mathrm{str}})$. 
\end{enumerate}
\end{cor}
For $X = \operatorname{dSp}(A)$ a smooth classical affinoid, we let $D_\mathcal{C}(X)$ denote Bode's category of $\mathcal{C}$-complexes\footnote{To be precise, we mean its $\infty$-categorical enhancement, see Definition \ref{defn:Ccomplexprime}.} on $X$ \cite[\S 8]{bode_six_2021}. The category $D_\mathcal{C}(X)$ is a derived enhancement of the abelian category of coadmissible $\wideparen{\mathcal{D}}_X(X)$-modules. It is easy (c.f. Lemma \ref{lem:Ccomplex}) to see that $D_\mathcal{C}(X)$ is contained in $\operatorname{RMod}_{\wideparen{\mathcal{D}}_X(X)} \operatorname{Fr}(X)$. Hence we obtain the following.
\begin{cor}\label{cor:EmbeddingofCcomplexes}
Assume that $X = \operatorname{dSp}(A)$ is classical and \'etale over a polydisk. Then there is a fully-faithful functor  $D_\mathcal{C}(X) \hookrightarrow \operatorname{QCoh}(X_{\mathrm{str}})$. 
\end{cor}
Recalling that the stack $X_{\mathrm{str}}$ is defined as the quotient of $X$ by the analytic germ of the diagonal, one could say that the above Corollaries give an interpretation of $\wideparen{\mathcal{D}}$-modules in terms of ``$\dagger$-infinitesimal parallel transport". 

In order to globalize from the affinoid case, it is natural to investigate descent for the various categories of $\wideparen{\mathcal{D}}$-modules thus far considered. Our results are as follows; we recall the definition of the Banach-completed differential operators $\mathcal{D}^n_X(X)$ from, for instance \cite[\S2]{bode_six_2021}; these are Noetherian Banach algebras and one has $\wideparen{\mathcal{D}}_X(X) = \operatorname{lim}_n\mathcal{D}^n_X(X)$. 
\begin{thm}[c.f. \S \ref{sec:DescentforDCap}]\label{thm:introdescentthm}
Assume that $X = \operatorname{dSp}(A)$ is classical and \'etale over a polydisk.
    \begin{enumerate}[(i)]
        \item The prestacks 
        \begin{equation}
            \begin{aligned}
            \operatorname{RMod}_{\wideparen{\mathcal{D}}_X(-)} D(\mathsf{CBorn}_K)  && \text{ and } && \operatorname{RMod}_{\mathcal{D}^n_X(-)} D(\mathsf{CBorn}_K) 
            \end{aligned}
        \end{equation}
        are sheaves in the weak topology on $X$.
        \item The prestack $\operatorname{RMod}^{\mathrm{b,fg}}_{\mathcal{D}^n_X(-)} D(\mathsf{CBorn}_K)$ is a sheaf on the site of $p^n$-accessible\footnote{By this  we mean $p^n\mathcal{T}_X$-accessible in the sense of \cite[\S4.5]{Dcap1}.} affinoid subdomains of $X$. Here the superscript $\mathrm{b,fg}$ denotes cohomologically bounded complexes with finitely-generated cohomology.
        \item There is an equivalence of $\infty$-categories
        \begin{equation}
            D_\mathcal{C}(X) \simeq \underset{n}{\operatorname{lim}} \operatorname{RMod}^{\mathrm{b,fg}}_{\mathcal{D}^n_X(X)} D(\mathsf{CBorn}_K), 
        \end{equation}
        where the left-hand-side again denotes Bode's $\mathcal{C}$-complexes \cite[\S8]{bode_six_2021}. 
        \item The prestack $D_\mathcal{C}(-)$ is a sheaf in the weak topology on $X$.  
    \end{enumerate}
\end{thm}
An important ingredient in the proof of Theorem \ref{thm:introdescentthm} is the noncommutative notion of descendability developed in Appendix \ref{subsec:NCdescent}, which may be of independent interest. A description as in (iii) above was conjectured in the introduction to \cite{bode_six_2021}, and I am grateful to Andreas Bode for helpful discussions about the proof of this. 

Now we investigate whether the embedding of $\mathcal{C}$-complexes from Corollary \ref{cor:EmbeddingofCcomplexes} is compatible with restrictions, so that it can be globalised. This turns out to be true (that is Theorem \ref{thm:CcomplexFFaffine}) but the proof is non-trivial for the following reason: the comparison between $\operatorname{Strat}(X)$ and modules over the monad $\mathcal{D}^\infty_X$ comes from forgetting via an \emph{upper-shriek} functor, but we want to use the \emph{upper-star} restriction functors for $\wideparen{\mathcal{D}}$-modules. In fact, we have to use the ``parametrized monadicity theorem" of Appendix \ref{subsec:BarrBeckFamilies} to manage some of the coherences. Using this compatibility with restrictions we obtain the following.
\begin{thm}[c.f. \S\ref{sec:global}]
Let $X$ be a smooth classical rigid space. For such $X$ we define the value of $D_\mathcal{C}(-)$ by Kan extension\footnote{If you like, you could call this the ``$\infty$-stackification" of Bode's $\mathcal{C}$-complexes.}, see Definition \ref{defn:CComplexstackification}. Then there is a fully-faithful functor 
\begin{equation}\label{eq:CComplexFFintro}
    D_\mathcal{C}(X) \hookrightarrow \operatorname{Strat}(X) = \operatorname{QCoh}(X_{\mathrm{str}}). 
\end{equation}
This induces a fully-faithful functor 
\begin{equation}
    \{ \text{coadmissible }\mathcal{D}_X\text{-modules}\} \hookrightarrow \operatorname{Strat}(X). 
\end{equation}
\end{thm}
In \S\ref{sec:essentialimage} we attempt but ultimately fail to find an intrinsic characterization of the essential image of the functor \eqref{eq:CComplexFFintro} in terms of various finiteness conditions arising in the context of six-functor formalisms. To summarise\footnote{We direct the reader there for relevant definitions.} the results of \S\ref{sec:essentialimage} : let $f: X_{\mathrm{str}} \to *$ be the canonical morphism. Then the essential image of \eqref{eq:CComplexFFintro} is not contained in the \emph{dualizable} nor the \emph{$f$-prim} objects. It seems most plausible that the essential image of  \eqref{eq:CComplexFFintro} is contained in the \emph{$f$-suave} objects, though we were unable to prove this. The latter was conjectured in \cite[\S5.2]{Camargo_deRham}.

\paragraph{Acknowledgements.} This work was done while the author was a DPhil student at the University of Oxford funded by an EPSRC scholarship, and formed part of the the author's DPhil thesis. I would like to thank my supervisor Konstantin Ardakov for his interest and encouragement and for helpful mathematical discussions. I would like to thank Andreas Bode, Jack Kelly, Kobi Kremnizer, Ken Lee, and Finn Wiersig for helpful  discussions.
\section{Reminders on quasi-abelian categories}\label{sec:quasiabelian}
In order for this article to make sense, it is necessary to recall certain features of the theory of analytic stacks following \cite{DAnG} and \cite{soor_six-functor_2024}. In these works homological algebra is done in the setting of quasi-abelian categories. We record the relevant constructions in this section, which may be skipped on a first reading. 

A \emph{quasi-abelian category} is an additive category $\mathscr{A}$ which has all kernels and cokernels, such that strict\footnote{Recall that a morphism $f$ is called strict if the natural morphism $\operatorname{coker} \operatorname{ker}f \xrightarrow[]{ }\operatorname{ker}\operatorname{coker}f$ is an isomorphism.} epimorphisms (resp. monomorphisms) are stable under pullbacks (resp. pushouts).

Let us recall the following notions of exactness in quasi-abelian categories. A functor between quasi-abelian categories is called \emph{left exact} (resp. \emph{strongly left exact}) if it preserves the kernels of strict morphisms (resp. all morphisms). Dually, a functor between quasi-abelian categories is called \emph{right exact} (resp. \emph{strongly right exact}) if it preserves the cokernels of strict morphisms (resp. all morphisms). A functor between quasi-abelian categories is called \emph{exact} (resp. \emph{strongly exact}) if it is left and right exact (resp. strongly left exact and strongly right exact). 

An object $P$ in a quasi-abelian category $\mathscr{A}$ is called \emph{projective} if the functor $\operatorname{Hom}(P,-): \mathscr{A} \to \mathsf{Ab}$ takes strict epimorphisms to surjections. We say that $\mathscr{A}$ has \emph{enough projectives}  if for each object $M \in \mathscr{A}$ there exists a projective object $P$ together with a strict epimorphism $P \twoheadrightarrow M$. 
\begin{defn}\cite[Definition 2.1.10]{schneiders_quasi-abelian_1999}
Let $\mathscr{A}$ be a quasi-abelian category. Assume that $\mathscr{A}$ admits (small) coproducts.
\begin{enumerate}[(i)]
    \item A (small) subcategory $\mathscr{P}$ of $\mathscr{A}$ is called \emph{generating} if for each object $M$ of $\mathscr{A}$ there exists a (small) collection $\{P_i\}_{i}$ of objects of $\mathscr{P}$ together with a strict epimorphism $\bigoplus_i P_i \twoheadrightarrow M$.
    \item An object $C \in \mathscr{A}$ is called \emph{small} if $\operatorname{Hom}(C,-): \mathscr{A} \to \mathsf{Ab}$ commutes with (small) coproducts.  
    \item The category $\mathscr{A}$ is called \emph{quasi-elementary} if it is cocomplete and has a (small) generating subcategory $\mathscr{P} \subseteq \mathscr{A}$ of small projective objects. 
\end{enumerate}
\end{defn}
\begin{defn}
Let $\mathcal{S}$ be a subcategory of a quasi-abelian category $\mathscr{A}$. 
\begin{enumerate}[(i)]
    \item Let $\mathcal{I}$ be a filtered category. An object $C \in \mathscr{A}$ is called \emph{$(\mathcal{I},\mathcal{S})$-tiny} if the functor $\operatorname{Hom}(C,-): \mathscr{A} \to \mathsf{Ab}$ commutes with colimits of diagrams in $\operatorname{Fun}_\mathcal{S}(\mathcal{I}, \mathscr{A})$. Here $\operatorname{Fun}(\mathcal{I}, \mathscr{A}) \subseteq \operatorname{Fun}_\mathcal{S}(\mathcal{I}, \mathscr{A})$ denotes the sub-class of those functors which factor through $\mathcal{S}$.
    \item The category $\mathscr{A}$ is called \emph{$(\mathcal{I},\mathcal{S})$-elementary} if $\mathscr{A}$ is generated by a subcategory $\mathscr{P} \subseteq \mathscr{A}$ consisting of $(\mathcal{I},\mathcal{S})$-tiny projective objects. 
        \item An object $C \in \mathscr{A}$ is called \emph{$\mathcal{S}$-tiny} if the functor $\operatorname{Hom}(C,-): \mathscr{A} \to \mathsf{Ab}$ commutes with colimits of diagrams in $\operatorname{Fun}_\mathcal{S}(\mathcal{I},\mathscr{A})$, for any filtered category $\mathcal{I}$. 
    \item The category $\mathscr{A}$ is called \emph{$\mathcal{S}$-elementary} if $\mathscr{A}$ is generated by a subcategory $\mathscr{P} \subseteq \mathscr{A}$ consisting of $\mathcal{S}$-tiny projective objects. 
\end{enumerate}
\end{defn}
In practice we often take $\mathcal{S}$ to be a \emph{wide subcategory} of $\mathscr{A}$, meaning a subcategory which contains all the objects of $\mathscr{A}$. For instance, we may take $(\mathcal{I},\mathcal{S}) = (\mathbf{N}, \mathrm{SplitMon})$, where $\mathrm{SplitMon}$ is the wide subcategory of $\mathscr{A}$ on split monomorphisms, or $\mathcal{S} = \mathrm{AdMon}$,  where $ \mathrm{AdMon}$ is the wide subcategory of $\mathscr{A}$ on strict monomorphisms, or $\mathcal{S} = \mathrm{all} = \mathscr{A}$. In each of these cases we say that $\mathscr{A}$ is \emph{$(\mathbf{N}, \mathrm{SplitMon})$-elementary}, \emph{$\mathrm{AdMon}$-elementary}, and \emph{elementary}, respectively. They are related in the following way: \emph{elementary} implies \emph{$\mathrm{AdMon}$-elementary} implies \emph{quasi-elementary} implies \emph{$(\mathbf{N}, \mathrm{SplitMon})$-elementary} implies \emph{enough projectives}.

We shall use $\operatorname{Ch}(\mathscr{A})$ to denote the category of cochain complexes valued in $\mathscr{A}$. We always use \emph{superscripts} for \emph{cohomological} indexing notation and \emph{subscripts} for \emph{homological} indexing notation. These conventions are of course related by $M^i = M_{-i}$ for $i \in \mathbf{Z}$.
\begin{thm}\cite[Theorem 4.59, Theorem 4.65]{kelly_homotopy_2021}\label{thm:kelly}
\begin{enumerate}[(i)]
    \item Let $\mathscr{A}$ be a quasi-abelian category with enough projectives. Then the \emph{projective model structure} on $\operatorname{Ch}^{\leqslant 0}(\mathscr{A})$ exists. The weak equivalences, fibrations and cofibrations may be described as follows:
    \begin{enumerate}
    \item[(W)] A morphism is a weak equivalence if it is a strict quasi-isomorphism, i.e., its cone is strictly exact. 
    \item[(F)] A morphism is a fibration if the its components are strict epimorphisms in positive degrees. 
    \item[(C)] A morphism is a cofibration if it is a degreewise strict monomorphism with degreewise projective cokernel.
    \end{enumerate}
    Further, this is a simplicial model structure. 
    \item Assume that $\mathscr{A}$ is a $(\mathbf{N},\mathrm{SplitMon})$-elementary quasi-abelian category. Then the \emph{projective model structure} on $\operatorname{Ch}(\mathscr{A})$ exists. The weak equivalences and fibrations may be described as follows:
    \begin{enumerate}
    \item[(W)] A morphism is a weak equivalence if it is a strict quasi-isomorphism, i.e., its cone is strictly exact. 
    \item[(F)] A morphism is a fibration if it is a degreewise strict epimorphism.
    \end{enumerate}
    Further, this is a stable and simplicial model structure. 
\end{enumerate}
\end{thm}
This permits us to make the following definition.
\begin{defn}
Let $\mathscr{A}$ be a $(\mathbf{N},\mathrm{SplitMon})$-elementary quasi-abelian category. The \emph{derived $\infty$-category} of $\mathscr{A}$ is defined to be the underlying $\infty$-category of the simplicial model category $\operatorname{Ch}(\mathscr{A})$. That is, it is the $\infty$-categorical localization
\begin{equation}
    D(\mathscr{A}) := N(\operatorname{Ch}(\mathscr{A}))[W^{-1}].
\end{equation}
This is a stable $\infty$-category.
\end{defn}
We recall \cite[\S 1.2.2]{schneiders_quasi-abelian_1999} that $D(\mathscr{A})$ is equipped with two canonical $t$-structures. Of these, it is conventional to prefer the \emph{left $t$-structure} which may be described as follows: $D^{\leqslant 0}(\mathscr{A})$ (resp. $D^{\geqslant 0}(\mathscr{A})$) is the sub-$\infty$-category on complexes which are strictly exact in positive (resp. negative) degrees. The heart of this $t$-structure is called the \emph{left heart} of $\mathscr{A}$ \cite[\S 1.2.3]{schneiders_quasi-abelian_1999} and denoted by $LH(\mathscr{A})$. We obtain cohomology functors 
\begin{equation}
    H^i :D(\mathscr{A}) \to LH(\mathscr{A})
\end{equation}
for each $i \in \mathbf{Z}$. The left heart $LH(\mathscr{A})$ may be explicitly described as the full subcategory of\footnote{Here $K(\mathscr{A})$ is the category with $\operatorname{Ob}(K(\mathscr{A})) = \operatorname{Ob}(\operatorname{Ch}(\mathscr{A}))$ and whose morphisms are chain-homotopy classes of morphisms in $\operatorname{Ch}(\mathscr{A})$.} $K(\mathscr{A})$ on two-term complexes
\begin{equation}
 0\to    M^{-1} \xrightarrow[]{d} M^0 \to 0
\end{equation}
with $d$ a monomorphism, \emph{localized} at the class of morphisms $f: [M^{-1} \to M^0] \to [N^{-1} \to N^0]$ such that
\begin{equation}
\begin{tikzcd}
	{M^{-1}} & {M^0} \\
	{N^{-1}} & {N^0}
	\arrow[from=1-1, to=1-2]
	\arrow["{f^{(-1)}}"', from=1-1, to=2-1]
	\arrow["{f^0}", from=1-2, to=2-2]
	\arrow[from=2-1, to=2-2]
\end{tikzcd}
\end{equation}
is a Cartesian and coCartesian square in $\mathscr{A}$. With respect to this description, the canonical functor $I: \mathscr{A} \to LH(\mathscr{A})$ is induced by the functor given by 
    \begin{equation}
        M \mapsto [0 \to M].
    \end{equation}
    This is fully-faithful and admits a left adjoint $C: LH(\mathscr{A}) \to \mathscr{A}$ which is induced by the functor
    \begin{equation}
     [M^{-1} \xrightarrow[]{d} M^0] \mapsto \operatorname{coker}d.
    \end{equation}
    on two-term complexes, so that $\mathscr{A}$ is a reflective subcategory of $LH(\mathscr{A})$. We can also describe the cohomology functor $H^i : D(\mathscr{A}) \to LH(\mathscr{A})$. It is given on objects by 
    \begin{equation}
       H^i: M^\bullet \mapsto [\operatorname{coker} \operatorname{ker} d^{i-1} \to \operatorname{ker} d^i].
    \end{equation}
    In particular, a complex $M^\bullet \in D(\mathscr{A})$ is strict\footnote{Meaning that its differentials are strict morphisms.} if and only if the cohomology objects $H^iM^\bullet \in LH(\mathscr{A})$ factor through the essential image of $\mathscr{A}$. We also see that a complex $M^\bullet$ is strictly exact if and only if $H^iM^\bullet = 0$ for all $i \in \mathbf{Z}$.
\subsection{Properties of the derived category}
Let us first recall the following basic definitions from \cite[Definition 5.5.7.1]{HigherAlgebra}. 
\begin{defn}
\begin{enumerate}[(i)]
    \item Let $\mathscr{C}$ be an $\infty$-category admitting filtered colimits. An object $C \in \mathscr{C}$ is called \emph{compact} if $\operatorname{Hom}(C,-)$ commutes with filtered colimits. We let $\mathscr{C}^\omega \subseteq \mathscr{C}$ denote the full subcategory spanned by the compact objects.
    \item An $\infty$-category $\mathscr{C}$ is called \emph{compactly generated} if there exists a small $\infty$-category $\mathscr{C}_0$ admitting finite colimits, and an equivalence $\operatorname{Ind}(\mathscr{C}_0) \simeq \mathscr{C}$.   
\end{enumerate}
\end{defn}
And here are some well-known properties of compactly generated $\infty$-categories. For proofs of the assertions (ii) and (iii) below, we direct the reader to \cite[\S2.2]{soor_six-functor_2024}, although it is certain that these facts are proved elsewhere. 
\begin{prop}\label{prop:compactgenprops} Let $\mathscr{C}$ be a compactly generated $\infty$-category. 
\begin{enumerate}[(i)]
    \item \cite[\S5.5.7]{HigherToposTheory} The full subcategory $\mathscr{C}^\omega \subseteq \mathscr{C}$ is essentially small, admits finite colimits, and the inclusion induces an equivalence $\operatorname{Ind}(\mathscr{C}^\omega) \xrightarrow[]{\sim} \mathscr{C}$. 
    \item For any regular cardinal $\kappa$, one has that $\kappa$-filtered colimits commute with $\kappa$-small limits in $\mathscr{C}$.
    \item Let $F: \mathscr{C} \leftrightarrows \mathscr{D}:G$ be an adjunction between compactly generated $\infty$-categories. Then the right adjoint $G$ preserves filtered colimits if and only if the left adjoint $F$ preserves compact objects. 
\end{enumerate}
\end{prop}
Given a small $\infty$-category $\mathscr{D}_0$ admitting finite coproducts, we denote by 
\begin{equation}
 \operatorname{sInd}(\mathscr{D}_0) = \operatorname{Fun}^\Pi(\mathscr{D}_0^\mathsf{op}, \infty\mathsf{Grpd})  
\end{equation} 
its free sifted cocompletion \cite[\S 5.5.8]{HigherToposTheory}. For proofs of the below statements we direct the reader to \cite[\S2.2]{soor_six-functor_2024}.
\begin{prop}\label{prop:Univprops1}
Let $\mathscr{A}$ be a quasi-elementary quasi-abelian category. Fix a (small) generating set of small projective objects $\mathscr{P} \subseteq \mathscr{A}$. Assume that $\mathscr{P}$ is closed under finite products in $\mathscr{A}$.
\begin{enumerate}[(i)]
    \item There is an equivalence of $\infty$-categories $D^{\leqslant 0}(\mathscr{A}) \simeq \operatorname{sInd}(N(\mathscr{P}))$. 
    \item Let $\mathscr{D}$ be any cocomplete $\infty$-category. Precomposition with $N(\mathscr{P}) \to D^{\leqslant 0}(\mathscr{A})$ induces an equivalence of $\infty$-categories $\operatorname{Fun}^L(D^{\leqslant 0}(\mathscr{A}), \mathscr{D}) \xrightarrow[]{\sim} \operatorname{Fun}^\amalg(N(\mathscr{P}), \mathscr{D})$. 
    \item The $\infty$-category $D^{\leqslant 0}(\mathscr{A})$ is compactly generated. The compact objects in $D^{\leqslant 0}(\mathscr{A})$ admit the following description: Let $j :N(\mathscr{P}) \hookrightarrow D(\mathscr{A})$ be the inclusion. Then an object $C$ of $D^{\leqslant 0}(\mathscr{A})$ is compact if and only if there exists a finite diagram $p: K \to \mathscr{P}$ such that $C$ is a retract of $\operatorname{colim} j \circ p$. 
    \item The $t$-structure on $D(\mathscr{A})$ is both left and right complete. In particular (by right completeness) there is a $t$-exact equivalence of $\infty$-categories $D(\mathscr{A}) \simeq \operatorname{Sp}(D^{\leqslant 0}(\mathscr{A}))$, where the latter denotes the $\infty$-category of spectrum objects.
    \item Let $\mathscr{D}$ be any stable presentable $\infty$-category. Precomposition with $N(\mathscr{P}) \to D(\mathscr{A})$ induces an equivalence of $\infty$-categories $\operatorname{Fun}^L(D(\mathscr{A}), \mathscr{D}) \xrightarrow[]{\sim} \operatorname{Fun}^\amalg(N(\mathscr{P}), \mathscr{D})$.
    \item Let $\mathscr{D}$ be any stable presentable $\infty$-category with $t$-structure $(\mathscr{D}^{\leqslant 0}, \mathscr{D}^{\geqslant 0})$. Precomposition with $N(\mathscr{P}) \to D(\mathscr{A})$ induces an equivalence of $\infty$-categories $\operatorname{Fun}^\prime(D(\mathscr{A}), \mathscr{D}) \xrightarrow[]{\sim} \operatorname{Fun}^\amalg(N(\mathscr{P}), \mathscr{D})$. Here $\operatorname{Fun}^\prime$ denotes the colimit-preserving and right $t$-exact functors. 
    \item The $\infty$-category $D(\mathscr{A})$ is compactly generated. The compact objects in $D(\mathscr{A})$ admit the following description: Let $j: N(\mathscr{P}) \to D(\mathscr{A})$ be the inclusion. An object $C$ of $D(\mathscr{A})$ is compact if and only if there exists $n \geqslant 0$ and a finite diagram $p: K \to \mathscr{P}$ such that $C$ is a retract of $\operatorname{colim} \Omega^n j \circ p$. Here $\Omega = [-1]$ denotes\footnote{Recall that our default is \emph{cohomological} indexing notation.} the loops functor on $D(\mathscr{A})$. 
\end{enumerate}
\end{prop}
Because the compact objects are stable under finite colimits and retracts, one may also describe $D^{\leqslant 0}(\mathscr{A})^\omega$ as the full subcategory of $D^{\leqslant 0}(\mathscr{A})$ generated under cones, suspensions and retracts by the essential image of $N(\mathscr{P})$. Similarly, we may describe $D(\mathscr{A})^\omega$ as the full subcategory of $D(\mathscr{A})$ generated under cones, shifts and retracts by the essential image of $N(\mathscr{P})$. 

\subsection{Monoidal structure} Now we consider the situation when $\mathscr{A}$ is endowed with a monoidal structure. In the remainder of this section we make the following \emph{assumptions}:
\begin{itemize}
    \item[$\star$] $\mathscr{A}$ is endowed with a closed monoidal structure $(\mathscr{A}, \otimes, \underline{\operatorname{Hom}})$;
    \item[$\star$] The monoidal structure on $\mathscr{A}$ restricts to $\mathscr{P}$;
    \item[$\star$] Every object of $\mathscr{P}$ is flat\footnote{An object $M \in \mathscr{A}$ is called \emph{flat} if $M \otimes -$ is an exact functor in the sense of quasi-abelian categories.}.
\end{itemize}
Kelly has proved the following:
\begin{thm}\cite[Theorem 4.69]{kelly_homotopy_2021}. Under the above assumptions. The projective model structures on $\operatorname{Ch}(\mathscr{A})$ and $\operatorname{Ch}^{\leqslant 0}(\mathscr{A})$ are monoidal, i.e., $(\operatorname{Ch}(\mathscr{A}), \otimes)$ and $(\operatorname{Ch}^{\leqslant 0}(\mathscr{A}), \otimes)$ are monoidal model categories.
\end{thm}
By Proposition \ref{prop:Univprops1} and the dictionary between model categories and $\infty$-categories, this implies that $(D^{\leqslant 0}(\mathscr{A}), \otimes^\mathbf{L})$ and $(D(\mathscr{A}), \otimes^\mathbf{L})$ are both presentably symmetric monoidal. Building on Proposition \ref{prop:Univprops1} above, one can also formulate universal properties for $D^{\leqslant 0}(\mathscr{A})$ and $D(\mathscr{A})$ as symmetric monoidal $\infty$-categories. For proofs of the below statements we direct the reader to \cite[\S2.3]{soor_six-functor_2024}.
\begin{prop} Under the above assumptions:
\begin{enumerate}[(i)]
    \item The equivalence of Proposition \ref{prop:Univprops1}(i) upgrades to an equivalence of presentably symmetric monoidal $\infty$-categories:
    \begin{equation}
        (\operatorname{sInd}(N(\mathscr{P})), \otimes_{\mathrm{Day}}) \xrightarrow[]{\sim} (D^{\leqslant 0}(\mathscr{A}), \otimes^\mathbf{L}),
    \end{equation}
    where $\otimes_{\mathrm{Day}}$ denotes the Day convolution monoidal structure on the sifted cocompletion;
    \item Let $\mathscr{D}$ be any symmetric monoidal $\infty$-category such that the tensor product $\mathscr{D} \times \mathscr{D} \to \mathscr{D}$ commutes with colimits separately in each variable. Then there is an equivalence of $\infty$-categories 
    \begin{equation}
        \operatorname{Fun}^{L, \otimes}(D^{\leqslant 0}(\mathscr{A}), \mathscr{D}) \simeq \operatorname{Fun}^{\amalg,\otimes}(N(\mathscr{P}), \mathscr{D}).
    \end{equation}
    \item There is an equivalence $D^{\leqslant 0}(\mathscr{A}) \otimes \mathsf{Sp} \simeq D(\mathscr{A})$ of commutative algebra objects in $\mathsf{Pr}^L$. In particular, for any stable presentably symmetric monoidal $\infty$-category $\mathscr{D}$ there is an equivalence of $\infty$-categories
    \begin{equation}
         \operatorname{Fun}^{L, \otimes}(D(\mathscr{A}), \mathscr{D}) \simeq \operatorname{Fun}^{\amalg,\otimes}(N(\mathscr{P}), \mathscr{D}).
    \end{equation}
    \item Let $\mathscr{D}$ be a stable presentably symmetric monoidal $\infty$-category equipped with $t$-structure $(\mathscr{D}^{\leqslant 0}, \mathscr{D}^{\geqslant 0})$ such that the monoidal structure restricts to $\mathscr{D}^{\leqslant 0}$. Then there is an equivalence of $\infty$-categories
    \begin{equation}
        \operatorname{Fun}^{\prime,\otimes}(D(\mathscr{A}), \mathscr{D}) \simeq \operatorname{Fun}^{\amalg, \otimes}(N(\mathscr{P}), \mathscr{D}^{\leqslant0}), 
    \end{equation}
    where $\operatorname{Fun}^{\prime,\otimes}$ denotes those symmetric monoidal functors which are colimit-preserving and right $t$-exact.
\end{enumerate}
\end{prop}
Since $D(\mathscr{A})$ is presentably symmetric monoidal, it is closed symmetric monoidal (by the adjoint functor theorem). We write $R\underline{\operatorname{Hom}}(-,-)$ for the internal Hom-bifunctor on $D(\mathscr{A})$. The following will occasionally be useful:
\begin{prop}\cite[Lemma 2.50]{soor_six-functor_2024}\label{prop:RHomfilteredColimits}
With assumptions as above.
\begin{enumerate}[(i)]
    \item If $C^\bullet, C^{\prime,\bullet} \in D(\mathscr{A})^\omega$, then $ C^\bullet \otimes^\mathbf{L} C^{\prime,\bullet} \in D(\mathscr{A})^\omega$;
    \item If $C^\bullet \in D(\mathscr{A})^\omega$ then $R \underline{\operatorname{Hom}}(C^\bullet, -)$ commutes with filtered colimits. 
\end{enumerate}
\end{prop}

\section{Reminders on complete-bornological spaces}
As in \cite{DAnG} and \cite{soor_six-functor_2024}, our underlying functional analysis is based on the theory of complete bornological spaces. In this section we attempt to summarize the necessary features of this theory. 

Let $(K, |\cdot|)$ be a complete field extension of $\mathbf{Q}_p$, with ring of integers $o = \{ x \in K : |x| = 1\}$. We assume that $|p|= p^{-1}$ so that $|\cdot|$ restricts to the $p$-adic norm on $\mathbf{Q}_p$. By a \emph{$K$-Banach space} we mean a complete normed $K$-vector space. In this article all norms are assumed to satisfy the ultrametric inequality. We recall the following definition:
\begin{defn}
Let $V$ be a $K$-vector space. A \emph{bornology} on $V$ is a collection of $\mathscr{B}$ of \emph{bounded subsets} of $V$ satifying the following properties:
\begin{enumerate}[(i)]
    \item If $B \in \mathscr{B}$ and $B^\prime \subseteq B$ then $B^\prime \in \mathscr{B}$;
    \item If $v \in V$ then $\{v\} \in \mathscr{B}$;
    \item $\mathscr{B}$ is closed under finite unions;
    \item If $B \in \mathscr{B}$ and $\lambda \in K$ then $\lambda B \in \mathscr{B}$;
    \item If $B \in \mathscr{B}$ then the $o$-submodule $o \cdot B \in \mathscr{B}$.
\end{enumerate}
The pair $V = (V,\mathscr{B})$ is called a \emph{(convex) bornological $K$-vector space}. A morphism $\varphi: V \to W$ of $K$-vector spaces is called \emph{bounded} if $\varphi(B) \subseteq W$ is bounded for every bounded subset $B \subseteq V$. In this way we obtain the category $\mathsf{Born}_K$ of bornological $K$-vector spaces.
\end{defn}
A seminormed $K$-vector space acquires a bornology in an obvious way. We do not regard the seminorm as part of the datum of a seminormed $K$-vector space, only the bornology. This is captured in the following definition.
\begin{defn}
The category of \emph{seminormed $K$-vector spaces}, (resp. \emph{normed $K$-vector spaces}, resp. \emph{$K$-Banach spaces}), is defined to be the full subcategory of $\mathsf{Born}_K$ on objects $V$ whose bornology is induced by a seminorm (resp. a norm, resp. a norm making $V$ into a Banach space). We denote these categories by $\mathsf{SNrm}_K, \mathsf{Nrm}_K$, and $\mathsf{Ban}_K$, respectively.
\end{defn}
\begin{defn}
Let $V \in \mathsf{Born}_K$. 
\begin{enumerate}[(i)]
    \item Given a bounded $o$-submodule $B \subseteq V$ we define $V_B := \operatorname{span}_KB \subseteq V$ equipped with the bornology defined by the gauge seminorm $\|x\|_B := \inf \{|\lambda| : x \in \lambda B \}$. 
    \item $V$ is called \emph{complete} if for every $B \in \mathscr{B}$ there exists a bounded $o$-submodule $B^\prime \supseteq B$ such that $V_{B^\prime}$ is a $K$-Banach space. We define $\mathsf{CBorn}_K \subseteq \mathsf{Born}_K$ to be the full subcategory of \emph{complete bornological $K$-vector spaces}. 
\end{enumerate}
\end{defn}
We recall that $\mathsf{Born}_K$ is closed symmetric monoidal \cite{BambozziDaggerBanach}. Given $V, W \in \mathsf{Born}_K$ the bornology on the tensor product $V \otimes W$ is generated by subsets of the form $B \otimes_o B^\prime$ for bounded $o$-submodules $B \subseteq V$, $B^\prime \subseteq V$. The internal Hom, denoted $\underline{\operatorname{Hom}}_K(V,W)$, is the space of bounded linear maps equipped with the equibounded bornology. 

The inclusion $\mathsf{CBorn}_K \subseteq \mathsf{Born}_K$ admits a left adjoint $\widehat{(\cdot)}: \mathsf{Born}_K  \to \mathsf{CBorn}_K$ (completion, c.f. \cite{HogbeNlendThesis, ProsmansHomological}). The internal Hom $\underline{\operatorname{Hom}}_K(V,W)$ between two bornological $K$-vector spaces $V, W$ is complete whenever $W$ is. From this it follows formally that 
\begin{equation}
  (\mathsf{CBorn}_K, \widehat{\otimes}_K, \underline{\operatorname{Hom}}_K)  
\end{equation}
is closed symmetric monoidal, where the \emph{completed tensor product} $\widehat{\otimes}_K$ is the completion of the bornological tensor product. 
There is an adjunction 
\begin{equation}
    \operatorname{diss}: \mathsf{CBorn}_K \leftrightarrows \mathsf{Ind}(\mathsf{Ban}_K) : \operatorname{colim}
\end{equation}
in which the right adjoint $\operatorname{diss}: V \mapsto \indlim \widehat{V}_B$ is fully faithful. The essential image is given by the \emph{essentially monomorphic $\mathrm{Ind}$-objects}, i.e., those $\mathrm{Ind}$-objects which are equivalent to $\mathrm{Ind}$-systems of monomorphisms. Consequently there is an equivalence of categories 
\begin{equation}\label{eq:CBornessentiallymono}
\operatorname{diss}: \mathsf{CBorn}_K \simeq \mathsf{Ind}^m(\mathsf{Ban}_K) : \operatorname{colim}.
\end{equation}
It is possible to give a more efficient ``colimit presentation" than this, which is still functorial. This is the purpose of the following Definition. 
\begin{defn}
Let $V \in \mathsf{Born}_K$. 
\begin{enumerate}[(i)]
    \item We define a transitive, reflexive relation $\lesssim$ on the bounded $o$-submodules of $V$ by $B \lesssim B^\prime$ if there exists a bounded $o$-submodule $S \subseteq K$ such that $B \subseteq S\cdot B^\prime$. If $B \lesssim B^\prime$ we say that \emph{$B^\prime$ absorbs $B$}.
    \item We define an equivalence relation $\asymp$ on the bounded $o$-submodules of $V$ by $B \asymp B^\prime$ if $B \lesssim B^\prime$ and $B^\prime \lesssim B$. 
    \item We let $(\mathfrak{S}(V), \lesssim )$ denote the collection of $\asymp$-equivalence classes viewed as a poset with the partial order induced by $\lesssim$. 
\end{enumerate} 
\end{defn}
\begin{lem}
If $V \in \mathsf{CBorn}_K$ then $\indlim_{[B] \in \mathfrak{S}(V)} \widehat{V}_B \in \mathsf{Ind}(\mathsf{Ban}_K)$ is essentially monomorphic and the natural morphism
    \begin{equation}
\underset{[B] \in \mathfrak{S}(V)}{\operatorname{colim}} \widehat{V}_B \xrightarrow{\sim} V
    \end{equation}
    is an isomorphism in $\mathsf{CBorn}_K$.
\end{lem}
\begin{proof}
Given the equivalence \eqref{eq:CBornessentiallymono}, the only thing to note is that if $B \asymp B^\prime$ then there is an equality $V_B = V_{B^\prime}$.
\end{proof}
\begin{example}
An object $V \in \mathsf{CBorn}_K$ belongs to the full subcategory $\mathsf{Ban}_K$ if and only if $\mathfrak{S}(V)$ has a terminal object. 
\end{example}
\begin{defn}\label{defn:metrizable}\cite{HogbeNlendThesis}
Let $V \in \mathsf{Born}_K$.
\begin{enumerate}[(i)]
    \item We say that $V$ is of \emph{countable type} if it has a countable base for its bornology. Equivalently, the poset $\mathfrak{S}(V)$ has a countable cofinal subset.
    \item We say that $V$ is \emph{metrizable} if the poset $\mathfrak{S}(V)$ is $\aleph_1$-filtered. 
\end{enumerate}
\end{defn}
The relation to locally-convex vector spaces is the following.
\begin{prop}\cite{HogbeNlendThesis}
There is an adjunction 
\begin{equation}
   (-)^t: \mathsf{Born}_K \leftrightarrows \mathsf{LCVS}_K : (-)^b 
\end{equation}
which is given as follows:
\begin{itemize}
    \item[$\star$] The functor $W \mapsto W^b$ endows a locally-convex $K$-vector space $W$ with its \emph{von Neumann bornology}: One has $W^b = W$ as $K$-vector spaces, and a subset $B \subseteq W^b$ is bounded if for every lattice $L \subseteq V$ there exists $\lambda \in K$ such that $B \subseteq \lambda L$.
    \item[$\star$] \emph{Dually}, the functor $V \mapsto V^t$ endows $V$ with the \emph{topology of bornivorous subsets}: One has $V^t = V$ as $K$-vector spaces, and an $o$-submodule $L \subseteq V^t$ is an open neighbourhood of $0$ if for every bounded subset $B \subseteq V$ there exists $\lambda \in K$ such that $B \subseteq \lambda L$.  
\end{itemize}
\end{prop}
\begin{lem}\label{lem:frechetAleph1}\cite{HogbeNlendThesis}
\begin{enumerate}[(i)]
    \item If $W \in \mathsf{LCVS}_K$ is metrizable, then the counit morphism $W^{bt} \to W$ is an isomorphism. In particular we obtain a fully-faithful functor from the full subcategory of metrizable objects to $\mathsf{Born}_K$.
    \item If $W \in \mathsf{LCVS}_K$ is complete and metrizable, then $W^b$ is complete as a bornological space and metrizable in the sense of Definition \ref{defn:metrizable}. In particular we obtain a fully-faithful functor $(-)^b: \mathsf{Fr}_K \hookrightarrow \mathsf{CBorn}_K$ from the category $\mathsf{Fr}_K$ of $K$-Fr\'echet spaces. 
\end{enumerate} 
\end{lem}
\begin{proof}
We only prove the metrizability part in (ii), the rest of the assertions being well-known. This appears already in \cite[p.200]{HogbeNlendThesis} where it is attributed to Mackey. Let $\{B_n\}_n$ be a countable collection of bounded $o$-submodules of $V$ and let $\{L_n\}_n$ be a (decreasing) fundamental system of lattices defining the Fr\'echet topology. By definition of the von Neumann bornology, for each $n$ there exists $\lambda_n \in K^\times$ such that $\lambda_n B_n \subseteq L_n$. Set $B^\prime:= \sum_{n}\lambda_n B_n$. Obviously, $B_n \lesssim B^\prime$ for each $n$. Further, for each $N \geqslant 0$ one has $B^\prime \subseteq L_N + \sum_{n=1}^N\lambda_n B_n$, which implies that $B^\prime$ is von Neumann bounded. 
\end{proof}
\begin{defn}
An object $V \in \mathsf{CBorn}_K$ is called \emph{conuclear} if for all $K$-Banach spaces $W$, the canonical morphism
\begin{equation}
\underline{\operatorname{Hom}}_K(V,K)\widehat{\otimes}_K W \to \underline{\operatorname{Hom}}_K(V,W)
\end{equation}
is an isomorphism.
\end{defn}
\begin{example}
Let $K$ be a non-trivially valued non-Archimedean field. Let $\varpi \in K$ with $0 < |\varpi| < 1$. Then 
\begin{equation*}
K \langle x/\varpi^\infty \rangle:=  \underset{n}{\operatorname{colim}} K\langle x/ \varpi^n \rangle = \Big\{\sum_{n =0}^\infty a_n x^n : |a_n||\varpi|^{-nk} \xrightarrow[]{n \to \infty} 0 \text{ for some } k \in \mathbf{Z} \Big\}.
\end{equation*}
is a conuclear object of $\mathsf{CBorn}_K$. 
\end{example}
Let us now introduce an important class of $K$-Banach spaces.
\begin{defn}
Let $S$ be a (small) set and $V \in \mathsf{Ban}_K$. The space of \emph{$V$-valued zero sequences} is
\begin{equation}
    c_0(S,V) := \{ \varphi: S \to V: \forall \varepsilon > 0 , \exists \text{ at most finitely many }s \in S: \|\varphi(s)\| > \varepsilon\}, 
\end{equation}
viewed as $K$-Banach space via the norm $\|\varphi\| := \sup_{s \in S} \|\varphi (s)\|$. When $V = K$ we write $c_0(S) := c_0(S,K)$. 
\end{defn}
Here are some basic properties of these spaces.
\begin{prop}
\begin{enumerate}[(i)]
    \item Let $S$ be a (small) set and $V \in \mathsf{CBorn}_K$. There is a natural (in $S$ and $V$) isomorphism
\begin{equation}
    \operatorname{Hom}_K(c_0(S),V) \cong \{\text{functions }f:S \to V: f(S) \subseteq V \text{ is bounded}\}. 
\end{equation}
    \item 
Let $S, S^\prime$ be (small) sets. There are canonical isomorphisms
\begin{equation}
    c_0(S) \widehat{\otimes}_K c_0(S^\prime) \cong  c_0(S \times S^\prime) \cong c_0(S, c_0(S^\prime)). 
    \end{equation}
\end{enumerate}
\end{prop}
Now we can state the important categorical properties of $\mathsf{CBorn}_K$. The main takeaway is that the formalism of \S\ref{sec:quasiabelian} applies with $\mathscr{A}= \mathsf{CBorn}_K$. 
\begin{prop}\label{prop:CbornKproperties}
\begin{enumerate}[(i)]
    \item $\mathsf{CBorn}_K$ is a quasi-abelian category. A morphism $\varphi: V \to W$ is strict if and only if $\operatorname{im} \varphi \subseteq W$ is bornologically closed and the bornology on $\operatorname{im}\varphi$ coincides with the quotient bornology on $V/\operatorname{ker} \varphi$. 
    \item $\mathsf{CBorn}_K$ is $\operatorname{AdMon}$-elementary (in particular it is complete and cocomplete). A generating family of $\operatorname{AdMon}$-tiny projective objects is given by $\mathscr{P} = \{c_0(S)\}_S$ for $S$ ranging over (small) sets. 
    \item In $\mathsf{CBorn}_K$, colimits of (essentially) monomorphic filtered systems are strongly exact. 
\end{enumerate}
\end{prop}
\section{Proof of the main result}
In \cite{soor_six-functor_2024}, attached a to derived rigid space $X$ we defined a stack $X_{\mathrm{str}}$ called the \emph{stratifying stack of $X$}. Let us recall here all the features of this construction starting from the definition of derived rigid spaces.

Recall that we work relative to the $\infty$-category $D(\mathsf{CBorn}_K)$, that is the derived category of complete-bornological spaces. Our category of ``analytic" algebras is simply\footnote{In derived geometry it appears to be conventional to use homological indexing convention for algebra objects.} $\mathsf{dAlg} := \mathsf{CAlg}(D_{\geqslant 0}(\mathsf{CBorn}_K))$. Our category of affines is nothing but $\mathsf{dAff}:= \mathsf{dAlg}^\mathsf{op}$. For $A \in \mathsf{dAlg}$ we denote the corresponding object of $\mathsf{dAff}$ by the formal expression $\operatorname{dSp}(A)$. For such, the category of quasi-coherent sheaves is defined to be 
\begin{equation}
    \operatorname{QCoh}(\operatorname{dSp}(A)) := \operatorname{Mod}_AD(\mathsf{CBorn}_K),
\end{equation}
which, via the pullback functors, yields a functor $\mathsf{dAff}^\mathsf{op} \to \mathsf{Cat}_\infty$. We define the $\infty$-category of \emph{prestacks} to be $\mathsf{PStk} := \operatorname{Fun}(\mathsf{dAff}^\mathsf{op}, \infty\mathsf{Grpd})$. By Kan extension along the Yoneda embedding, $\operatorname{QCoh}$ extends to a functor $\mathsf{PStk}^\mathsf{op} \to \mathsf{Cat}_\infty$. 

The category $\mathsf{dAfndAlg}$ of \emph{derived affinoid algebras} was defined as a certain full subcategory of $\mathsf{dAlg}$. To be precise, $\mathsf{dAfndAlg}$ is the full subcategory on objects $A$ such that $\pi_0A$ is an affinoid algebra in the classical sense (a quotient of a Tate algebra in finitely many variables), and for $m \geqslant 1$, $\pi_mA$ is finitely-generated as a $\pi_0A$-module. Then we define $\mathsf{dAfnd} := \mathsf{dAfndAlg}^\mathsf{op}$. We defined a certain Grothendieck topology on $\mathsf{dAfnd}$ (the \emph{weak topology}). The covering sieves in this Grothendieck topology are generated by finite families of \emph{derived rational localizations}. The functor $\operatorname{QCoh}: \mathsf{dAfnd}^\mathsf{op} \to \mathsf{Cat}_\infty$ is a sheaf in the weak topology. The category $\mathsf{dRig}$ of \emph{derived rigid spaces} was defined as a certain full subcategory of $\operatorname{Shv}_{\mathrm{weak}}(\mathsf{dAfnd}, \infty\mathsf{Grpd})$. The $\infty$-category $\mathsf{dRig}$ is equipped with a Grothendieck topology, the \emph{strong Grothendieck topology.} The subcategory $\mathsf{dAfnd}$ equipped with the the weak topology, is a basis for  $\mathsf{dRig}$, equipped with the strong topology. The value of $\operatorname{QCoh}$ on $\mathsf{dRig}$ can be defined by Kan extension along $\mathsf{dAfnd}^\mathsf{op} \to \mathsf{dRig}^\mathsf{op}$, and it follows formally that $\operatorname{QCoh}: \mathsf{dRig}^\mathsf{op} \to \mathsf{Cat}_\infty$ is a sheaf in the strong topology.

Let us briefly comment on a notational convention. We write $\widehat{\otimes}_X$ for the monoidal structure on $\operatorname{QCoh}(X)$. When $X = \operatorname{dSp}(A)$ is a derived affinoid we write $\widehat{\otimes}_X = \widehat{\otimes}^\mathbf{L}_A$ for the monoidal structure on $\operatorname{QCoh}(\operatorname{dSp}(A)) = \operatorname{Mod}_AD(\mathsf{CBorn}_K)$. That is, we suppress the $\mathbf{L}$ when the subscript is a ``space".

The category of classical rigid spaces embeds fully-faithfully into $\mathsf{dRig}$. The inclusion admits a right adjoint $X \mapsto X_0$, the functor of \emph{classical truncation}, which extends $\operatorname{dSp}(A) \mapsto \operatorname{dSp}(\pi_0A)$ on derived affinoids. Every $X \in \mathsf{dRig}$ has an underlying topological space $|X|$ which is invariant under classical truncation: $|X_0| \xrightarrow[]{\sim} |X|$. In the case when $X = \operatorname{dSp}(A)$ is a derived affinoid, then $|X| \cong |\operatorname{Spa}(\pi_0A, (\pi_0A)^\circ)|$ is the Huber spectrum of the classical truncation. In general $|X|$ is the underlying topological space of the analytic adic space associated to $X_0$. 

Now let us recall the definition of the stratifying stack $X_{\mathrm{str}}$, first in the case when $X = \operatorname{dSp}(A)$ is a derived affinoid. We first define the algebra (resp. space) of \emph{germs along the diagonal}:
\begin{equation}
\begin{aligned}
    (A \widehat{\otimes}^\mathbf{L}_KA)^\dagger_\Delta  := \underset{ U \supseteq \Delta X}{\operatorname{colim}} (A \widehat{\otimes}^\mathbf{L}_KA)_U && \text{ and } && (X \subseteq X \times X)^\dagger:=\operatorname{dSp}((A \widehat{\otimes}^\mathbf{L}_KA)^\dagger_\Delta).
\end{aligned}
\end{equation}
Here the colimit is taken\footnote{This is important because the colimit does not exist in $\mathsf{dAfndAlg}$.} in $\mathsf{dAlg}$, and ranges over all affinoid subdomains $U \subseteq X \times X$ such that $|U|$ contains the diagonal $|\Delta X| \subseteq |X \times X|$. In a similar way one may define the germ $(X \subseteq X^{n+1})^\dagger$ along the diagonal, for any $n \geqslant 0$. Letting $n$ vary, these can be arranged into a simplicial object $(X \subseteq X^{\bullet +1})^\dagger$. The \emph{stratifying stack} is defined as 
\begin{equation}
    X_{\mathrm{str}} := \underset{[n] \in \Delta^\mathsf{op}}{\operatorname{colim}} (X \subseteq X^{\bullet + 1})^\dagger,
\end{equation}
where the colimit is taken in $\mathsf{PStk}$. For general $X \in \mathsf{dRig}$ the value of $X_{\mathrm{str}}$ is defined by Kan extension of the functor $(-)_\mathrm{str}$, so that $X_{\mathrm{str}} = \operatorname{colim}_{(Y \to X )\in\mathsf{dAfnd}_{/X}} Y_{\mathrm{str}}$. 

We can consider the canonical morphism $p:X \to X_{\mathrm{str}}$ and the induced adjunctions $p^* \dashv p_*$ and $p_! \dashv p^!$ on quasi-coherent sheaves. As it turns out, there is a canonical equivalence $p_! \xrightarrow[]{\sim} p_*$. We define the \emph{monad of infinite-order differential operators}, resp. the \emph{comonad of analytic jets}, to be 
\begin{equation}
    \begin{aligned}
        \mathcal{D}^\infty_X := p^!p_! && \text{and} && \mathcal{J}^\infty_X := p^*p_*,
    \end{aligned}
\end{equation}
acting on $\operatorname{QCoh}(X)$. By the Giraud(--Rezk--Lurie) axioms \cite[Proposition 6.1.0.1]{HigherToposTheory}, the following square is Cartesian\footnote{More specifically, we are using the axiom which says that all groupoid objects in an $\infty$-topos are effective.}:
\begin{equation}
\begin{tikzcd}
	{( X \subseteq X \times X)^\dagger} & X \\
	X & {X_{\mathrm{str}}}
	\arrow["{\widetilde{\pi}_1}", from=1-1, to=1-2]
	\arrow["{\widetilde{\pi}_2}"', from=1-1, to=2-1]
	\arrow["\lrcorner"{anchor=center, pos=0.125}, draw=none, from=1-1, to=2-2]
	\arrow["p", from=1-2, to=2-2]
	\arrow["p"', from=2-1, to=2-2]
\end{tikzcd}
\end{equation}
here $\widetilde{\pi}_1, \widetilde{\pi}_2 : (X \subseteq X \times X)^\dagger \to X$ are the two projections. Hence by base change\footnote{The base-change isomorphism is furnished by the six-functor formalism of \cite[Theorem 4.67]{soor_six-functor_2024}.} the underlying functor of $\mathcal{D}^\infty_X$ can be described as $\mathcal{D}^\infty_X \simeq \widetilde{\pi}_{2,*} \widetilde{\pi}^!_1$. When $X = \operatorname{dSp}(A)$
 is a derived affinoid this is 
 \begin{equation}\label{eq:underlyingendofunctor}
    \mathcal{D}^\infty_X \simeq \tilde{\pi}_{2,*}  \tilde{\pi}_1^! \simeq R\underline{\operatorname{Hom}}_A((A \widehat{\otimes}_KA)^\dagger_\Delta,-),
\end{equation}
where $R\underline{\operatorname{Hom}}_A$ is taken with respect to the $A$-module structure on the first factor. It is \emph{critically important} to note that $(A \widehat{\otimes}_KA)^\dagger_\Delta$ is an $A$-$A$ bimodule.  We adopt the following convention:
\begin{itemize}
    \item[$\star$] The left $A$-module structure on $\mathcal{D}^\infty_XM^\bullet$ (for $M^\bullet \in \operatorname{QCoh}(X)$) comes from the $A$-module structure on the \emph{first factor} of $(A\widehat{\otimes}_KA)^\dagger_\Delta$,
    \item[$\star$] The right $A$-module structure on $\mathcal{D}^\infty_XM^\bullet$ comes from the $A$-module structure on the \emph{second factor} of $(A\widehat{\otimes}_KA)^\dagger_\Delta$. 
\end{itemize}
As an endofunctor of $\operatorname{QCoh}(X)$, $\mathcal{D}^\infty_X(-)$ is viewed as an $A$-module via the \emph{right} $A$-module structure. However, there may be certain situations where we wish to use the \emph{left} $A$-module structure, which we will try to make clear. 

The first step towards describing modules over the monad $\mathcal{D}^\infty_X$ as modules over a ring is the following. 
\begin{lem}\label{lem:naturaltransformlema}
Let $A$ be a (derived) affinoid algebra and let  $X=\operatorname{dSp}(A)$. Then the object $\mathcal{D}^\infty_X1_X$ acquires the canonical structure of an algebra object in the $\infty$-category 
\begin{equation}
    {}_A\operatorname{BMod}_A D(\mathsf{CBorn}_K) = \operatorname{QCoh}(X \times X)
\end{equation} 
of $A$-$A$ bimodule objects under convolution. There is a canonical morphism of monads
\begin{equation}\label{eq:naturaltransformlema}
(-) \widehat{\otimes}_X \mathcal{D}^\infty_X 1_X  \to \mathcal{D}^\infty_X(-).
\end{equation}
We emphasise that the tensor product is taken with respect to the \emph{right} $A$-module structure on $\mathcal{D}^\infty_X1_X$. The $A$-module structure on the left side of \eqref{eq:naturaltransformlema} comes from the \emph{left} $A$-module structure on $\mathcal{D}^\infty_X1_X$.
\end{lem}
\begin{proof}
We note that the functor $p^!$ is right adjoint to $p_!$ which is $\operatorname{QCoh}(*)$-linear. Therefore by Theorem \ref{HaugsengLaxTheoremC}, the functor $p^!$ and hence also $\mathcal{D}^\infty_X = p^!p_!$ acquires a canonical lax $\operatorname{QCoh}(*)$-linear structure. Applying the functor $\kappa$ of Corollary \ref{cor:Bimodmonoidal}, then $\mathcal{D}^\infty_X1_X$ acquires the structure of an algebra object in the category of $A$-$A$ bimodule objects under convolution. Further, the morphism \eqref{eq:naturaltransformlema} of monads is obtained from the counit of the adjunction induced by Corollary \ref{cor:Bimodmonoidal} on algebra objects. 
\end{proof}
\begin{rmk}
One can alternatively construct the algebra structure on $\mathcal{D}^\infty_X1_X$ via an ``adjoint" Fourier--Mukai transform. Namely, the usual Fourier--Mukai transform gives a (strongly monoidal) functor
\begin{equation}\label{eq:FMexample1}
    \operatorname{QCoh}(X \times X) \to \operatorname{Fun}^L_{\operatorname{QCoh}(*)}(\operatorname{QCoh}(X), \operatorname{QCoh}(X))
\end{equation}
By Theorem \ref{HaugsengLaxTheoremC} there is a strongly monoidal functor 
\begin{equation}\label{eq:FMexample2}
    \operatorname{Fun}^L_{\operatorname{QCoh}(*)}(\operatorname{QCoh}(X), \operatorname{QCoh}(X)) \to \operatorname{Fun}^{R, \mathrm{lax}}_{\operatorname{QCoh}(*)}(\operatorname{QCoh}(X), \operatorname{QCoh}(X))^{\mathsf{op}},
\end{equation}
obtained by passing to adjoints.  By Corollary \ref{cor:Bimodmonoidal} there is an adjunction 
\begin{equation}\label{eq:FMexample3}
\operatorname{QCoh}(X \times X) \leftrightarrows \operatorname{Fun}^{R, \mathrm{lax}}_{\operatorname{QCoh}(*)}(\operatorname{QCoh}(X), \operatorname{QCoh}(X)),
\end{equation}
in which the left adjoint is strongly monoidal (for convolution), hence the right adjoint is canonically lax monoidal. The object $(A \widehat{\otimes}_K A)^\dagger_\Delta \in \operatorname{QCoh}(X \times X)$ is a coalgebra under convolution. The image of this object under the composite of \eqref{eq:FMexample1}, \eqref{eq:FMexample2} and the right adjoint in \eqref{eq:FMexample3} gives the object $\mathcal{D}^\infty_X1_X \in \operatorname{QCoh}(X \times X)$ together with its algebra object structure (with respect to convolution).
\end{rmk}
The main point of the rest of this section is to identify a class of objects such that the natural transformation \eqref{eq:naturaltransformlema} restricts to an equivalence on such objects. Furthermore, the class of such objects should be large enough to include the examples of interest, e.g. the underlying $A$-modules of $\mathcal{C}$-complexes \cite[\S 8]{bode_six_2021}, and be preserved by the monad $\mathcal{D}^\infty_X$. For this purpose we introduce the following definition.
\begin{defn}
Let $A \in \mathsf{dAfnd}$ and let $X = \operatorname{dSp}(A)$.
\begin{enumerate}[(i)]
    \item We define 
\begin{equation}
   \operatorname{Fr}(X) \subseteq \operatorname{QCoh}(X) 
\end{equation}
to be the full sub-$\infty$-category spanned by those objects $M^\bullet$ whose underlying object $M^\bullet \in D(\mathsf{CBorn}_K)$ is such that, for each $j \in \mathbf{Z}$, $H^j(M^\bullet)$ is a Fr\'echet space.
\item We define 
\begin{equation}
    \operatorname{sFr}(X) \subseteq \operatorname{Fr}(X) 
\end{equation}
to be the full sub-$\infty$-category spanned by those objects $M^\bullet$ such that $A_U \widehat{\otimes}^\mathbf{L}_A M^\bullet \in \operatorname{Fr}(U)$ for every affinoid subdomain $U \subseteq X$. We may refer to such objects as \emph{stably Fr\'echet} complexes. 
\end{enumerate}
\end{defn}
Before proceeding further we fix notations as in \cite[\S4.16]{soor_six-functor_2024}. If $X = \operatorname{Sp}(A)$ is a classical affinoid rigid space equipped with an \'etale morphism $X \to \mathbf{D}^r_K$, we let $x_1,\dots,x_r \in A$ be the corresponding \'etale coordinates. For every $0 \leqslant m < \infty$ we define the $K$-linear pairing
\begin{equation}
    K \langle dx/p^m\rangle \times K \langle p^m \partial \rangle \to K 
\end{equation}
by $((dx)^\alpha, \partial^\beta) := \alpha! \delta_{\alpha \beta}$, for every pair of multi-indices $\alpha, \beta \in \mathbf{N}^r$. We define $K \langle dx/p^\infty \rangle := \operatorname{colim}_m K \langle dx/p^m \rangle$ and $K\langle p^\infty \partial \rangle := \operatorname{lim}_m K \langle p^m \partial \rangle$ where the (co)limits are taken in $\mathsf{CBorn}_K$. 
\begin{lem}\label{lem:frechetiso}
    If $V \in \mathsf{CBorn}_K$ is a Fr\'echet space, then the canonical morphism
    \begin{equation}\label{eq:frechetiso}
        V \widehat{\otimes}_K^\mathbf{L} K \langle p^\infty \partial \rangle \to R \underline{\operatorname{Hom}}_K(K \langle dx/p^\infty \rangle, V) 
    \end{equation}
    is an equivalence in $D(\mathsf{CBorn}_K)$. Further, the natural morphism 
    \begin{equation}
        V \widehat{\otimes}_K^\mathbf{L} K \langle p^\infty \partial \rangle \to V \widehat{\otimes}_K K \langle p^\infty \partial \rangle
    \end{equation} 
    is an equivalence, so that both sides are concentrated in degree $0$.
\end{lem}
\begin{proof}
By Lemma \ref{lem:frechetAleph1} $V$ can be presented as an $\aleph_1$-filtered colimit 
\begin{equation}
    V \simeq \underset{[B] \in \mathfrak{S}(V)}{\operatorname{colim}} V_B
\end{equation} 
of Banach spaces. By Proposition \ref{prop:CbornKproperties}(iii) this is even the colimit in $D(\mathsf{CBorn}_K)$. Using this together with the fact that we can exchange countable limits with $\aleph_1$-filtered colimits (by Proposition \ref{prop:Univprops1}(vii) and Proposition \ref{prop:compactgenprops}(ii)), we obtain 
\begin{equation}
\begin{aligned}
     R \underline{\operatorname{Hom}}_K(K \langle dx/p^\infty \rangle, V) &\simeq R\underset{n}{\operatorname{lim}}L\underset{[B]}{\operatorname{colim}}  \underline{\operatorname{Hom}}_K(K \langle dx/p^n \rangle, V_B) \\
     &\simeq L\underset{[B]}{\operatorname{colim}}R\underset{n}{\operatorname{lim}} \underline{\operatorname{Hom}}_K(K \langle dx/p^n \rangle, V_B).
\end{aligned}
\end{equation}
Now cofinality together with the Mittag-Leffler result of \cite[Theorem 5.24]{bode_six_2021} implies that
\begin{equation}
    R\underset{n}{\operatorname{lim}} \underline{\operatorname{Hom}}_K(K \langle dx/p^n \rangle, V_B) \simeq R\underset{n}{\operatorname{lim}} ( V_B  \widehat{\otimes}_KK \langle  p^n \partial\rangle ) \simeq \underset{n}{\operatorname{lim}} ( V_B \widehat{\otimes}_K K \langle p^n \partial\rangle  ).
\end{equation}
viewed as an object in degree $0$. Next we note that 
\begin{equation}
   V_B \widehat{\otimes}_K K\langle p^\infty \partial\rangle\cong  \underset{n}{\operatorname{lim}} (V_B \widehat{\otimes}_K K\langle p^n \partial \rangle),
\end{equation}
in $\mathsf{CBorn}_K$, because both sides can be written as rapidly decreasing series with coefficients in the Banach space $V_B$. Because $K\langle p^\infty \partial\rangle$ is \emph{strongly flat} \cite[Corollary 5.36]{bode_six_2021}, we obtain $V_B \widehat{\otimes}_K K\langle p^\infty \partial\rangle \simeq V_B \widehat{\otimes}_K^\mathbf{L} K\langle p^\infty \partial\rangle$. Hence we may conclude by using that $\widehat{\otimes}^\mathbf{L}_K$ is compatible with colimits (separately in each variable). 
\end{proof}
Let us continue to recall some objects and morphisms introduced in \cite[\S4.16]{soor_six-functor_2024}. We previously defined the \emph{germ of the zero section} as an object of $\mathsf{dAff}$:
\begin{equation}
    (X \subseteq TX)^\dagger = \operatorname{dSp}( A\widehat{\otimes}_K K \langle dx/p^\infty \rangle).
\end{equation}
There is a augmentation $\varepsilon: A\widehat{\otimes}_K K \langle dx/p^\infty \rangle \to A$ and two algebra morphisms $\sigma, \tau: A \to A\widehat{\otimes}_K K \langle dx/p^\infty \rangle$. One has 
\begin{equation}
    \sigma:= \operatorname{id} \otimes 1 : A \to A \widehat{\otimes}_KK \langle dx/p^\infty\rangle,
\end{equation}
which gives the \emph{left} $A$-module structure on $A \widehat{\otimes}_KK \langle dx/p^\infty\rangle$. The algebra morphism $\tau$ sends a function to its Taylor series:
\begin{equation}
    \tau(a) := \sum_{\alpha \in \mathbf{N}^r} \frac{1}{\alpha!} \partial^\alpha (a)\otimes (dx)^\alpha.  
\end{equation}
This gives the \emph{right} $A$-module structure on $A \widehat{\otimes}_KK \langle dx/p^\infty\rangle$. There is an algebra morphism
\begin{equation}\label{eq:Psimorphism}
    \psi: A \widehat{\otimes} K \langle dx/p^\infty\rangle \to (A \widehat{\otimes}_K K \langle dx/p^\infty\rangle) \widehat{\otimes}_A (A \widehat{\otimes}_K K \langle dx/p^\infty\rangle)
\end{equation}
determined by $\psi(a\otimes 1) = (a \otimes 1) \otimes (1 \otimes 1)$ and 
\begin{equation}
    \psi(1 \otimes dx_i) = (1\otimes dx_i) \otimes (1 \otimes 1) + (1 \otimes 1) \otimes (1 \otimes dx_i). 
\end{equation}
It is important to remember that the tensor product \eqref{eq:Psimorphism} is taken with respect to the right $A$-module structure (via $\tau$) on the first factor, and the left  $A$-module structure on the second factor. Using the morphisms $\varepsilon, \sigma, \tau, \psi$ we obtain a groupoid object
\begin{equation}
\begin{tikzcd}
	\cdots & {(X \subseteq TX)^\dagger \times_X (X \subseteq TX)^\dagger} & {(X \subseteq TX)^\dagger} & X
	\arrow[shift right, from=1-1, to=1-2]
	\arrow[shift left, from=1-1, to=1-2]
	\arrow[shift right=3, from=1-1, to=1-2]
	\arrow[shift left=3, from=1-1, to=1-2]
	\arrow[from=1-2, to=1-3]
	\arrow[shift left=2, from=1-2, to=1-3]
	\arrow[shift right=2, from=1-2, to=1-3]
	\arrow[shift right, from=1-3, to=1-4]
	\arrow[shift left, from=1-3, to=1-4]
\end{tikzcd}
\end{equation}
where we suppressed the degeneracy maps and we emphasise that the fiber product is taken with respect to $\sigma$ and $\tau$. Let us call this groupoid object $\operatorname{exp}(\mathcal{T}_X)$. Now we recall the following from \cite[Theorem 4.97]{soor_six-functor_2024}:
\begin{thm}\label{thm:ExpTX}
With notations as above. The morphism
\begin{equation}
    A \widehat{\otimes}_K K \langle dx/p^\infty \rangle \to (A \widehat{\otimes}_KA)^\dagger_\Delta
\end{equation}
is an equivalence in $\mathsf{dAlg}$ and induces an equivalence 
\begin{equation}
    \operatorname{exp}(\mathcal{T}_X) \simeq (X \subseteq X^{\bullet +1})^\dagger
\end{equation}
of simplicial objects in $\mathsf{dAff}$.
\end{thm}
In the following Proposition we improve \cite[Proposition 4.102]{soor_six-functor_2024} to an equivalence of algebra objects (not just $A$-modules). 
\begin{prop}\label{prop:algebraiso}
Let $X = \operatorname{Sp}(A)$ be a classical affinoid equipped with an \'etale morphism $X \to \mathbf{D}^r_K$. There is an equivalence $\mathcal{D}^\infty_X1_X \simeq \wideparen{\mathcal{D}}_X(X)$ of algebra objects in $\operatorname{QCoh}(X)$.
\end{prop}
\begin{proof}
Let $q: X \to \operatorname{X}/\operatorname{exp}(\mathcal{T}_X)$ be the canonical morphism. Thanks to Theorem \ref{thm:ExpTX} and the construction of Lemma \ref{lem:naturaltransformlema} we know that there is an equivalence of algebra objects\footnote{In the category of $A$-$A$ bimodule objects with the convolution monoidal structure.} $\mathcal{D}^\infty_X1_X \simeq q^!q_!1_X$. Now by base-change one has 
\begin{equation}
    q^!q_! 1_X \simeq R\underline{\operatorname{Hom}}_A(A \widehat{\otimes}_KK \langle dx/p^\infty \rangle , A)
\end{equation}
and using Lemma \ref{lem:frechetiso} above we know that 
\begin{equation}\label{eq:Dcapiso}
   \wideparen{\mathcal{D}}_X(X) \simeq A \widehat{\otimes}_K K \langle p^\infty \partial \rangle  \simeq R\underline{\operatorname{Hom}}_A(A \widehat{\otimes}_KK \langle dx/p^\infty \rangle , A), 
\end{equation}
as left $A$-modules, so we deduce that $q^!q_! 1_X$ is (strongly) flat and concentrated in degree $0$. In particular it comes from an algebra object of the ordinary category ${}_A\operatorname{BMod}_A\mathsf{CBorn}_K$. We recall that the isomorphism \eqref{eq:Dcapiso} identifies each $\partial^\alpha$ with the $A$-linear map determined by 
\begin{equation}\label{eq:idenpartial}
    \partial^\alpha((dx)^\beta)=\alpha! \delta_{\alpha \beta}.
\end{equation}
We recall the definitions of the algebra morphisms $\varepsilon, \sigma, \tau$ and $\psi$ from the previous page. We regard $\underline{\operatorname{Hom}}_A(A \widehat{\otimes}_K K \langle dx/p^\infty\rangle, A)$ as a left $A$-module via $(a.\eta)(-) := \eta(\sigma(a) \cdot -)$ and as a right $A$-module via $(\eta.a)(-) := \eta(\tau(a)\cdot-)$, for $a \in A$ and $\eta \in \underline{\operatorname{Hom}}_A(A \widehat{\otimes}_K K \langle dx/p^\infty\rangle, A)$.  Under the identification \eqref{eq:idenpartial}, the right and left actions become 
\begin{equation}
\begin{aligned}
    a.\partial^\alpha = a\partial^\alpha && \text{ and } && \partial^\alpha.a = \sum_{\beta + \gamma = \alpha} \binom{\beta}{ \alpha} \partial^\beta(a)\partial^\gamma. 
\end{aligned}
\end{equation}
We need to check that the composite 
\begin{multline}\label{eq:composite}
\underline{\operatorname{Hom}}_A(A \widehat{\otimes}_K K \langle dx/p^\infty\rangle, A) \widehat{\otimes}_A \underline{\operatorname{Hom}}_A(A \widehat{\otimes}_K K \langle dx/p^\infty\rangle, A) \\
\to \underline{\operatorname{Hom}}_A((A \widehat{\otimes}_K K \langle dx/p^\infty\rangle) \widehat{\otimes}_A (A \widehat{\otimes}_K K \langle dx/p^\infty\rangle),A)\\
\xrightarrow[]{\psi^\lor} \underline{\operatorname{Hom}}_A(A \widehat{\otimes}_K K \langle dx/p^\infty\rangle, A)
\end{multline}
agrees with the multiplication on $\wideparen{\mathcal{D}}_X(X)$. In the first line, the tensor product is taken with respect to the right $A$-module structure on the first factor and the left $A$-module structure on the second factor. To be completely explicit, the first morphism sends $\eta \otimes \eta^\prime$ to the morphism $\eta \widetilde{\otimes} \eta^\prime$ determined by 
\begin{equation}
   (\eta \widetilde{\otimes}\eta^\prime)(j\otimes j^\prime) := (\eta.\eta^\prime(j^\prime))(j) = \eta(\tau(\eta^\prime(j^\prime))j),
\end{equation}
for $j, j^\prime \in A \widehat{\otimes}_K K \langle dx/p^\infty\rangle$. One checks that 
\begin{equation}\label{eq:partialproduct}
    \partial^\alpha \widetilde{\otimes} \partial^\beta((1\otimes dx)^\gamma \otimes (1\otimes dx)^\epsilon) = \alpha! \beta!\delta_{\alpha \gamma} \delta_{\beta \epsilon}.
\end{equation} 
Let us denote the composite \eqref{eq:composite} by $m$. It follows from \eqref{eq:partialproduct} that
\begin{equation}
\begin{aligned}
    m(\partial^\alpha\otimes\partial^\beta)((1 \otimes dx)^\gamma) &= (\partial^\alpha\widetilde{\otimes}\partial^\beta)(\psi((1\otimes dx)^\gamma)) \\ 
    &= (\partial^\alpha\widetilde{\otimes}\partial^\beta)\Big(\binom{\gamma}{\delta} (1 \otimes dx)^\delta \otimes (1 \otimes dx)^\epsilon\Big)\\
    &= \gamma ! \delta_{\alpha + \beta, \gamma } \\
    & = \partial^{\alpha + \beta }((1\otimes dx)^\gamma),
\end{aligned}
\end{equation}
so that $m(\partial^\alpha\otimes\partial^\beta) = \partial^{\alpha + \beta }$. Using this together with the fact that $m$ is balanced and left $A$-linear, one has 
\begin{equation}
\begin{aligned}
    m(f\partial^\alpha \otimes g\partial^\beta) &= m((f.\partial^\alpha.g) \otimes \partial^\alpha) \\
    &=  m\Big( \sum_{\eta + \nu = \alpha} \binom{\alpha}{\eta} f \partial^\eta(g)\partial^\nu \otimes \partial^\beta\Big) \\
    &= \sum_{\eta + \nu = \alpha} \binom{\alpha}{\eta} f \partial^\eta(g)\partial^{\nu+\beta}.
\end{aligned}
\end{equation}
Therefore $m(f\partial^\alpha \otimes g\partial^\beta) = \sum_{\eta + \nu = \alpha} \binom{\alpha}{\eta}f \partial^\eta(g) \partial^{\nu + \beta}$ agrees with the multiplication on $\wideparen{\mathcal{D}}_X(X)$, as required. 
\end{proof}

\begin{thm}\label{thm:PNFtensor}
Suppose that $X = \operatorname{Sp}(A)$ is a smooth classical affinoid equipped with an \'etale morphism $X \to \mathbf{D}^r_K$. Let $M^\bullet \in \operatorname{Fr}(X)$. Then the canonical morphism
\begin{equation}
     M^\bullet \widehat{\otimes}_X \mathcal{D}^\infty_X1_X \to \mathcal{D}^\infty_XM^\bullet
\end{equation}
of Lemma \ref{lem:naturaltransformlema}, is an equivalence.
\end{thm}
\begin{proof}
Using Theorem \ref{thm:ExpTX}, what we need to show is that the canonical morphism 
\begin{equation}\label{eq:canonical45}
    M^\bullet  \widehat{\otimes}^\mathbf{L}_AR\underline{\operatorname{Hom}}_K(K\langle dx/p^\infty \rangle,A) \to  R\underline{\operatorname{Hom}}_K(K\langle dx/p^\infty \rangle,M^\bullet)
\end{equation}
is an equivalence in $D(\mathsf{CBorn}_K)$. The right side of \eqref{eq:canonical45} is equivalent to 
\begin{equation}
    R \underset{n}{\operatorname{lim}} \underline{\operatorname{Hom}}_K(K\langle dx/p^n \rangle,M^\bullet),
\end{equation} 
and therefore by \cite[Lemma 3.3]{bode_six_2021}, for each $j \in \mathbf{Z}$ we obtain a short-exact sequence 
    \begin{multline}
         0 \to R^1 \underset{n}{\operatorname{lim}} \underline{\operatorname{Hom}}_K(K\langle dx/p^n \rangle,H^{j-1}(M^\bullet)) \\ \to H^j(R\underline{\operatorname{Hom}}_K(K\langle dx/p^\infty \rangle,M^\bullet))  \\ \to \underset{n}{\operatorname{lim}} \underline{\operatorname{Hom}}_K(K\langle dx/p^n \rangle,H^{j}(M^\bullet)) \to 0.
    \end{multline}
    Now by Lemma \ref{lem:frechetiso}, the first term is zero and the second term is isomorphic to 
    \begin{equation}
        H^{j}(M^\bullet) \widehat{\otimes}_K K \langle p^\infty \partial \rangle \cong H^j( M^\bullet \widehat{\otimes}^\mathbf{L}_K  K \langle p^\infty \partial \rangle), 
    \end{equation}
    where we again used \cite[Corollary 5.36]{bode_six_2021}. This shows that \eqref{eq:canonical45} is an isomorphism after taking cohomology, and therefore an equivalence. 
\end{proof}
\begin{prop}\label{prop:PNFpreserve}
Let $X = \operatorname{Sp}(A)$ be a smooth classical affinoid which admit an \'etale morphism to a polydisk. Then the endofunctors $\mathcal{D}^\infty_X(-)$ and $(-)\widehat{\otimes}_X \mathcal{D}^\infty_X1_X $ preserve the full subcategories $\operatorname{Fr}(X)$ and $\operatorname{sFr}(X)$ of $\operatorname{QCoh}(X)$. 
\end{prop}
\begin{proof}
Let $M^\bullet \in \operatorname{QCoh}(X)$. After forgetting to $D(\mathsf{CBorn}_K)$, the object $M^\bullet \widehat{\otimes}^\mathbf{L}_X\mathcal{D}^\infty_X 1_X $ is nothing but 
\begin{equation}
    M^\bullet \widehat{\otimes}_A^\mathbf{L} A\widehat{\otimes}^\mathbf{L}_KK \langle p^\infty \partial\rangle  \simeq  M^\bullet \widehat{\otimes}_K^\mathbf{L} K \langle p^\infty \partial\rangle.
\end{equation}
where we used \cite[Corollary 5.36]{bode_six_2021}, and by the result of \emph{loc. cit.} again one has $H^j(M^\bullet \widehat{\otimes}_K^\mathbf{L} K \langle p^\infty \partial\rangle) \cong H^j(M^\bullet) \widehat{\otimes}_K K \langle p^\infty \partial\rangle$. This proves that $(-)\widehat{\otimes}_X \mathcal{D}^\infty_X1_X$ preserves $\operatorname{Fr}(X)$. In order to show that $(-)\widehat{\otimes}_X \mathcal{D}^\infty_X1_X$ preserves $\operatorname{sFr}(X) \subseteq \operatorname{Fr}(X)$ it then suffices to show that the natural morphism $\mathcal{D}^\infty_X 1_X \widehat{\otimes}^\mathbf{L}_A A_U \to \mathcal{D}^\infty_U1_U$ is an equivalence, for each affinoid subdomain $U \subseteq X$. This can be deduced (for instance) from Proposition \ref{prop:algebraiso}, because $  \wideparen{\mathcal{D}}_X(X)\xrightarrow[]{\sim} \wideparen{\mathcal{D}}_U(U) \widehat{\otimes}^\mathbf{L}_A A_U$. \end{proof}
\begin{cor}\label{cor:Dmodequivalence}
Let $X = \operatorname{Sp}(A)$ be a smooth classical affinoid equipped with an \'etale morphism $X \to \mathbf{D}^r_K$. The morphism \eqref{eq:naturaltransformlema} restricts to an equivalence of monads on $\operatorname{Fr}(X)$. Consequently, there is an equivalence of $\infty$-categories
    \begin{equation}
\operatorname{RMod}_{\mathcal{D}^\infty_X1_X}\operatorname{Fr}(X) \simeq  \operatorname{Mod}_{\mathcal{D}^\infty_X}\operatorname{Fr}(X).
    \end{equation}
The same holds with $\operatorname{sFr}(X)$ in place of $\operatorname{Fr}(X)$.
\end{cor}
\begin{proof}
Follows by assembling Theorem \ref{thm:PNFtensor} and Proposition \ref{prop:PNFpreserve}.
\end{proof}
\emph{In the remainder of this subsection $X = \operatorname{Sp}(A)$ denotes a classical affinoid equipped with an \'etale morphism $X \to \mathbf{D}^r_K$}. We recall the definition of the Banach completed differential operators $\mathcal{D}^n_X(X)$ from \cite[\S 2]{bode_six_2021}. These are Noetherian Banach algebras and $\wideparen{\mathcal{D}}_X(X) = \operatorname{lim}_n \mathcal{D}^n_X(X)$ gives a presentation of $\wideparen{\mathcal{D}}_X(X)$ as a Fr\'echet--Stein algebra. 
\begin{defn}\label{defn:Ccomplexprime}
An object $M^\bullet \in \operatorname{RMod}_{\wideparen{\mathcal{D}}_X(X)}D(\mathsf{CBorn}_K)$ is called a \emph{$\mathcal{C}$-complex} if:
\begin{enumerate}[(i)]
    \item each $M_n^\bullet := M^\bullet \widehat{\otimes}^\mathbf{L}_{\wideparen{\mathcal{D}}_X(X)} \mathcal{D}^n_X(X)$ is such that each  $H^j(M_n^\bullet) $ is a finitely-generated $\mathcal{D}^n_X(X)$-module and $H^j(M_n^\bullet) = 0$ for $|j| \gg 0$;
    \item the canonical morphism $M^\bullet \to R \operatorname{lim}_n M_n^\bullet$ is an equivalence. 
\end{enumerate}
We denote the full subcategory spanned by such objects, by $D_\mathcal{C}(X)$. 
\end{defn}
\begin{rmk}\label{rmk:iiprime}
By \cite[Lemma 8.11]{bode_six_2021}, condition (ii) in Definition \ref{defn:Ccomplexprime} can be replaced with the following (which may be easier to check in practice):
\begin{enumerate}
    \item[(ii)${}^\prime$] for each $j \in \mathbf{Z}$ the canonical morphism $H^j(M^\bullet) \to\operatorname{lim}_n H^j(  M_n^\bullet)$ is an isomorphism.
\end{enumerate}
\end{rmk}
\begin{lem}\label{lem:Ccomplex}
Suppose that $M^\bullet$ is a $\mathcal{C}$-complex. Then the  underlying object $M^\bullet \in \operatorname{QCoh}(X)$ belongs to the full subcategory $\operatorname{Fr}(X) \subseteq \operatorname{QCoh}(X)$, so that one has an inclusion 
\begin{equation}
    D_\mathcal{C}(X) \subseteq \operatorname{RMod}_{\wideparen{\mathcal{D}}_X(X)}\operatorname{Fr}(X).
\end{equation}
\end{lem}
\begin{proof}
This is clear from Remark \ref{rmk:iiprime}.
\end{proof}
We recall the following Theorem from \cite[Theorem 4.101]{soor_six-functor_2024}:
\begin{thm}\label{thm:StratMonadicity}
Let $X = \operatorname{dSp}(A)$ be a smooth classical affinoid which is \'etale over a polydisk. Then the adjunction $p^! \dashv p_!$ is monadic, so that there is an equivalence of $\infty$-categories 
\begin{equation}
    \operatorname{Strat}(X) \simeq \operatorname{Mod}_{\mathcal{D}^\infty_X } \operatorname{QCoh}(X). 
\end{equation}
\end{thm}
Assembling all the above together with Theorem \ref{thm:StratMonadicity} we may draw the following diagram relating various categories.
\begin{equation}\label{eq:bigdiagram}
\begin{tikzcd}
	& {\operatorname{Strat}(X)} \\
	{\operatorname{Mod}_{\mathcal{D}^\infty_X}\operatorname{Fr}(X)} & {\operatorname{Mod}_{\mathcal{D}^\infty_X}\operatorname{QCoh}(X)} \\
	{\operatorname{RMod}_{\mathcal{D}^\infty_X1_X}\operatorname{Fr}(X)} & {\operatorname{RMod}_{\mathcal{D}^\infty_X1_X}\operatorname{QCoh}(X)} \\
	{\operatorname{RMod}_{\wideparen{\mathcal{D}}_X(X)}\operatorname{Fr}(X)} & {\operatorname{RMod}_{\wideparen{\mathcal{D}}_X(X)}\operatorname{QCoh}(X)} \\
	{D_\mathcal{C}(X)}
	\arrow[hook, from=2-1, to=2-2]
	\arrow["\simeq"', no head, from=2-2, to=1-2]
	\arrow["\simeq"', no head, from=3-1, to=2-1]
	\arrow[hook, from=3-1, to=3-2]
	\arrow["{\simeq }"', no head, from=4-1, to=3-1]
	\arrow[hook, from=4-1, to=4-2]
	\arrow["\simeq"', no head, from=4-2, to=3-2]
	\arrow[hook', from=5-1, to=4-1]
\end{tikzcd}
\end{equation}
In particular we obtain the following.
\begin{thm}
Let $X= \operatorname{Sp}(A)$ be a smooth classical affinoid equipped with an \'etale morphism $X \to \mathbf{D}^r_K$. Then there is a fully-faithful functor of $\infty$-categories
\begin{equation}
    D_\mathcal{C}(X) \hookrightarrow \operatorname{Strat}(X). 
\end{equation}
\end{thm}
\section{Descent for $\wideparen{\mathcal{D}}$-modules}\label{sec:DescentforDCap}
In this section we continue to let $X = \operatorname{Sp}(A)$ be a smooth affinoid equipped with an \'etale morphism $X \to \mathbf{D}^r_K$.
Let $X_w$ (resp. $X_{n,w}$) denote the poset of affinoid subdomains (resp. $p^n$-accessible subdomains\footnote{By this we mean $p^n\mathcal{T}_X$-accessible, in the sense of \cite[\S4.5]{Dcap1}.}) of $X$. By using \cite[Proposition 4.6.2.17]{HigherAlgebra} and unstraightening we obtain functors 
\begin{equation}
\begin{aligned}
    \operatorname{RMod}_{\wideparen{\mathcal{D}}_X(-)}D(\mathsf{CBorn}_K) : X_w^\mathsf{op} &\to \mathsf{Cat}_\infty\\
    \operatorname{RMod}_{\mathcal{D}^n_X(-)}D(\mathsf{CBorn}_K) : X_{n,w}^\mathsf{op} &\to \mathsf{Cat}_\infty
\end{aligned}
\end{equation}
which send $t: V \hookrightarrow U$ to the pullback functors $(-)\widehat{\otimes}_{\wideparen{\mathcal{D}}_X(U)}^\mathbf{L}\wideparen{\mathcal{D}}_X(V)$ and $(-)\widehat{\otimes}_{\mathcal{D}^n_X(U)}^\mathbf{L}\mathcal{D}^n_X(V)$ respectively. We recall that $\wideparen{\mathcal{D}}_X$ (resp. $\mathcal{D}^n_X(X)$) is a sheaf of algebras on $X_w$ (resp. $X_{n,w}$). More precisely, we recall that Ardakov and Wadsley have proved the counterpart of Tate acyclicity for this Fr\'echet-Stein algebra.  
\begin{prop}\label{prop:DCapdescent}
Let $X_w$ (resp. $X_{n,w}$) be the poset of affinoid subdomains (resp. $p^n$-accessible subdomains) of $X$ equipped with the \emph{weak G-topology}. Then:
\begin{enumerate}[(i)]
    \item \cite[\S 8.1]{Dcap1} Let $\{U_i \to X\}_{i=1}^s$ be a covering in $X_w$. Then the augmented (alternating) \v{C}ech complex $C^\bullet_{\mathrm{aug}}(\{U_i\},\wideparen{\mathcal{D}}_X)$ is exact. 
    \item \cite[Theorem 3.5]{Dcap1} Let $\{U_i \to X\}_{i=1}^s$ be a covering in $X_{w,n}$.Then the augmented (alternating) \v{C}ech complex $C^\bullet_{\mathrm{aug}}(\{U_i\},\mathcal{D}^n_X)$ is exact. 
\end{enumerate}
\end{prop}
\begin{rmk}
This immediately implies that $C^\bullet_{\mathrm{aug}}(\{U_i\},\wideparen{\mathcal{D}}_X)$ (resp. $C^\bullet_{\mathrm{aug}}(\{U_i\},\mathcal{D}^n_X)$) is strictly exact: as it is a complex of Fr\'echet (resp. Banach) spaces, we can appeal to the open mapping theorem. 
\end{rmk}
\begin{lem}\label{lem:DCaphomotopyepi}
\begin{enumerate}[(i)]
    \item Let $U, V\subseteq X$ be affinoid subdomains. Then the canonical morphism
\begin{equation}\label{eq:DCapderived}
\wideparen{\mathcal{D}}_X(U) \widehat{\otimes}^\mathbf{L}_{\wideparen{\mathcal{D}}_X(X)} \wideparen{\mathcal{D}}_X(V) \to \wideparen{\mathcal{D}}_X(U\cap V)
\end{equation}
is an equivalence of $\wideparen{\mathcal{D}}_X(U)$-$\wideparen{\mathcal{D}}_X(V)$ bimodule objects in $D(\mathsf{CBorn}_K)$. 
\item Let $U, V \subseteq X$ be $p^n$-accessible affinoid subdomains. Then the canonical morphism \begin{equation}
\mathcal{D}^n_X(U) \widehat{\otimes}^\mathbf{L}_{\mathcal{D}^n_X(X)} \mathcal{D}^n_X(V) \to \mathcal{D}^n_X(U\cap V)
\end{equation}
is an equivalence of $\mathcal{D}^n_X(U)$-$\mathcal{D}^n_X(V)$ bimodule objects in $D(\mathsf{CBorn}_K)$. 
\end{enumerate}
\end{lem}
\begin{proof}
We only prove (i) as the proof of (ii) is very similar. Let $A_U, A_V, A_{U \cap V}$ denote the corresponding affinoid localizations. Using the isomorphism $\wideparen{\mathcal{D}}_X(U) \simeq A_U \widehat{\otimes}_K^\mathbf{L} K \langle p^\infty \partial \rangle$, and associativity properties of the (derived) tensor product, the morphism \label{eq:Dcapderived} is equivalent to the morphism \begin{equation}
A_U \widehat{\otimes}^\mathbf{L}_{A} A_V \widehat{\otimes}_K^\mathbf{L} K \langle p^\infty \partial \rangle \to A_{U \cap V} \widehat{\otimes}_K^\mathbf{L} K \langle p^\infty \partial\rangle,
\end{equation} 
which evidently comes from $A_U \widehat{\otimes}^\mathbf{L}_A A_V \to A_{U \cap V}$ by tensoring on the right. The latter is an equivalence by \cite[Theorem 5.16]{BBKNonArch}.
\end{proof}
\begin{lem}\label{lem:Dcapdescent}
With notations as above. Let $\{U_i \to X\}_{i=1}^s$ be an admissible covering of $X$ by affinoid subdomains. Then:
\begin{enumerate}[(i)]
    \item The morphism $\wideparen{\mathcal{D}}_X(X) \to \prod_{i=1}^s \wideparen{\mathcal{D}}_X(U_i)$ is descendable in the sense of Definition \ref{defn:descendableNC}. 
    \item Let $Y := \coprod_{i=1}^s U_i$. Then the augmented simplicial object $\wideparen{\mathcal{D}}_X(Y^{\bullet+1/X})$ satisfies the Beck-Chevalley condition of Definition 
    \ref{defn:Beckchevalley}. 
    \item The canonical morphism  \begin{equation}\operatorname{RMod}_{\wideparen{\mathcal{D}}_X(X)}D(\mathsf{CBorn}_K) \to \underset{[m] \in \Delta}{\operatorname{lim}} \operatorname{RMod}_{\wideparen{\mathcal{D}}_X(Y^{m+1/X})}D(\mathsf{CBorn}_K)
    \end{equation}
    is an equivalence of $\infty$-categories. In particular $\operatorname{RMod}_{\wideparen{\mathcal{D}}_X(-)}D(\mathsf{CBorn}_K)$
    is a sheaf on $X_w$. 
\end{enumerate}
\end{lem}
\begin{proof}
(i): By combining Proposition \ref{prop:DCapdescent} with Lemma \ref{lem:DCaphomotopyepi}, and using the Dold-Kan correspondence, one deduces that 
\begin{equation*}
    \wideparen{\mathcal{D}}_X(X) \to R\operatorname{lim} \Big(
\begin{tikzcd}[column sep=small]
	{\prod_{i} \wideparen{\mathcal{D}}_X(U_i)} & {\prod_{i < j} \wideparen{\mathcal{D}}_X(U_i) \widehat{\otimes}^\mathbf{L}_{\wideparen{\mathcal{D}}_X(X)}\wideparen{\mathcal{D}}_X(U_j)} & \cdots
	\arrow[shift right, from=1-1, to=1-2]
	\arrow[shift left, from=1-1, to=1-2]
	\arrow[from=1-2, to=1-3]
	\arrow[shift right=2, from=1-2, to=1-3]
	\arrow[shift left=2, from=1-2, to=1-3]
\end{tikzcd} \Big),
\end{equation*}
as $\wideparen{\mathcal{D}}_X(X)$-$\wideparen{\mathcal{D}}_X(X)$ bimodule objects in $D(\mathsf{CBorn}_K)$. We note that the limit on the right is finite because we used the alternating \v{C}ech complex. This establishes (i).

(ii): This is immediate from Lemma \ref{lem:DCaphomotopyepi}. 

(iii): By using (i) and (ii) above, this follows from Lemma \ref{lem:NCdescent2}. 
\end{proof}
In an entirely similar way one has the following. 
\begin{lem}\label{lem:Dndescent}
With notations as above. Let $\{U_i \to X\}_{i=1}^s$ be an admissible covering of $X$ by $p^n$-accessible affinoid subdomains. Then:
\begin{enumerate}[(i)]
    \item The morphism $\mathcal{D}^n_X(X) \to \prod_{i=1}^s \mathcal{D}^n_X(U_i)$ is descendable in the sense of Definition \ref{defn:descendableNC}. 
    \item Let $Y := \coprod_{i=1}^s U_i$. Then the augmented simplicial object $\mathcal{D}^n_X(Y^{\bullet+1/X})$ satisfies the Beck-Chevalley condition of Definition 
    \ref{defn:Beckchevalley}. 
    \item The canonical morphism  \begin{equation}\operatorname{RMod}_{\mathcal{D}^n_X(X)}D(\mathsf{CBorn}_K) \to \underset{[m] \in \Delta}{\operatorname{lim}} \operatorname{RMod}_{\mathcal{D}^n_X(Y^{m+1/X})}D(\mathsf{CBorn}_K)
    \end{equation}
    is an equivalence of $\infty$-categories. In particular $\operatorname{RMod}_{\mathcal{D}^n_X(-)}D(\mathsf{CBorn}_K)$
    is a sheaf on $X_{n,w}$. 
\end{enumerate}
\end{lem}
\begin{proof}
This is the same, \emph{mutandis mutatis}, as the proof of Lemma \ref{lem:Dcapdescent}. 
\end{proof}
For each $n \geqslant 0$ we denote by
\begin{equation}
    \operatorname{RMod}^{\mathrm{b,fg}}_{\mathcal{D}^{n}_X(X)} D(\mathsf{CBorn}_K) \subseteq  \operatorname{RMod}_{\mathcal{D}^{n}_X(X)} D(\mathsf{CBorn}_K)
\end{equation}
the full subcategory spanned by (cohomologically) bounded complexes with finitely-generated cohomology groups.
\begin{lem}\label{lem:fgpullback}
Let $U \subseteq X$ be a $p^n$-accessible affinoid subdomain. Then:
\begin{enumerate}[(i)]
    \item Finitely-generated right $\mathcal{D}^n_X(X)$-modules are acyclic for $(-)\widehat{\otimes}_{\mathcal{D}^n_X(X)}\mathcal{D}^n_X(U)$.
    \item The pullback functor $(-)\widehat{\otimes}^\mathbf{L}_{\mathcal{D}^n_X(X)} \mathcal{D}^n_X(U)$ restricts to a functor 
\begin{equation}
\operatorname{RMod}^{\mathrm{b,fg}}_{\mathcal{D}^{n}_X(X)} D(\mathsf{CBorn}_K) \to \operatorname{RMod}^{\mathrm{b,fg}}_{\mathcal{D}^{n}_X(U)} D(\mathsf{CBorn}_K), 
\end{equation}
which furthermore is $t$-exact. In particular we obtain a sub-prestack
\begin{equation}
\operatorname{RMod}^{\mathrm{b,fg}}_{\mathcal{D}^{n}_X(-)} D(\mathsf{CBorn}_K) \subseteq \operatorname{RMod}_{\mathcal{D}^{n}_X(-)} D(\mathsf{CBorn}_K): X_w^\mathsf{op} \to \mathsf{Cat}_\infty. 
\end{equation}
\end{enumerate}
\end{lem}
\begin{proof}
(i): By \cite[Theorem 4.9]{Dcap1}, $\mathcal{D}^n_X(U)$ is flat on both sides as an abstract $\mathcal{D}^n_X(X)$-module. Hence the claim follows from \cite[Lemma 5.32]{bode_six_2021}, noting that the $\operatorname{Tor}$-groups in \emph{loc. cit.} refer to the abstract  $\operatorname{Tor}$-groups. 
 
 (ii): Let $M^\bullet$ be a bounded complex with finitely-generated cohomology groups. Using the result of (i), an easy spectral sequence argument implies that
\begin{equation}
    H^j(M^\bullet \widehat{\otimes}^\mathbf{L}_{\mathcal{D}^n_X(X)} \mathcal{D}^n_X(U)) \cong H^j (M^\bullet)\widehat{\otimes}_{\mathcal{D}^n_X(X)} \mathcal{D}^n_X(U),
\end{equation}
proving the Lemma. 
\end{proof}
\begin{thm}\label{thm:Fgbdescent}
Let $\{U_i \to X\}_{i=1}^s$ be an admissible covering of $X$ by $p^n$-accessible subdomains. Then the canonical morphism
\begin{equation}\label{eq:Bfgdescent}
    \operatorname{RMod}^{\mathrm{b,fg}}_{\mathcal{D}^n_X(X)}D(\mathsf{CBorn}_K) \to \underset{[m] \in \Delta}{\operatorname{lim}} \operatorname{RMod}^{\mathrm{b,fg}}_{\mathcal{D}^n_X(Y^{m+1/X})}D(\mathsf{CBorn}_K)
\end{equation}
is an equivalence of $\infty$-categories.
\end{thm}
\begin{proof}
Let $(M_m^\bullet)_{[m] \in \Delta}$ be an object of the right-hand side of \eqref{eq:Bfgdescent}. Because of Lemma \ref{lem:Dndescent}, the only thing to show is that $M_{-1}^\bullet := R\operatorname{lim}_{[m] \in \Delta} M_m^\bullet$ is bounded with finitely-generated cohomology groups. We consider the hypercohomology spectral sequence
\begin{equation}\label{eq:spectralsequenceFG}
E_2^{pq}:R^p\underset{[m] \in \Delta}{\operatorname{lim}} H^q M_m^\bullet \Rightarrow H^{p+q}(M_{-1}).
\end{equation}
This converges because $(M_m^\bullet)_{[m] \in \Delta}$ is uniformly cohomologically bounded by Lemma \ref{lem:fgpullback}. Further, Lemma \ref{lem:fgpullback} implies that 
\begin{equation}
    H^q(M_k^\bullet) \widehat{\otimes}_{\mathcal{D}^n_X(Y^{k+1/X})} \mathcal{D}^n_X(Y^{l+1/X}) \cong H^q(M_l^\bullet)
\end{equation}
is an isomorphism for every cosimplicial morphism $[l] \to [k]$. Thus the counterpart of Kiehl's theorem in this setting \cite[\S 5]{Dcap1} then implies that 
\begin{equation}
    R^p\underset{[m] \in \Delta}{\operatorname{lim}} H^q M_m^\bullet = 0,  \text{ whenever } p > 0. 
\end{equation}
Thus the spectral sequence \eqref{eq:spectralsequenceFG} collapses and gives an isomorphism
\begin{equation}
\begin{aligned}
    H^nM^\bullet_{-1} &\cong \underset{[m] \in \Delta}{\operatorname{lim}} H^n(M_m^\bullet)\\
    &= \operatorname{eq}\big(H^n(M_0^\bullet) \rightrightarrows H^n(M^\bullet_1)\big) 
    \end{aligned}
\end{equation}
which, by the theorem of descent for finitely-generated $\mathcal{D}^n_X$-modules \cite[\S 5]{Dcap1}, is a finitely generated $\mathcal{D}^n_X(X)$-module. Further, we see that the cohomological amplitude of $M_{-1}^\bullet$ is contained in the same interval as $M_0^\bullet$. 
\end{proof}
\begin{scholium}
Looking at the proof of Theorem \ref{thm:Fgbdescent}, we notice that in fact 
\begin{equation}
  \operatorname{RMod}^{[c,d], \mathrm{fg}}_{\mathcal{D}^n_X(X)}D(\mathsf{CBorn}_K) \to \underset{[m] \in \Delta}{\operatorname{lim}} \operatorname{RMod}^{[c,d],\mathrm{fg}}_{\mathcal{D}^n_X(Y^{m+1/X})}D(\mathsf{CBorn}_K)
\end{equation}
is an equivalence of $\infty$-categories, for any interval $[c,d] \subseteq \mathbf{R}$ with $d < \infty$.
\end{scholium}
We remark that since each $\mathcal{D}^n_X(X)$ is flat (on both sides) as an abstract $\mathcal{D}^{n+1}_X(X)$-module the pullback functors restrict\footnote{To be precise, this abstract flatness together with \cite[Lemma 5.32]{bode_six_2021} implies that finitely generated $\mathcal{D}^{n+1}_X(X)$-modules are acyclic for $-\widehat{\otimes}_{\mathcal{D}^{n+1}_X(X)} \mathcal{D}^n_X(X)$, and then an easy spectral-sequence argument gives the claim.} to functors
\begin{equation}
\operatorname{RMod}^{\mathrm{b,fg}}_{\mathcal{D}^{n+1}_X(X)}D(\mathsf{CBorn}_K) \to \operatorname{RMod}^{\mathrm{b,fg}}_{\mathcal{D}^{n}_X(X)} D(\mathsf{CBorn}_K)
\end{equation}
for each $n$. There is an obvious functor 
\begin{equation}
   \phi: D_\mathcal{C}(X) \to \underset{n}{\operatorname{lim}} \operatorname{RMod}^{\mathrm{b,fg}}_{\mathcal{D}^n_X(X)} D(\mathsf{CBorn}_K)
\end{equation}
which, on objects, sends $M^\bullet \mapsto (M_n^\bullet)_n$ where $M_n^\bullet = M^\bullet \widehat{\otimes}_{\wideparen{\mathcal{D}}_X(X)}^\mathbf{L} \mathcal{D}^n_X(X)$. 
\begin{thm}\label{thm:CComplexdescent}
The functor
\begin{equation}\label{eq:Ccomplexlimit}
   \phi: D_\mathcal{C}(X) \to \underset{n}{\operatorname{lim}} \operatorname{RMod}^{\mathrm{b,fg}}_{\mathcal{D}^n_X(X)} D(\mathsf{CBorn}_K)
\end{equation}
is an equivalence of $\infty$-categories. 
\end{thm}
\begin{proof}
We will first show that, given a Cartesian section $(N_n^\bullet)_n$ belonging to the left-side of \eqref{eq:Ccomplexlimit}, then $N^\bullet:= R\operatorname{lim}_n N_n^\bullet$ belongs to the full subcategory of $\mathcal{C}$-complexes, so that we obtain a right adjoint $\psi: (N_n^\bullet) \mapsto N^\bullet := R\operatorname{lim}_n N_n^\bullet$ to the functor $\phi$ above. (If we can show that $N^\bullet$ is a $\mathcal{C}$-complex, then it will immediately follow that $\phi$ is fully-faithful, as by definition we have $M^\bullet \simeq \psi \phi (M^\bullet)$ for any $\mathcal{C}$-complex $M^\bullet$). 

For each $j \in \mathbf{Z}$, the system $\{H^j(N_n^\bullet)\}_n$ satisfies $H^j(N_{n+1}^\bullet) \widehat{\otimes}_{\mathcal{D}^{n+1}_X(X)} \mathcal{D}^{n}_X(X) \xrightarrow[]{\sim} H^j(N_{n}^\bullet)$, by flatness\footnote{See previous footnote.} of $\mathcal{D}^{n}_X(X)$ as a (left) $\mathcal{D}^{n+1}_X(X)$-module. In particular the system $\{H^j(N_n^\bullet)\}_n$ is pre-nuclear with dense images (c.f. the remark under \cite[Definition 5.24]{bode_six_2021}). Hence the usual short-exact sequence 
\begin{equation}
    0 \to R^1 \underset{n}{\operatorname{lim}} H^{j-1}(N_n^\bullet) \to H^j(N^\bullet) \to \underset{n}{\operatorname{lim}} H^j(N_n^\bullet) \to 0
\end{equation}
together with the Mittag-Leffler result of \cite[Theorem 5.26]{bode_six_2021} implies that $H^j(N^\bullet) \cong \operatorname{lim}_nH^j(N^\bullet_n)$ has coadmissible cohomology. We claim that for each $m$, the canonical morphism
\begin{equation}\label{eq:canonicalNmap}
     N^\bullet \widehat{\otimes}^\mathbf{L}_{\wideparen{\mathcal{D}}_X(X)} \mathcal{D}^m_X(X)  \to N_m^\bullet
\end{equation}
is an equivalence. We may take cohomology. By \cite[Corollary 5.38]{bode_six_2021} coadmissible right $\wideparen{\mathcal{D}}_X(X)$-modules are acyclic for $(-) \widehat{\otimes}_{\wideparen{\mathcal{D}}_X(X)} \mathcal{D}^m_X(X)$. This implies that 
\begin{equation}
     H^j(N^\bullet \widehat{\otimes}^\mathbf{L}_{\wideparen{\mathcal{D}}_X(X)} \mathcal{D}^m_X(X)) \cong H^j(N^\bullet)\widehat{\otimes}_{\wideparen{\mathcal{D}}_X(X)} \mathcal{D}^m_X(X)
\end{equation}
for each $j \in \mathbf{Z}$: indeed, when $N^\bullet$ is bounded-above, this is an easy spectral-sequence argument, and in general one writes $N^\bullet$ as a (homotopy) colimit of its truncations, and uses that the derived tensor product commutes with colimits separately in each variable. However by properties of coadmissible modules we know that 
\begin{equation}
H^j(N^\bullet)\widehat{\otimes}_{\wideparen{\mathcal{D}}_X(X)} \mathcal{D}^m_X(X) \cong H^j(N_m). 
\end{equation}
This implies that \eqref{eq:canonicalNmap} is an equivalence. By the comment above this implies that $\phi$ is fully-faithful. In fact, looking at \eqref{eq:canonicalNmap} we see that $\phi\psi \simeq \operatorname{id}$, so that $(\phi, \psi)$ give an equivalence of categories.
\end{proof}
\begin{rmk}\label{rmk:isolate}
We isolate the following useful fact from the proof of Theorem \ref{thm:CComplexdescent}. Let $(N_n^\bullet)_n$ be a Cartesian section belonging to the right side of \eqref{eq:Ccomplexlimit}. Then for each $j \in \mathbf{Z}$, $N^\bullet := R \operatorname{lim}_n N_n^\bullet$ satisfies $H^j(N^\bullet) \xrightarrow[]{\sim} \operatorname{lim}_n H^j(N_n^\bullet)$. 
\end{rmk}
\begin{lem}\label{lem:D_ntransversellemma}
Let $M^\bullet \in D_\mathcal{C}(X)$ and let $n \geqslant 0$. Then for every $j \in \mathbf{Z}$ there is an isomorphism
\begin{equation}
    H^j(M^\bullet \widehat{\otimes}_{\wideparen{\mathcal{D}}_X(X)}^\mathbf{L} \mathcal{D}^n_X(X)) \cong H^j(M^\bullet)\widehat{\otimes}_{\wideparen{\mathcal{D}}_X(X)} \mathcal{D}^n_X(X).
\end{equation}
\end{lem}
\begin{proof}
By \cite[Corollary 5.38]{bode_six_2021}, coadmissible right $\wideparen{\mathcal{D}}_X(X)$-modules are acyclic for $(-)\widehat{\otimes}_{\wideparen{\mathcal{D}}_X(X)}^\mathbf{L} \mathcal{D}^n_X(X)$, which gives the claim (when $M^\bullet$ is bounded-above, this is an easy spectral-sequence argument, and in general one writes $M^\bullet$ as a homotopy colimit of its truncations and uses the commutation of tensor products with colimits). 
\end{proof}
\begin{lem}\label{lem:Coadmissiblepullback}
Let $M^\bullet \in D_\mathcal{C}(X)$ and let $U \subseteq X$ be an affinoid subdomain. Then:
\begin{enumerate}[(i)]
    \item $M_U^\bullet:= M^\bullet \widehat{\otimes}^\mathbf{L}_{\wideparen{\mathcal{D}}_X(X)} \wideparen{\mathcal{D}}_X(U)$ belongs to $D_\mathcal{C}(U)$.
    \item For each $j \in \mathbf{Z}$, one has $H^j(M_U^\bullet) \cong H^j(M^\bullet) \widehat{\otimes}_{\wideparen{\mathcal{D}}_X(X)}\wideparen{\mathcal{D}}_X(U)$.
\end{enumerate}
 In particular we obtain a sub-prestack
\begin{equation}
    D_\mathcal{C}(-)  \subseteq \operatorname{RMod}_{\wideparen{\mathcal{D}}_X(-)} D(\mathsf{CBorn}_K) : X_w^\mathsf{op} \to \mathsf{Cat}_\infty. 
\end{equation}
\end{lem}
\begin{proof}
In the following argument we always take $n$ large enough so that $U$ is $p^n$-accessible. By associativity properties of the tensor product and Lemma \ref{lem:fgpullback} we know that each $M_U^\bullet  \widehat{\otimes}^\mathbf{L}_{\wideparen{\mathcal{D}}_X(U)}\mathcal{D}^n_X(U)$ is bounded with finitely-generated cohomology. Now we show that the canonical morphism
\begin{equation}\label{eq:MUmorphism}
   M_U^\bullet \to R \underset{n}{\operatorname{lim}} M_U^\bullet  \widehat{\otimes}^\mathbf{L}_{\wideparen{\mathcal{D}}_X(U)}\mathcal{D}^n_X(U) 
\end{equation}
is an equivalence. Let us first examine the left side of \eqref{eq:MUmorphism}. By \cite[Proposition 5.37]{bode_six_2021}, coadmissible right $\wideparen{\mathcal{D}}_X(X)$-modules are acyclic for $(-)\widehat{\otimes}_{\wideparen{\mathcal{D}}_X(X)} \wideparen{\mathcal{D}}_X(U)$. This implies that 
\begin{equation}
    H^j(M_U^\bullet) \cong H^j(M^\bullet) \widehat{\otimes}_{\wideparen{\mathcal{D}}_X(X)} \wideparen{\mathcal{D}}_X(U)
\end{equation}
for each $j \in \mathbf{Z}$. Indeed, when $M^\bullet$ is bounded above this follows from an easy spectral-sequence argument, and in general one writes $M^\bullet$ as a colimit of its truncations and uses the commutation of tensor products with colimits.

Now let us examine the right side of \eqref{eq:MUmorphism}. Using Remark \ref{rmk:isolate} above, one has 
\begin{equation}
H^j(R\underset{n}{\operatorname{lim}}M_U^\bullet  \widehat{\otimes}^\mathbf{L}_{\wideparen{\mathcal{D}}_X(U)}\mathcal{D}^n_X(U)) \cong \underset{n}{\operatorname{lim}} H^j(M^\bullet \widehat{\otimes}^\mathbf{L}_{\wideparen{\mathcal{D}}_X(X)}  \mathcal{D}^n_X(U)).
\end{equation}
By Lemma \ref{lem:D_ntransversellemma}, we know that
\begin{equation}
    H^j(M^\bullet \widehat{\otimes}^\mathbf{L}_{\wideparen{\mathcal{D}}_X(X)}  \mathcal{D}^n_X(U)) \cong H^j(M^\bullet) \widehat{\otimes}_{\wideparen{\mathcal{D}}_X(X)}\mathcal{D}^n_X(U). 
\end{equation}
So, to show that \eqref{eq:MUmorphism} is an equivalence we are reduced to showing that
\begin{equation}
    H^j(M^\bullet) \widehat{\otimes}_{\wideparen{\mathcal{D}}_X(X)} \wideparen{\mathcal{D}}_X(U) \to \underset{n}{\operatorname{lim}} H^j(M^\bullet) \widehat{\otimes}_{\wideparen{\mathcal{D}}_X(X)}\mathcal{D}^n_X(U)
\end{equation}
is an isomorphism for each $j \in \mathbf{Z}$. This follows from \cite[Proposition 5.33]{bode_six_2021} (see also the preceding discussion about Ardakov--Wadsley's ``cap tensor product" in \emph{loc. cit.}). 
\end{proof}
\begin{cor}\label{cor:ccomplexsFr}
Suppose that $M^\bullet$ is a $\mathcal{C}$-complex. Then the underlying object $M^\bullet \in \operatorname{QCoh}(X)$ belongs to $\operatorname{sFr}(X)$, so that there is an inclusion
\begin{equation}
    D_\mathcal{C}(X) \subseteq \operatorname{RMod}_{\wideparen{\mathcal{D}}_X(X)} \operatorname{sFr}(X).
\end{equation}
\end{cor}
\begin{proof}
Let $U \subseteq X$ be an affinoid subdomain. Then by Lemma \ref{lem:Coadmissiblepullback},  $M^\bullet \widehat{\otimes}_{\wideparen{\mathcal{D}}_X(X)}^\mathbf{L} \wideparen{\mathcal{D}}_U(U)$ is again a $\mathcal{C}$-complex, whose underlying object in $\operatorname{QCoh}(X)$ is $M^\bullet\widehat{\otimes}_A^\mathbf{L}A_U$. Now the claim follows from Lemma \ref{lem:Ccomplex}. 
\end{proof}
Lemma \ref{lem:Coadmissiblepullback} also has the following consequence. By functoriality of pullbacks, we obtain a functor $X_{w}^\mathsf{op} \to \operatorname{Fun}(\Delta^1,\mathsf{Cat}_\infty)$ which sends $U \in X_{w}$ to the $1$-morphism
\begin{equation}
\operatorname{RMod}_{\wideparen{\mathcal{D}}_X(U)}D(\mathsf{CBorn}_K)\to   \underset{n}{\operatorname{lim}}\operatorname{RMod}_{\mathcal{D}^n_X(U)}D(\mathsf{CBorn}_K).
\end{equation}
The upshot of Lemma \ref{lem:Coadmissiblepullback} is that, by restriction of this functor, we obtain a functor $X^\mathsf{op}_{w} \to \operatorname{Fun}(\Delta^1, \mathsf{Cat}_\infty)$ which sends $U \in X_{w}$ to the (invertible) 1-morphism
\begin{equation}
    D_\mathcal{C}(U) \xrightarrow[]{\sim} \underset{n}{\operatorname{lim}} \operatorname{RMod}_{\mathcal{D}^n_X(U)}^{\mathrm{b,fg}}D(\mathsf{CBorn}_K). 
\end{equation}
We may use this implicitly in the following.
\begin{thm}\label{thm:CCOmplexanalytictopology}
Let $\{U_i \to X\}_{i=1}^s$ be an admissible covering of $X$ by affinoid subdomains. Let $Y:= \coprod_{i=1}^s U_i \to X$. Then the canonical morphism
\begin{equation}\label{eq:Ccomplexdescent}
    D_\mathcal{C}(X) \to \underset{[m] \in \Delta}{\operatorname{lim}} D_\mathcal{C}(Y^{m+1/X}) 
\end{equation}
is an equivalence of $\infty$-categories. 
\end{thm}
\begin{proof}
In the following argument we always take $n$ large enough so that all the $\{U_i\}_{i=1}^s$ are $p^n$-accessible. By Lemma \ref{lem:Coadmissiblepullback} the following square commutes:
\begin{equation}
\begin{tikzcd}[column sep=large]
	{\operatorname{RMod}_{\mathcal{D}^n_X(X)}^{\mathrm{b,fg}}D(\mathsf{CBorn}_K)} & {\operatorname{RMod}_{\mathcal{D}^n_X(Y^{m+1/X})}^{\mathrm{b,fg}}D(\mathsf{CBorn}_K)} \\
	{D_\mathcal{C}(X)} & {D_\mathcal{C}(Y^{m+1/X})}
	\arrow[from=1-1, to=1-2]
	\arrow[from=2-1, to=1-1]
	\arrow[from=2-1, to=2-2]
	\arrow[from=2-2, to=1-2]
\end{tikzcd}
\end{equation}
The bottom arrow is $-\widehat{\otimes}^\mathbf{L}_{\wideparen{\mathcal{D}}_X(X)}\wideparen{\mathcal{D}}_X(Y^{m+1/X})$, the right is $-\widehat{\otimes}^\mathbf{L}_{\wideparen{\mathcal{D}}_X(Y^{m+1/X})} \mathcal{D}^n_X(Y^{m+1/X})$, the top is $- \widehat{\otimes}^\mathbf{L}_{\mathcal{D}^n_X(X)} \mathcal{D}^n_X(Y^{m+1/X})$, and the left is $-\widehat{\otimes}^\mathbf{L}_{\wideparen{\mathcal{D}}_X(X)} \mathcal{D}^n_X(X)$. Passing to the limit over $n$, using Theorem \ref{thm:CComplexdescent}, and then taking the limit over $[m] \in \Delta$, we see that the morphism \eqref{eq:Ccomplexdescent} is equivalent to 
\begin{equation}
\begin{aligned}
\underset{n}{\operatorname{lim}} \operatorname{RMod}^{\mathrm{b,fg}}_{\mathcal{D}^n_X(X)} D(\mathsf{CBorn}_K) &\to \underset{[m] \in \Delta}{\operatorname{lim}} \underset{n}{\operatorname{lim}} \operatorname{RMod}^{\mathrm{b,fg}}_{\mathcal{D}^n_X(Y^{m+1/X})} D(\mathsf{CBorn}_K) \\
&\simeq \underset{n}{\operatorname{lim}} \underset{[m] \in \Delta}{\operatorname{lim}}\operatorname{RMod}^{\mathrm{b,fg}}_{\mathcal{D}^n_X(Y^{m+1/X})} D(\mathsf{CBorn}_K),
\end{aligned}
\end{equation}
where in the last line we used that limits commute with limits. Hence, the claim follows by taking limits over $n$ in Theorem \ref{thm:Fgbdescent}. 
\end{proof}
\begin{defn}
We define a pair of full subcategories
\begin{equation}
    (D_\mathcal{C}^{\leqslant 0}(X), D^{\geqslant 0}_\mathcal{C}(X)) 
\end{equation}
of $D_\mathcal{C}(X)$ by $M^\bullet \in D_\mathcal{C}^{\leqslant 0}(X)$ (resp. $M^\bullet \in D_\mathcal{C}^{\geqslant 0}(X)$) if $H^j(M^\bullet) = 0$ for all $j \geqslant 1$ (resp. if $H^j(M^\bullet) = 0$ for all $j \leqslant -1$). 
\end{defn}
\begin{lem}\label{lem:texact}
With notations as above. 
\begin{enumerate}[(i)]
    \item The pair $(D_\mathcal{C}^{\leqslant 0}(X), D^{\geqslant 0}_\mathcal{C}(X)) $ determines a $t$-structure on $D_\mathcal{C}(X)$. 
    \item For $U \subseteq X$ an affinoid subdomain, the pullback $- \widehat{\otimes}^\mathbf{L}_{\wideparen{\mathcal{D}}_X(X)}\wideparen{\mathcal{D}}_X(U) $ restricts to a $t$-exact functor $D_\mathcal{C}(X) \to D_\mathcal{C}(U)$. 
\end{enumerate}
\end{lem}
\begin{proof}
(i): We need to check that if $M^\bullet \in D_\mathcal{C}(X)$, then so does $\tau^{\leqslant 0}M^\bullet$ and $\tau^{\geqslant 0}M^\bullet$. Looking at Definition \ref{defn:Ccomplexprime}, we see that this follows immediately from Lemma \ref{lem:D_ntransversellemma} above.

(ii): This is Lemma \ref{lem:Coadmissiblepullback}.
\end{proof}
\begin{lem}\label{lem:intervaldescent}
Let $[c,d] \subseteq \mathbf{R}$ be any interval with $d < \infty$. Let $\{U_i \to X\}_{i=1}^n$ be an admissible covering of $X$ by affinoid subdomains. Let $Y = \coprod_{i=1}^nU_i$. Then the natural morphism
\begin{equation}\label{eq:boundeddescentCComplex}
    D^{[c,d]}_\mathcal{C}(X) \to \underset{[m] \in \Delta}{\operatorname{lim}} D^{[c,d]}_\mathcal{C}(Y^{m+1/X})
\end{equation}
is an equivalence of $\infty$-categories.
\end{lem}
\begin{proof}
Let $(M^\bullet_m)_{[m] \in \Delta}$ be an object of the right-side of \eqref{eq:boundeddescentCComplex}. Due to Theorem \ref{thm:CCOmplexanalytictopology}, and Lemma \ref{lem:texact} the only thing to prove is that $M^\bullet_{-1} := R \operatorname{lim}_{[m] \in \Delta} M^\bullet_m$ has cohomology in the interval $[c,d]$. One argues in essentially the same way as the proof of Theorem \ref{thm:Fgbdescent}. We consider the hypercohomology spectral sequence 
\begin{equation}\label{eq:spectralseqcoadmissible}
    E_2^{pq}: R^p \underset{[m] \in \Delta}{\operatorname{lim}} H^qM_m^\bullet \Rightarrow H^{p+q}(M_{-1}^\bullet).
\end{equation}
This converges because $(M^\bullet_m)_{[m] \in \Delta}$ is uniformly homologically bounded by Lemma \ref{lem:Coadmissiblepullback}. Lemma \ref{lem:Coadmissiblepullback} also implies that 
\begin{equation}
    H^q(M_k^\bullet) \widehat{\otimes}_{\wideparen{\mathcal{D}}_X(Y^{k+1/X})} \wideparen{\mathcal{D}}_X(Y^{l+1/X}) \cong H^q(M_l^\bullet)
\end{equation}
is an isomorphism for every cosimplicial morphism $[l] \to [k]$. Then the counterpart of Kiehl's theorem in this setting \cite[\S8]{Dcap1} implies that 
\begin{equation}
    R^p \underset{[m] \in \Delta}{\operatorname{lim}} H^q M_n^\bullet = 0, \text{ whenever }p>0. 
\end{equation}
Thus the spectral-sequence \eqref{eq:spectralseqcoadmissible} collapses and gives an isomorphism
\begin{equation}
    H^nM^\bullet_{-1} \cong \underset{[m] \in \Delta}{\operatorname{lim}} H^n M^\bullet_m = \operatorname{eq}(H^n(M_0^\bullet) \rightrightarrows H^n(M_1^\bullet)),
\end{equation}
which shows that the cohomological amplitude of $M^\bullet_{-1}$ is contained in the same interval as $M_0^\bullet$. 
\end{proof}
\section{Compatibility with restrictions}
Let $X= \operatorname{Sp}(A)$ be a classical affinoid equipped with an \'etale morphism $X \to \mathbf{D}^r_K$. Let $U \subseteq X$ be an affinoid subdomain. We recall from \cite[Example 4.77]{soor_six-functor_2024} that the square 
\begin{equation}
\begin{tikzcd}[cramped]
	U & {U_{\mathrm{str}}} \\
	X & {X_{\mathrm{str}}}
	\arrow[from=1-1, to=1-2]
	\arrow[from=1-1, to=2-1]
	\arrow["\ulcorner"{anchor=center, pos=0.125}, draw=none, from=1-1, to=2-2]
	\arrow[from=1-2, to=2-2]
	\arrow[from=2-1, to=2-2]
\end{tikzcd}
\end{equation}
is Cartesian in the category $\mathsf{PStk}$. We recall that $\operatorname{Strat}(-)$ and $\operatorname{QCoh}(-)$ are functorial with respect to upper-star pullbacks. Let $X_w$ be the poset of affinoid subdomains of $X$. Using the base-change equivalences from the six-functor formalism obtained in \cite[Theorem 4.67]{soor_six-functor_2024} we obtain a natural transformation of functors $X_w^\mathsf{op} \to \mathsf{Cat}_\infty$:
\begin{equation}\label{eq:qCohfibration}
p_{(-),!}: \operatorname{QCoh}(-) \to \operatorname{Strat}(-),
\end{equation}
whose component for each $U \in X_w$ is the lower-shriek pushforwards $p_{U,!}: \operatorname{QCoh}(U) \to \operatorname{Strat}(U)$ induced by $p_U: U \to U_{\mathrm{str}}$.
By passing to right adjoints pointwise we a \emph{lax natural transformation}:
\begin{equation}\label{eq:uppershirekLaxmorphism}
p^!_{(-)}: \operatorname{Strat}(-) \to \operatorname{QCoh}(-),
\end{equation}
whose component over $U \in X_w$ is the upper-shriek pullback $p_U^!$. This essentially means that the natural transformations in each square 
\begin{equation}
\begin{tikzcd}
	{\operatorname{Strat}(X)} & {\operatorname{QCoh}(X)} \\
	{\operatorname{Strat}(U)} & {\operatorname{QCoh}(U)}
	\arrow["{p_X^!}", from=1-1, to=1-2]
	\arrow["{t_{\mathrm{str}}^*}", from=1-1, to=2-1]
	\arrow[draw=none, from=1-2, to=2-1]
	\arrow[shorten >=5pt, Rightarrow, from=1-2, to=2-1]
	\arrow["{t^*}", from=1-2, to=2-2]
	\arrow["{p_U^!}", from=2-1, to=2-2]
\end{tikzcd}
\end{equation}
can be composed vertically in a natural way. Here $t: U \hookrightarrow X$ is the inclusion.
\begin{lem}\label{lem:UpperstarPullbackCompatibilityOfDeRhamDmoduleForgetfulFunctor!!!} With notations as above. The natural transformation 
\begin{equation}
t^* p_X^! \to p_U^! t_{\mathrm{str}}^*
\end{equation}
restricts to an equivalence on the full subcategory spanned by those $M \in \operatorname{Strat}(X)$ such that $p_X^!M$ belongs to $\operatorname{Fr(X)} \subseteq \operatorname{QCoh}(X)$. 
\end{lem}
Before proving this Lemma, let me say why it is relevant, first introducing the following Definition: 
\begin{defn}
We define $\operatorname{Strat}_{\operatorname{Fr}}(X) \subseteq \operatorname{Strat}(X)$ (resp. $\operatorname{Strat}_{\operatorname{sFr}}(X) \subseteq \operatorname{Strat}(X)$)  to be the full sub-$\infty$-category spanned by objects $M$ such that $p^!_XM $ belongs to $\operatorname{Fr}(X)$ (resp. $\operatorname{sFr}(X)$). 
\end{defn} The upshot of Lemma \ref{lem:UpperstarPullbackCompatibilityOfDeRhamDmoduleForgetfulFunctor!!!} is twofold:
\begin{cor}\label{Cor:CRYSPNFpullbackcompatible}
\begin{enumerate}[(i)]
    \item $\operatorname{Strat}_{\operatorname{sFr}}(-)$ forms a sub-prestack of $\operatorname{Strat}(-)$ on $X_w$, with respect to the upper-star functors.
    \item The lax natural transformation \eqref{eq:uppershirekLaxmorphism} restricts to a natural transformation
    \begin{equation}\label{eq:PNFfibration}
       p^!_{(-)}: \operatorname{Strat}_{\operatorname{sFr}}(-) \to \operatorname{sFr}(-).
    \end{equation}
    We emphasise that the lax structure here is in fact strong, so that \eqref{eq:PNFfibration} is a natural transformation of functors $X_w^\mathsf{op} \to \mathsf{Cat}_\infty$, in which both sides are viewed as prestacks via the \emph{upper-star} functors. 
\end{enumerate}
\end{cor}
\begin{proof}[Proof of Lemma \ref{lem:UpperstarPullbackCompatibilityOfDeRhamDmoduleForgetfulFunctor!!!}] 
We would like to show that, for $M \in \operatorname{Strat}_{\operatorname{Fr}}(X)$, and an affinoid subdomain $t: U \hookrightarrow X$, the natural morphism
\begin{equation}
    t^*p_X^! M \to p_U^! t^*_{\mathrm{str}}M
\end{equation}
is an equivalence. Let us first consider the case when $M = p_{X,!}N$ for some $N \in \operatorname{sFr}(X)$. (That such $M$ belongs to $\operatorname{Strat}_{\operatorname{sFr}}(X)$ is a consequence of Proposition \ref{prop:PNFpreserve}). Then, we are asking for 
\begin{equation}
  t^*\mathcal{D}^\infty_XN =  t^*p_{X}^!p_{X,!} N \to p_U^!t^*_{\mathrm{str}} p_{X,!}N \simeq p^!_Up_{U,!} t^*N = \mathcal{D}^\infty_Ut^*N
\end{equation}
to be an equivalence, where in the last part we used base-change. Because $N$ and $t^*N$ are a Frechet-strict complexes, by Theorem \ref{thm:PNFtensor} this will follow if $t^*\mathcal{D}_X^\infty1_X \simeq \mathcal{D}^\infty_U1_U$. But this can be deduced, for instance, from Proposition \ref{prop:algebraiso}, because $  \wideparen{\mathcal{D}}_X(X) \widehat{\otimes}^\mathbf{L}_A A_U \xrightarrow[]{\sim} \wideparen{\mathcal{D}}_U(U)$. 

Now let us consider the general case. By monadicity (Theorem \ref{thm:StratMonadicity}) of the adjunction $p_{X,!} \dashv p_X^!$, every $M \in \operatorname{Strat}_{\operatorname{Fr}}(X)$ may be expressed as the colimit of a $p_X^!$-split simplicial object:
\begin{equation}
    M \simeq \underset{[n] \in \Delta^\mathsf{op}}{\operatorname{colim}}p_{X,!}N_n, 
\end{equation}
where $N_n := (p_X^!p_{X,!})^np_X^!M$. Each $N_n$ belongs to $\operatorname{sFr}(X)$, by Proposition \ref{prop:PNFpreserve}. Using that $p_X^!$ (resp. $p_U^!$) commutes with geometric realizations of $p_X^!$-split  (resp. $p_U^!$-split) simplicial objects, and that the upper-star functors are colimit-preserving, we compute  
\begin{equation}
\begin{aligned}
    t^*p_{X}^!M &\simeq \underset{[n] \in \Delta^\mathsf{op}}{\operatorname{colim}} t^*p_{X}^!p_{X,!}N_n \\ &\simeq \underset{[n] \in \Delta^\mathsf{op}}{\operatorname{colim}} p_{U}^! t_{\mathrm{str}}^*p_{X,!}N_n  \simeq p_U^!t_{\mathrm{str}}^*M,
\end{aligned}
 \end{equation}
 where we used the previous case. 
\end{proof}
\begin{lem}\label{lem:restrictions}
With notations as above. Let $t: U \to X$ be the inclusion.
\begin{enumerate}[(i)]
  \item There are restriction functors
    \begin{equation}
\operatorname{Mod}_{\mathcal{D}^\infty_X} \operatorname{sFr}(X) \to \operatorname{Mod}_{\mathcal{D}^\infty_U} \operatorname{sFr}(U)
 \end{equation}
    induced by $t^*$ on $\operatorname{sFr}(X)$. These functors are natural in $U$, so that we obtain a prestack $\operatorname{Mod}_{\mathcal{D}^\infty_{(-)}} \operatorname{sFr}(-)$ on $X_w$. 
     \item The equivalences $\operatorname{Strat}_{\operatorname{sFr}}(X) \simeq \operatorname{Mod}_{\mathcal{D}^\infty_X}\operatorname{sFr}(X)$ is compatible with these restriction functors, and the upper-star pullback functors on $\operatorname{Strat}_{\operatorname{sFr}}$ (c.f. Corollary \ref{Cor:CRYSPNFpullbackcompatible}), so that we obtain an equivalence \begin{equation}
     \operatorname{Mod}_{\mathcal{D}^\infty_{(-)}} \operatorname{sFr}(-) \simeq \operatorname{Strat}_{\operatorname{sFr}}(-) 
     \end{equation}of prestacks on $X_w$.
     \end{enumerate}
\end{lem}
\begin{proof}
This follows by combining Corollary \ref{Cor:CRYSPNFpullbackcompatible}, Theorem \ref{thm:StratMonadicity} and Corollary \ref{cor:BarrBeckFamiliesCor}.
\end{proof}
\begin{rmk}\label{rmk:monadfunctor}
Before moving on we give a more concrete description of the restriction functors in (i). Due to Lemma \ref{lem:UpperstarPullbackCompatibilityOfDeRhamDmoduleForgetfulFunctor!!!} and base-change, for each $M \in \operatorname{sFr}(X)$ there is a canonical equivalence \begin{equation}\label{eq:sigmatranformation}
  \sigma: \mathcal{D}^\infty_U t^* \simeq  t^*\mathcal{D}^\infty_X 
\end{equation} 
of functors on $\operatorname{sFr}(X)$. Due to naturality of base-change, this satisfies the expected naturalities making $t^*$ into a \emph{monad functor} from $\mathcal{D}^\infty_U1_U$ to $\mathcal{D}^\infty_X 1_X$. Roughly speaking, this says that there are diagrams
\begin{equation}
\begin{tikzcd}
	{t^* } & {\mathcal{D}^\infty_X t^*} \\
	& {t^*\mathcal{D}^\infty_U}
	\arrow["{\eta_U t^*}", from=1-1, to=1-2]
	\arrow["{t^* \eta_X}"', from=1-1, to=2-2]
	\arrow["{\sigma }", from=1-2, to=2-2]
\end{tikzcd}
\end{equation}
and 
\begin{equation}
\begin{tikzcd}
	{\mathcal{D}^\infty_Xt^*\mathcal{D}^\infty_U} & {\mathcal{D}^\infty_X\mathcal{D}^\infty_Xt^*} & {\mathcal{D}^\infty_Xt^*} \\
	{t^*\mathcal{D}^\infty_U\mathcal{D}^\infty_U} && {t^*\mathcal{D}^\infty_U}
	\arrow["{\mathcal{D}^\infty_X\sigma}", from=1-1, to=1-2]
	\arrow["{\mu_X t^*}", from=1-2, to=1-3]
	\arrow["{\sigma \mathcal{D}^\infty_U}", from=2-1, to=1-1]
	\arrow["{t^*\mu_U}", from=2-1, to=2-3]
	\arrow["{\sigma }"', from=2-3, to=1-3]
\end{tikzcd}
\end{equation}
of functors on $\operatorname{sFr}(X)$, which are homotopy-commutative, together with various higher coherences. Then, if $M \in \operatorname{Mod}_{\mathcal{D}^\infty_X}\operatorname{sFr}(X)$, there is a canonical $\mathcal{D}^\infty_U$-module structure on $t^*M$ in which the action of $\mathcal{D}^\infty_U$ is given by the composite 
\begin{equation}
    \mathcal{D}^\infty_U t^*M \xrightarrow[]{\sigma} t^*\mathcal{D}^\infty_X M \xrightarrow[]{t^*\mathrm{act}} t^*M.
\end{equation}
\end{rmk}
\begin{lem}
    \begin{enumerate}[(i)]
        \item There are restriction functors 
     \begin{equation}
         \operatorname{RMod}_{\mathcal{D}^\infty_X1_X} \operatorname{QCoh}(X) \to \operatorname{RMod}_{\mathcal{D}^\infty_U1_U} \operatorname{QCoh}(U)
     \end{equation}
     induced by $t^*$ on $\operatorname{sFr}(X)$. These functors are natural in $U$, so that we obtain a prestack on $X_w$.
     \item The equivalence $\operatorname{RMod}_{\mathcal{D}^\infty_X1_X} \operatorname{sFr}(X) \simeq \operatorname{Mod}_{\mathcal{D}^\infty_X} \operatorname{sFr}(X)$ is compatible with these restrictions and those coming from (i), so that we obtain an equivalence 
     \begin{equation}
         \operatorname{RMod}_{\mathcal{D}^\infty_{(-)}1_{(-)}} \operatorname{sFr}(-) \simeq \operatorname{Mod}_{\mathcal{D}^\infty_{(-)}} \operatorname{sFr}(-)
     \end{equation}
     of prestacks on $X_w$. 
    \end{enumerate}
\end{lem}
\begin{proof}
    (i): We note that there is a canonical equivalence of $A$-$A_U$ bimodule objects 
    \begin{equation}
      \mathcal{D}^\infty_U 1_U  \simeq   \mathcal{D}^\infty_X 1_X \widehat{\otimes}_A^\mathbf{L} A_U.
    \end{equation}
    Thus for $M \in \operatorname{RMod}_{\mathcal{D}^\infty_X1_X}\operatorname{QCoh}(X)$, the object $t^*M = M \widehat{\otimes}_A^\mathbf{L} A_U$ obtains the canonical stucture of a $\mathcal{D}^\infty_U1_U$-module object via
    \begin{equation}
        ( M \widehat{\otimes}^\mathbf{L}_A A_U )  \widehat{\otimes}_{A_U}^\mathbf{L} \mathcal{D}^\infty_U1_U  \simeq  (M \widehat{\otimes}^\mathbf{L}_A\mathcal{D}^\infty_X 1_X ) \widehat{\otimes}^\mathbf{L}_A A_U  \to  M \widehat{\otimes}^\mathbf{L}_A  A_U,
    \end{equation}
    which gives the required restriction functors. A different way to say this is that the functor $t^* = (-) \widehat{\otimes}_A^\mathbf{L} A_U$, together with the equivalence 
    \begin{equation}\label{eq:monadfunctorprime}
        \sigma^\prime :   (-)\widehat{\otimes}^\mathbf{L}_A  A_U \widehat{\otimes}_{A_U}^\mathbf{L} \mathcal{D}^\infty_U1_U   \simeq(-)\widehat{\otimes}^\mathbf{L}_A  \mathcal{D}^\infty_X 1_X \widehat{\otimes}^\mathbf{L}_A A_U 
    \end{equation}
    gives a monad functor (in the sense of Remark \ref{rmk:monadfunctor}) from $(-)\widehat{\otimes}_{A_U}^\mathbf{L}\mathcal{D}^\infty_U1_U $ to $(-)\widehat{\otimes}^\mathbf{L}_A \mathcal{D}^\infty_X 1_X $. 
    
    (ii): We need to show that the natural transformation \eqref{eq:naturaltransformlema} is compatible with the monad functors $\sigma$ and $\sigma^\prime$ of \eqref{eq:sigmatranformation} and \eqref{eq:monadfunctorprime} respectively. Applying  Lemma \ref{lem:Vlinearcattheorem} (with $\mathscr{V} = D(\mathsf{CBorn}_K)$ and  $\mathcal{M} = D(\mathsf{CBorn}_K)$ and, in the notations of that Lemma), there is an adjunction 
    \begin{equation}\label{eq:FunLax3}
       \iota:  _{A}\operatorname{BMod}_B\mathscr{V} \leftrightarrows \operatorname{Fun}^\mathrm{lax}_\mathscr{V}(\operatorname{Mod}_A \mathscr{V}, \operatorname{Mod}_B\mathscr{V}) : \kappa. 
    \end{equation}
    in which the left adjoint $\iota$ is fully-faithful. The equivalence $\sigma: \mathcal{D}^\infty_Ut^* \simeq t^*\mathcal{D}^\infty_X$ of functors on $\operatorname{sFr}(X)$ comes from restricting the natural transformation $\tau: t^*\mathcal{D}^\infty_X \to \mathcal{D}^\infty_U t^*$ of functors on $\operatorname{QCoh}(X) = \operatorname{Mod}_A \mathscr{V}$. We may view $\tau$ as a morphism in the category on the the right side of \eqref{eq:FunLax3}. Now we note that the inverse of $\sigma^\prime$ agrees with (the restriction to $\operatorname{sFr}(X)$ of) $\iota \kappa (\tau)$: indeed, $\kappa \iota \kappa (\tau)$ identifies with $\mathcal{D}^\infty_X 1_X \widehat{\otimes}^\mathbf{L}_A A_U \xrightarrow[]{\sim} \mathcal{D}^\infty_U 1_U$. Hence, using that $\iota \kappa \to \operatorname{id}$ is a natural transformation we obtain the desired commutative square 
    \begin{equation}
\begin{tikzcd}
	{\mathcal{D}^\infty_U(-\widehat{\otimes}^\mathbf{L}_A A_U)} & {(\mathcal{D}^\infty_X(-))\widehat{\otimes}^\mathbf{L}_A A_U} \\
	{(-)\widehat{\otimes}^\mathbf{L}_A A_U \widehat{\otimes}^\mathbf{L}_{A_U} \mathcal{D}^\infty_U 1_U} & {(-)\widehat{\otimes}^\mathbf{L}_A \mathcal{D}^\infty_X1_X \widehat{\otimes}^\mathbf{L}_A A_U}
	\arrow["{\simeq \sigma }", no head, from=1-1, to=1-2]
	\arrow[from=2-1, to=1-1]
	\arrow["{\simeq \sigma^\prime}", no head, from=2-1, to=2-2]
	\arrow[from=2-2, to=1-2]
\end{tikzcd}
    \end{equation}
    of endofunctors on $\operatorname{sFr}(X)$ in which the vertical arrows are induced by \eqref{eq:naturaltransformlema}. 
\end{proof}
It is not hard to see that the equivalence 
\begin{equation}\label{eq:nothardtosee}
\operatorname{RMod}_{\mathcal{D}^\infty_X1_X} \operatorname{QCoh}(X) \simeq \operatorname{RMod}_{\wideparen{\mathcal{D}}_X(X)} D(\mathsf{CBorn}_K)
\end{equation}
is compatible with restrictions to $U \in X_w$, so that there is an equivalence of prestacks 
\begin{equation}\label{eq:equivalence}
\operatorname{RMod}_{\mathcal{D}^\infty_{(-)}1_{(-)}} \operatorname{QCoh}(-) \simeq \operatorname{RMod}_{\wideparen{\mathcal{D}}_{(-)}(-)} D(\mathsf{CBorn}_K)
\end{equation}
on $X_w$. We recall from Definition \ref{defn:Ccomplexprime} that the $\infty$-category $D_\mathcal{C}(X)$ is defined as a certain full subcategory of $\operatorname{RMod}_{\wideparen{\mathcal{D}}_X(X)} D(\mathsf{CBorn}_K)$. Under the equivalence \eqref{eq:nothardtosee}, the image of $D_\mathcal{C}(X)$ factors through the full subcategory $\operatorname{RMod}_{\mathcal{D}^\infty_{X}1_{X}} \operatorname{sFr}(X)$, by Corollary \ref{cor:ccomplexsFr}. In particular we obtain a fully-faithful functor 
\begin{equation}
    D_\mathcal{C}(X) \hookrightarrow \operatorname{RMod}_{\mathcal{D}^\infty_{X}1_{X}} \operatorname{sFr}(X)
\end{equation}
which is compatible with restrictions. Now combining all the assertions of Lemma \ref{lem:restrictions} we obtain the following. 
\begin{thm}\label{thm:CcomplexFFaffine}
The functor $D_\mathcal{C}(X) \to \operatorname{Strat}(X)$ is compatible with the restrictions induced by the upper-star pullback functors on $\operatorname{Strat}$. That is, there is a morphism
\begin{equation}
    D_\mathcal{C}(-) \to \operatorname{Strat}(-)
\end{equation}
of prestacks on $X_w$, which is pointwise fully-faithful. 
\end{thm}

\section{Globalising the embedding of $\mathcal{C}$-complexes}\label{sec:global}
In this subsection $X$ now denotes an arbitrary smooth rigid analytic variety. Let $X_w(\mathcal{T})$ denote the poset of affinoid subdomains of $X$ which are \'etale over a polydisk (this is a basis for the weak topology). There is obviously a functor
\begin{equation}
    X_w(\mathcal{T}) \to X_{\mathrm{strong}}
\end{equation}
where the latter is the poset of all admissible open subsets of $X$ equipped with the strong G-topology. 
\begin{defn}\label{defn:CComplexstackification}
The stack of $\mathcal{C}$-complexes on $X$ is defined to be the right Kan extension of the functor $D_\mathcal{C}(-) : X_w^\mathsf{op}(\mathcal{T}) \to \mathsf{Cat}_\infty$ along $X_w^\mathsf{op}(\mathcal{T}) \to X_{\mathrm{strong}}$. 
\end{defn}
With this definition, one has (for an arbitrary rigid smooth variety $X$):
\begin{equation}
    D_\mathcal{C}(X) = \varprojlim_{U} D_\mathcal{C}(U), 
\end{equation}
where the limit runs over all affinoid subdomains $U \subseteq X$ which are \'etale over a polydisk. 
\begin{thm}
With notations as above:
\begin{enumerate}[(i)]
    \item $D_\mathcal{C}(-)$ is a sheaf of $\infty$-categories on $X_{\mathrm{strong}}$. 
    \item There is a fully-faithful functor $D_\mathcal{C}(X) \hookrightarrow \operatorname{Strat}(X)$, which is compatible with restrictions to admissible open subsets $U \subseteq X$, so that we obtain a morphism
    \begin{equation}\label{eq:DCstrong}
        D_\mathcal{C}(-) \to \operatorname{Strat}(-)
    \end{equation}
    of $\mathsf{Cat}_\infty$-valued sheaves on $X_{\mathrm{strong}}$ which is pointwise fully-faithful.
\end{enumerate}
\end{thm}
\begin{proof}
(i): Because $X_w(\mathcal{T})$ is a basis for $X_\mathrm{strong}$, this follows from \cite[Proposition A.3.11(ii)]{mann_p-adic_2022} and Theorem \ref{thm:CComplexdescent}. 

(ii): By \cite[Corollary 4.85]{soor_six-functor_2024}, we know that for a smooth (classical) rigid variety $X$, that
\begin{equation}
    \operatorname{Strat}(X) \simeq \varprojlim_{U} \operatorname{Strat}(U),
\end{equation}
where the limit runs over all affinoid subdomains which are \'etale over a polydisk. In particular the morphism \eqref{eq:DCstrong} may be constructed using Theorem \ref{thm:CcomplexFFaffine} and Kan extension. The fully-faithfulness also follows from Theorem \ref{thm:CcomplexFFaffine}, together with the fact that the mapping space in a limit of $\infty$-categories, is the limit of the mapping spaces.  
\end{proof}
By Lemma \ref{lem:intervaldescent}, we obtain a subsheaf $D_\mathcal{C}^{\heartsuit}(-)$ of $D_\mathcal{C}(-)$ on $X_w(\mathcal{T})$. By Kan extension, we obtain a subsheaf $D_\mathcal{C}^{\heartsuit}(-)$ of $D_\mathcal{C}(-)$ on $X_{\mathrm{strong}}$.
\begin{thm}
Let $X$ be a smooth classical rigid-analytic space. The category $D^\heartsuit_\mathcal{C}(X)$ is equivalent to the category of coadmissible $\wideparen{\mathcal{D}}_X$-modules of \cite[\S 9.4]{Dcap1}. 
\end{thm}
\begin{proof}
By construction, the abelian category $D_\mathcal{C}(X)^\heartsuit = D_\mathcal{C}^{\heartsuit}(X)$ satisfies
\begin{equation}
    D_\mathcal{C}^{\heartsuit}(X) \simeq \varprojlim_{U \subseteq X} D^\heartsuit_\mathcal{C}(U),
\end{equation}
where the limit runs over all affinoid subdomains $U \subseteq X$ which are \'etale over a polydisk. It is clear that each $D^\heartsuit_\mathcal{C}(U)$ identifies with the category of coadmissible $\wideparen{\mathcal{D}}_U(U)$-modules. The limit on the left is the (2,1)-limit in the sense of ordinary category theory. To be precise, the $(2,1)$-limit is defined to be the category of Cartesian sections of the Grothendieck fibration corresponding to the $\mathsf{Cat}$-valued presheaf $D^\heartsuit_\mathcal{C}(-)$. In concrete terms, its objects are collections $(M_U)_U$ of objects in each category equipped with the data of equivalences $\phi_{UV}: M_U \widehat{\otimes}_{\wideparen{\mathcal{D}}_U(U)}\wideparen{\mathcal{D}}_V(V) \xrightarrow[]{\sim} M_V$ for every affinoid subdomain $V \subseteq U$, which satisfy an obvious cocycle condition. Thus, using \cite[\S 8]{Dcap1}, we are reduced to proving that
\begin{equation}\label{eq:equivalenceCoad}
\left\{ \text{coadmissible }\mathcal{D}_X\text{-modules} \right\} \simeq \varprojlim_U\left\{ \text{coadmissible }\mathcal{D}_U\text{-modules} \right\}
\end{equation}
where the limit on the right is the $(2,1)$-limit. Given \cite[Theorem 9.4]{Dcap1}, this is tautological: the functor from left to right is given by restriction and the functor from right to left glues a sheaf of $\mathcal{D}_X$-modules from local data, and the equivalence \eqref{eq:equivalenceCoad} expresses the fact that being coadmissible is local on $X_w(\mathcal{T})$. 
\end{proof}
\begin{cor}
Let $X$ be a smooth rigid-analytic variety. There is a fully-faithful functor
\begin{equation}
\left\{ \text{coadmissible }\mathcal{D}_X\text{-modules} \right\} \hookrightarrow \operatorname{Strat}(X).
\end{equation}
\end{cor}
\section{On the essential image of $\mathcal{C}$-complexes}\label{sec:essentialimage}
In this section we investigate the essential image of the functor $D_\mathcal{C}(X) \hookrightarrow \operatorname{Strat}(X)$, for $X$ a smooth rigid-analytic space. 

To explain this further, we recall that there is an ``ambient" six-functor formalism $\operatorname{QCoh}$ on $\mathsf{PStk}$, c.f. \cite[Theorem 4.67]{soor_six-functor_2024}. For the purposes of this section we assume that the structure morphism $f: X_{\mathrm{str}} \to *$ belongs to the class $\widetilde{E}$ of $!$-able morphisms in this six-functor formalism. This holds if, for instance, $X = \operatorname{dSp}(A)$ is a classical affinoid which is \'etale over a polydisk, by \cite[Theorem 4.101(i)]{soor_six-functor_2024} together with the fact that the class $\widetilde{E}$ is $!$-local on the source.

By considering the object $X_{\mathrm{str}} \in \mathsf{PStk}$ and the structure morphism $f: X_{\mathrm{str}} \to *$ we arrive at three natural finiteness conditions for objects in $\operatorname{Strat}(X) = \operatorname{QCoh}(X_{\mathrm{str}})$, as in \cite[\S4.4]{heyer_6-functor_2024}:
\begin{itemize}
    \item[$\star$] One can consider the \emph{$f$-suave objects} in $\operatorname{Strat}(X) = \operatorname{QCoh}(X_{\mathrm{str}})$, c.f. \cite[Definition 4.4.1(a)]{heyer_6-functor_2024}.
    \item[$\star$] One can consider the \emph{$f$-prim objects} in $\operatorname{Strat}(X) = \operatorname{QCoh}(X_{\mathrm{str}})$, c.f. \cite[Definition 4.4.1(b)]{heyer_6-functor_2024}.
    \item[$\star$] One can consider \emph{dualizable}\footnote{We recall that an object $A$ of a symmetric monoidal $\infty$-category $(\mathscr{V}, \otimes)$ is called \emph{dualizable} if there exists another object $A^\lor \in \mathscr{V}$ such that $A \otimes (-)$  is adjoint to $A^\lor \otimes(-)$.} objects in $\operatorname{Strat}(X)$. (By \cite[Example 4.4.3]{heyer_6-functor_2024} this coincides with the \emph{$\operatorname{id}$-suave} and \emph{$\operatorname{id}$-prim} objects). 
\end{itemize}
We will show that the essential image of $D_\mathcal{C}(X) \hookrightarrow \operatorname{Strat}(X)$ is \emph{not} contained in the $f$-prim or dualizable objects in $\operatorname{Strat}(X)$. From Proposition \ref{prop:suavesketch} below, it seems most likely that, if $D_\mathcal{C}(X)$ admits some intrinsic characterization, it would be as the $f$-suave objects in $\operatorname{Strat}(X)$, but we are unfortunately unable to prove this.

Let us first show that the essential image is not contained in the dualizable objects.
\begin{prop}\label{prop:notdualizable}
Let $X = \operatorname{dSp}(K \langle x \rangle)$ be the closed unit disk. Then the image of $\wideparen{\mathcal{D}}_X(X)$ under the functor $D_\mathcal{C}(X) \to \operatorname{Strat}(X)$ is not a dualizable object. 
\end{prop}
\begin{proof}
Set $A := K \langle x \rangle$. Recall the canonical morphism $p: X \to X_{\mathrm{str}}$. The image of $\wideparen{\mathcal{D}}_X(X)$ under the functor $D_\mathcal{C}(X) \to \operatorname{Strat}(X)$ is the object $p_*1_X \in \operatorname{Strat}(X) = \operatorname{QCoh}(X_{\mathrm{str}})$. Suppose for a contradiction that this object is dualizable. The functor $p^*$ is symmetric monoidal, hence it preserves dualizable objects. So 
\begin{equation}
    p^*p_*1_X  \simeq (A \widehat{\otimes}_K A)^\dagger_\Delta \simeq A \widehat{\otimes}_K K \langle t/p^\infty \rangle 
\end{equation}
would be a dualizable object of\footnote{Here $A \widehat{\otimes}_K K \langle t/p^\infty \rangle$ is viewed as an $A$-module by the action on $A$ only.} $\operatorname{QCoh}(X)$. We claim it isn't: after simplification using associativity of $\widehat{\otimes}^\mathbf{L}$, this reduces to showing that the canonical morphism 
\begin{equation}
R\underline{\operatorname{Hom}}_K(K\langle u/p^\infty \rangle, A) \widehat{\otimes}_K^\mathbf{L} K \langle t/p^\infty \rangle  \to R\underline{\operatorname{Hom}}_K(K \langle u/p^\infty \rangle , A \widehat{\otimes}_K K \langle t /p^\infty \rangle),
\end{equation}
is not an equivalence. Taking zeroth cohomology\footnote{In fact one can show that both sides are concentrated in degree $0$, but this is not necessary for the argument.}, and then arguing using cofinality, it is sufficient to show that
\begin{equation}\label{eq:sufficienttoshow}
    \underset{m}{\operatorname{colim}} \, \underset{n}{\operatorname{lim}}  A \widehat{\otimes}_KK \langle p^n s, t/p^m\rangle \to \underset{n}{\operatorname{lim}} \, \underset{m}{\operatorname{colim}} A \widehat{\otimes}_KK \langle p^n s, t/p^m\rangle,
\end{equation}
is not an equivalence; here $s$ is dual to $u$. We can explicitly describe both sides of \eqref{eq:sufficienttoshow}. The left side is:
\begin{equation}\label{eq:seriesleft}
 \left\{\sum_{k,l} c_{kl} s^k t^l: c_{kl} \in A, \exists m \forall n \ \|c_{kl}\| p^{nk-ml} \to 0 \text{ as }k,l \to \infty\right\}, 
\end{equation}
whereas the right side has the order of quantifiers reversed:  
\begin{equation}\label{eq:seriesright}
     \left\{\sum_{k,l} c_{kl} s^k t^l: c_{kl} \in A, \forall n \exists m  \ \|c_{kl}\| p^{nk-ml} \to 0 \text{ as }k,l \to \infty\right\}, 
\end{equation}
and so we can exhibit an element which belongs to \eqref{eq:seriesright} but not \eqref{eq:seriesleft}, for instance $\sum_{k,l} s^k t^l$. 
\end{proof}
Now we prove that the essential image of $D_\mathcal{C}(X) \hookrightarrow \operatorname{Strat}(X)$ is not contained in the $f$-prim objects. The proof is an instance of the ansatz ``compact + nuclear = dualizable", together with the relation between $f$-prim and compact objects \cite[Lemma 4.4.18]{heyer_6-functor_2024}.
\begin{prop}\label{prop:notdualizable}
Let $X = \operatorname{dSp}(K \langle x \rangle)$ be the closed unit disk. Then the image of $\wideparen{\mathcal{D}}_X(X)$ under the functor $D_\mathcal{C}(X) \to \operatorname{Strat}(X)$ is not an $f$-prim object.
\end{prop}
\begin{proof}
The unit object in $\operatorname{QCoh}(*) = D(\mathsf{CBorn}_K)$ is compact. Hence by \cite[Lemma 4.4.18]{heyer_6-functor_2024} every $f$-prim object in $\operatorname{Strat}(X)$ is compact. 

We recall again that the image of $\wideparen{\mathcal{D}}_X(X)$ under the functor $D_\mathcal{C}(X) \to \operatorname{Strat}(X)$ is the object $p_*1_X \in \operatorname{Strat}(X)$. Suppose that $p_*1_X$ is compact. Because $p_! \simeq p_*$, the functor $p_*$ preserves all colimits, and this formally implies that the left adjoint $p^*$ preserves compact objects. So $p^*p_*1_X \simeq (A \widehat{\otimes}_KA)^\dagger_\Delta \simeq A \widehat{\otimes}_K K \langle t /p^\infty \rangle$ would be a compact object of $\operatorname{QCoh}(X)$. We claim it isn't.

We know from, for instance, Lemma \ref{lem:frechetiso}, that if $V$ is a $K$-Banach space then the canonical morphism 
\begin{equation}
R\underline{\operatorname{Hom}}_A(A \widehat{\otimes}_KK\langle t/p^\infty \rangle, A) \widehat{\otimes}_A^\mathbf{L} (A \widehat{\otimes}_KV)  \to R\underline{\operatorname{Hom}}_A(A \widehat{\otimes}_KK \langle t/p^\infty \rangle , A \widehat{\otimes}_KV),
\end{equation}
is an equivalence. Now $R \underline{\operatorname{Hom}}_A(A \widehat{\otimes}_K K \langle t /p^\infty \rangle, -)$ commutes with finite colimits (since we are working with stable $\infty$-categories), and if $A \widehat{\otimes}_K K \langle t /p^\infty\rangle$ were compact, the functor $R \underline{\operatorname{Hom}}_A(A \widehat{\otimes}_K K \langle t /p^\infty \rangle, -)$ would also commute with filtered colimits by Proposition \ref{prop:RHomfilteredColimits}(iii). Hence using the compact generation and the explicit description of compact objects as finite colimits of Banach spaces (c.f. Proposition \ref{prop:compactgenprops}(vii) and Proposition \ref{prop:CbornKproperties}(ii)), we see that, if $A \widehat{\otimes}_K K \langle t /p^\infty\rangle$ were compact then the canonical morphism
\begin{equation}
R\underline{\operatorname{Hom}}_A(A \widehat{\otimes}_KK\langle t/p^\infty \rangle, A) \widehat{\otimes}_A^\mathbf{L} M^\bullet  \to R\underline{\operatorname{Hom}}_A(A \widehat{\otimes}_KK \langle t/p^\infty \rangle , M^\bullet),
\end{equation}
would be an equivalence for every $M^\bullet \in \operatorname{QCoh}(X)$, which is to say that $A \widehat{\otimes}_K K \langle t /p^\infty \rangle$ is dualizable. We already proved that this is false in Proposition \ref{prop:notdualizable}. 
\end{proof}
Let $X$ be a smooth rigid space. One can consider the full subcategory \begin{equation}
\operatorname{Perf}(\wideparen{\mathcal{D}}_X(X)) \subseteq D_\mathcal{C}(X)
\end{equation}
generated under finite (co)limits, shifts and retracts by the object $\wideparen{\mathcal{D}}_X(X)$. We give a sketch argument for the below Proposition. It is possibly an indication that $f$-suave objects are the most likely characterization of $\mathcal{C}$-complexes. 
\begin{prop}\label{prop:suavesketch}
Let $X = \mathring{\mathbf{D}}_K^n$ be the rigid open unit polydisk. The essential image of $\operatorname{Perf}(\wideparen{\mathcal{D}}_X(X)) \subseteq D_\mathcal{C}(X)$ under the functor $D_\mathcal{C}(X) \hookrightarrow \operatorname{Strat}(X)$ is contained in the $f$-suave objects. 
\end{prop}
\begin{proof}[Proof sketch] We expect that the morphism $g: X = \mathring{\mathbf{D}}_K^n \to *$ is \emph{suave} in the sense of \cite[\S4.5]{heyer_6-functor_2024}. All this means is that the unit object $1_X \in \operatorname{QCoh}(X)$ is $g$-suave. The proof of this should be essentially the same as in the dagger analytic setting, proved in \cite{soor_quasicoherent_2023} and details will be given in a future work. Now the morphism $p: X \to X_{\mathrm{str}}$ is \emph{prim} (this is essentially the content of the equivalence $p_! \simeq p_*$) and hence by \cite[Lemma 4.5.16(ii)]{heyer_6-functor_2024} suave objects are stable under the functor $p_!$. So $p_!1_X$ is $f$-suave, and this is the image of $\wideparen{\mathcal{D}}_X(X)$ under the functor $D_\mathcal{C}(X) \hookrightarrow \operatorname{Strat}(X)$. Since suave objects are stable under finite (co)limits and retracts \cite[Corollary 4.4.13]{heyer_6-functor_2024}, and the functor $D_\mathcal{C}(X) \hookrightarrow \operatorname{Strat}(X)$ is compatible with finite (co)limits and retracts, we conclude. 
\end{proof}
\appendix
\section{A way to produce monoids from monads}\label{sec:monoidsfromMonads}
For an operad $\mathscr{O}$ we let $\mathsf{CatMon}^\mathrm{lax}_{\mathscr{O}}$ denote the $(\infty,2)$-category of $\mathscr{O}$-monoidal categories with lax $\mathscr{O}$-linear functors \cite[Definition 3.4.1]{HausgengLax} and $\mathsf{CatMon}^\mathrm{oplax}_{\mathscr{O}}$ denote the $(\infty,2)$-category of $\mathscr{O}$-monoidal categories with oplax $\mathscr{O}$-linear functors \cite[Definition 3.4.3]{HausgengLax}.

Let $\mathsf{Assoc}$ denote the associative operad \cite[Definition 4.1.1.3]{HigherAlgebra} and let $\mathsf{LM}$ denote the operad of \cite[Notation 4.2.1.6]{HigherAlgebra}. Let $\mathscr{V}$ be a monoidal $(\infty,1)$-category which we identify with the object $\mathscr{V} \to \mathsf{Assoc}$ of $\mathsf{CatMon}^\mathrm{lax}_{\mathsf{Assoc}}$.  By abuse of notation we also identify $\mathscr{V}$ with the object $\mathscr{V}^\mathsf{op
} \to \mathsf{Assoc}^\mathsf{op}$ of $\mathsf{CatMon}^\mathrm{oplax}_{\mathsf{Assoc}}$.
\begin{defn}
\begin{enumerate}[(i)]
    \item The $(\infty,2)$-category of $\mathscr{V}$-linear categories with \emph{lax} $\mathscr{V}$-linear functors is defined to be 
\begin{equation}
\operatorname{LMod}_\mathscr{V}^\mathrm{lax} := \{\mathscr{V}\} \times_{\mathsf{CatMon}^\mathrm{lax}_{\mathsf{Assoc}}} \mathsf{CatMon}^\mathrm{lax}_{\mathsf{LM}}.
\end{equation}
    \item The $(\infty,2)$-category of $\mathscr{V}$-linear categories with \emph{oplax} $\mathscr{V}$-linear functors is defined to be 
\begin{equation}
\operatorname{LMod}_\mathscr{V}^\mathrm{oplax} := \{\mathscr{V}\} \times_{\mathsf{CatMon}^\mathrm{oplax}_{\mathsf{Assoc}}} \mathsf{CatMon}^\mathrm{oplax}_{\mathsf{LM}}.
\end{equation}
\end{enumerate}
\end{defn}
Here is an intuitive description. A \emph{$\mathscr{V}$-linear $\infty$-category} is an $\infty$-category $\mathscr{C}$ equipped with the data of a functor $ \otimes : \mathscr{V} \times \mathscr{C} \to \mathscr{C}$ which is unital and associative up to coherent homotopy. A \emph{lax $\mathscr{V}$-linear functor} $\mathscr{C} \to \mathscr{D}$ of $\mathscr{V}$-linear $\infty$-categories $\mathscr{C}, \mathscr{D}$ is a functor $F: \mathscr{C} \to \mathscr{D}$ equipped with the data of morphisms $V\otimes F(M) \to F(V \otimes M)$, for all $V \in \mathscr{V}, M \in \mathscr{C}$, which is unital and associative up to coherent homotopy. Oplax linear functors are defined similarly but with arrows reversed.
\begin{thm}\cite[Theorem C]{HausgengLax}\label{HaugsengLaxTheoremC}
Let $\operatorname{LMod}_\mathscr{V}^{R,\mathrm{lax}}$  denote the $1$-full $2$-subcategory of $\operatorname{LMod}_\mathscr{V}^{\mathrm{lax}}$ spanned by those functors which are objectwise right adjoints. Let $\operatorname{LMod}_\mathscr{V}^{L,\mathrm{oplax}}$  denote the $1$-full subcategory of $\operatorname{LMod}_\mathscr{V}^{\mathrm{oplax}}$ spanned by those functors which are objectwise left adjoints. Then there is an equivalence of $(\infty,2)$-categories 
\begin{equation}
\operatorname{LMod}_\mathscr{V}^{R,\mathrm{lax}} \simeq \big(\operatorname{LMod}_\mathscr{V}^{L,\mathrm{oplax}}\big)^{(1,2)-\mathsf{op}},
\end{equation}
obtained by passing to adjoints objectwise (in both directions). Here  $(1,2)-\mathsf{op}$ indicates that the direction of $1$- and $2$-morphisms are reversed. 
\end{thm}
\begin{proof}
This follows immediately from \cite[Theorem C]{HausgengLax}, taking the operad $\mathscr{O}$ in \emph{loc. cit.} to be $\mathsf{LM}$ and $\mathsf{Assoc}$ and then taking the fiber over $\mathscr{V}$\footnote{I am grateful to Shay Ben-Moshe for explaining this to me.}. 
\end{proof}
Intuitively, this may be explained as follows. Let $\mathscr{C}$ and $\mathscr{D}$ be $\mathscr{V}$-linear categories. Let $F : \mathscr{D} \leftrightarrows \mathscr{C} :G$ be an adjunction in which the left adjoint $F$ is oplax linear. Let $V \in \mathscr{V}, M \in \mathscr{C}$. The desired morphism $V \otimes G(M) \to G(V \otimes M)$ giving the lax linear structure on $G$ is adjoint to the composite
\begin{equation}
    F(V \otimes G(M)) \to V \otimes FG(M) \to V \otimes M,
\end{equation}
where the first is from the oplax linear structure on $F$ and the second is induced by the counit of $F \dashv G$. Similarly, a lax linear structure on $G$ determines an oplax linear structure on $F$. The Theorem says that these operations give a mutually inverse equivalence of $(\infty,2)$-categories.

\begin{lem}\label{lem:LaxLinV}
Let us view $\mathscr{V}$ as a $\mathscr{V}$-linear category. Then there is an adjunction of $\infty$-categories 
\begin{equation}
    \iota: \mathscr{V} \leftrightarrows \operatorname{Fun}^{\mathrm{lax}}_\mathscr{V}(\mathscr{V},\mathscr{V}) : \kappa
\end{equation}
in which the left adjoint $\iota$ sends $V \mapsto -\otimes V$ and the right adjoint $\kappa$ sends $F \mapsto F(1_\mathscr{V})$. (Note here that the superscript \emph{lax} stands for \emph{lax $\mathscr{V}$-linear} functors and \emph{not} lax monoidal functors).
\end{lem}
\begin{proof}
The unit morphism $\eta : \operatorname{id} \xrightarrow[]{\sim} \kappa \iota$ is the canonical equivalence $\operatorname{id} \simeq   1_\mathscr{V} \otimes\operatorname{id}$. The counit morphism $\varepsilon: \iota \kappa \to \operatorname{id}$ is deduced from the lax-linear structure $(-)\otimes F(1_\mathscr{V})  \to F(- \otimes 1_\mathscr{V}) \simeq F(-)$. We check the zig-zag identities. The composite $\kappa \to \kappa\iota\kappa \to \kappa$
identifies (objectwise) with $F(1_\mathscr{V}) \simeq 1_\mathscr{V} \otimes F(1_\mathscr{V}) \to F(1_\mathscr{V} \otimes 1_\mathscr{V}) \simeq F(1_\mathscr{V})$ which is an equivalence because lax $\mathscr{V}$-linear functors satisfy a unitality axiom. It is easy to see that the composite $\iota \to \iota\kappa \iota \to \iota$ is homotopic to the identity. 
\end{proof}
\begin{cor}
The functor $\kappa$ in Lemma \ref{lem:LaxLinV} acquires a canonical lax monoidal structure, for the composition monoidal structure on $\operatorname{Fun}^{\mathrm{lax}}_\mathscr{V}(\mathscr{V},\mathscr{V})$. Further, there is an induced adjunction on algebra objects. 
\end{cor}
\begin{proof}
It is clear that $\iota$ is strongly monoidal. Therefore $\kappa$ acquires a canonical lax monoidal structure by \cite[Corollary 7.3.2.7]{HigherAlgebra} or \cite[Theorem A]{HausgengLax}. 
\end{proof}
This Lemma has an obvious oplax version.
\begin{lem}
There is an adjunction of $\infty$-categories 
\begin{equation}
\kappa^\prime:\operatorname{Fun}^{\mathrm{oplax}}_\mathscr{V}(\mathscr{V},\mathscr{V})\leftrightarrows  \mathscr{V}: \iota^\prime
\end{equation}
in which the right adjoint $\iota^\prime$ sends $V \mapsto - \otimes V$ and the left adjoint $\kappa^\prime$ sends $F \mapsto F(1_\mathscr{V})$. (Note again here that the superscript \emph{oplax} stands for \emph{oplax $\mathscr{V}$-linear} functors and \emph{not} oplax monoidal functors). The functor $\kappa^\prime$ acquires a canonical oplax monoidal structure, for the composition monoidal structure on $\operatorname{Fun}^{\mathrm{oplax}}_\mathscr{V}(\mathscr{V},\mathscr{V})$. Further, there is an induced adjunction on coalgebra objects.
\end{lem}
Now let $A \in \mathsf{Alg}(\mathscr{V})$ be an algebra object. Then $\operatorname{RMod}_A\mathscr{V}$ is a $\mathscr{V}$-linear category. We can formulate the following generalization of Lemma \ref{lem:LaxLinV}. 
\begin{lem}\label{lem:Vlinearcattheorem}
Let $\mathscr{M}$ be a $\mathscr{V}$-linear category. There is an adjunction of $\infty$-categories
\begin{equation}
    \iota: \operatorname{LMod}_A\mathscr{M} \leftrightarrows \operatorname{Fun}^\mathrm{lax}_\mathscr{V}(\operatorname{RMod}_A \mathscr{V},\mathscr{M}): \kappa
\end{equation}
in which the underlying object of $\kappa(F)$ is $F(A)$ and the left adjoint $\iota$ sends $M \mapsto (V \mapsto V\otimes_A M)$.
\end{lem}
\begin{proof}
The functor $\kappa$ is explicitly defined as follows. We first note that the construction of \cite[Remark 4.6.2.9]{HigherAlgebra} works just as well for lax-linear functors, so that there is a canonical functor
\begin{equation}
\operatorname{Fun}^\mathrm{lax}_\mathscr{V}(\operatorname{RMod}_A \mathscr{V},\mathscr{M}) \to \operatorname{Fun}_\mathscr{V}(\operatorname{LMod}_A(\operatorname{RMod_A\mathscr{V}}), \operatorname{LMod}_A\mathscr{M})
\end{equation}
then evaluation on the object $A$ (considered as a left and right $A$-module) gives the required functor $\kappa$. The unit morphism $\operatorname{id} \to \kappa\iota$ is the canonical equivalence $\operatorname{id} \simeq A \otimes_A \operatorname{id}$. The counit morphism $\kappa \iota \to \operatorname{id}$ is the composite
\begin{equation}
    (-) \otimes_A F(A) \simeq \big|(-) \otimes A^{\otimes \bullet} \otimes F(A)\big| \to \big|F(- \otimes A^{\otimes \bullet+1})\big|
    \to F\big(|- \otimes A^{\otimes \bullet+1}|\big)  \simeq F(-),
\end{equation}
where we used the bar construction of \cite[\S 4.4.2]{HigherAlgebra} and lax linearity of $F$. It is easy to see that the composite $\iota \to \iota \kappa \iota \to \iota$ is homotopic to the identity. We check the composite $\kappa \to \kappa\iota \kappa \to \kappa$. This identifies with
\begin{equation}
F(A) \simeq A \otimes_A F(A) \simeq |A^{\otimes\bullet+1} \otimes F(A)| \to |F(A^{\otimes \bullet +2}) |\simeq F(A).
\end{equation}
The morphism $|A^{\otimes \bullet +1} \otimes F(A)| \to |F(A^{\otimes \bullet+2})|$ is equivalent to $1_\mathscr{V} \otimes F(A) \to F(1_\mathscr{V} \otimes A)$ which is an equivalence by the unitality property of lax $\mathscr{V}$-linear functors.   
\end{proof}
Specialising to the situation of $\mathscr{M} = \operatorname{RMod}_A\mathscr{V}$ gives the following.
\begin{cor}\label{cor:Bimodmonoidal}
There is an adjunction of $\infty$-categories
\begin{equation}
    \iota: {}_A\operatorname{BMod}_A \mathscr{V} \leftrightarrows \operatorname{Fun}^\mathrm{lax}_\mathscr{V}(\operatorname{RMod}_A \mathscr{V}, \operatorname{RMod}_A \mathscr{V}): \kappa 
\end{equation}
in which the left adjoint $\iota$ is strongly monoidal and the right adjoint $\kappa$ is lax monoidal (for the convolution monoidal structure on bimodules and the composition monoidal structure on endofunctors). Further, there is an induced adjunction on algebra objects. 
\end{cor}
\begin{proof}
It is well-known that the functor $\iota$ is strongly monoidal. Hence $\kappa$ acquires a canonical lax monoidal structure by \cite[Corollary 7.3.2.7]{HigherAlgebra} or \cite[Theorem A]{HausgengLax}. 
\end{proof}
It seems possible that there is an ``oplax" version of Corollary \ref{cor:Bimodmonoidal} in which $A$ is replaced with a coalgebra object, and one uses bicomodules instead of bimodules. 

\begin{example}
Let $\mathscr{C}$ be a $\mathscr{V}$-linear category, and let $F: \operatorname{RMod}_A \mathscr{V} \leftrightarrows \mathscr{C} : G$ be an adjunction in which the left adjoint $F$ is $\mathscr{V}$-linear. Then by Theorem \ref{HaugsengLaxTheoremC} the right adjoint $G$ acquires a canonical lax $\mathscr{V}$-linear structure, so that the endofunctor $GF$ of $\operatorname{RMod}_A\mathscr{V}$ acquires a canonical lax $\mathscr{V}$-linear structure. By Corollary \ref{cor:Bimodmonoidal} the object $GF(A)$ acquires structure of an $A$-$A$ bimodule object equipped with a natural transformation
\begin{equation}\label{eq:GFnaturaltransf}
    GF(A) \otimes_A (-) \to GF(-)
\end{equation}
of endofunctors of $\operatorname{RMod}_A\mathscr{V}$, coming from the counit $\iota \kappa \to \operatorname{id}$. Further, since $GF$ is a monad (i.e., an algebra object in endofunctors), Corollary \ref{cor:Bimodmonoidal} implies that $GF(A)$ acquires the structure of an algebra object under convolution and the natural transformation \eqref{eq:GFnaturaltransf} is a morphism of monads.  
\end{example}

\section{Noncommutative notion of descendability}\label{subsec:NCdescent}

This section is the main technical ingredient to prove the descent results for $\wideparen{\mathcal{D}}$-modules in \S\ref{sec:DescentforDCap}. It is a mild generalization of the results of \cite{MathewGalois}. Let $\mathscr{V}$ be a a presentably symmetric monoidal stable $\infty$-category. Let $\operatorname{Alg}(\mathscr{V})$ be the category of associative algebra objects in $\mathscr{V}$. The constructions in this section are phrased in terms of \emph{left} modules, but have obvious counterparts for right modules. As usual, for an algebra object $A \in \mathsf{Alg}(\mathscr{V})$, we denote by $\operatorname{LMod}_A := \operatorname{LMod}_A\mathscr{V}$ the $\infty$-category of \emph{left module objects} over $A$ \cite[\S 4.2]{HigherAlgebra}. Let $f^\#:A \to B$ be a morphism in $\operatorname{Alg}(\mathscr{V})$. We obtain a forgetful functor
\begin{equation}
    f_*: \operatorname{LMod}_B \to \operatorname{LMod}_A.
\end{equation}
By \cite[Corollary 4.2.3.2]{HigherAlgebra}, there is a functor $\operatorname{Alg}(\mathscr{V})^\mathsf{op} \to \mathsf{Cat}_\infty$ which sends $A \mapsto \operatorname{LMod}_A$ and each $f^\# : A \to B$ to the forgetful functor $f_*$.

The functor $f_*$ has a left adjoint by the following construction. We may regard $B$ as an object of ${}_B\operatorname{BMod}_B$, the category of $B$-$B$ bimodule objects in $\mathscr{V}$. (In this section, we will always regard bimodule objects as a monoidal category with respect to \emph{convolution} \cite[\S 4.4.3]{HigherAlgebra}). There is a functor 
\begin{equation}
    {}_B\operatorname{BMod}_B \to _{B}\operatorname{BMod}_A \simeq \operatorname{Fun}^L_\mathscr{V}(\operatorname{LMod}_A, \operatorname{LMod}_B), 
\end{equation}
the image of $B$ under this composite is the left adjoint $f^*$ to $f_*$. By \cite[Proposition 4.6.2.17]{HigherAlgebra} the formation of these pullback functors can be arranged in a functorial way. Accordingly we obtain a comonad $f^*f_* = B\otimes_A-$ on $\operatorname{LMod}_B$. 
\begin{defn}\label{defn:descendableNC}
With notations as above. We let $\langle B\rangle$ denote smallest full subcategory of $_A\operatorname{BMod}_A$ which contains $B$ and is stable under finite (co)limits, retracts, and convolution. We say that $f^{\#}:A \to B$ is \emph{descendable} if $A \in \langle B \rangle$. 
\end{defn}
\begin{lem}\label{lem:ncdescendable}
Suppose that $f^\#: A \to B$ is descendable. Then the adjunction $f^* \dashv f_*$ is comonadic. 
\end{lem}
Before the proof of this Lemma we recall some properties of $\operatorname{Pro}$- and $\operatorname{Ind}$-objects. 
\begin{defn}\cite[\S 3]{MathewGalois} Let $\mathscr{C}$ be an $\infty$-category with finite limits and colimits.
Let $\operatorname{Ind}(\mathscr{C})$ and $\operatorname{Pro}(\mathscr{C}) := \operatorname{Ind}(\mathscr{C}^\mathsf{op})^\mathsf{op}$ be the $\operatorname{Pro}$- and $\operatorname{Ind}$-categories of $\mathscr{C}$, respectively. 
\begin{enumerate}[(i)]
    \item An object $M \in \operatorname{Ind}(\mathscr{C})$ is called a \emph{constant $\operatorname{Ind}$-object} of $\mathscr{C}$ if it belongs to the essential image of the Yoneda embedding $\mathscr{C} \to \operatorname{Ind}(\mathscr{C})$. 
    \item An object $M \in \operatorname{Pro}(\mathscr{C})$ is called a \emph{constant $\operatorname{Pro}$-object} of $\mathscr{C}$ if it belongs to the essential image of the Yoneda embedding $\mathscr{C} \to \operatorname{Pro}(\mathscr{C})$. 
\end{enumerate}
\end{defn}
\begin{lem}\cite[\S 3]{MathewGalois}. Let $\mathscr{C}$ be an $\infty$-category with finite limits and colimits.
\begin{enumerate}[(i)]
    \item Let $p: I \to \mathscr{C}$ be a filtered diagram. The following are equivalent:
    \begin{enumerate}
        \item The object ${\indlim} p \in \operatorname{Ind}(\mathscr{C})$ is a constant $\operatorname{Ind}$-object;
        \item For every $\infty$-category $\mathscr{D}$ and every functor $F: \mathscr{C} \to \mathscr{D}$ which preserves finite colimits, the colimit of $p$ is preserved by $F$.
    \end{enumerate}
    \item Let $p: I \to \mathscr{C}$ be a cofiltered diagram. The following are equivalent:
    \begin{enumerate}
        \item The object ${\prolim} p \in \operatorname{Pro}(\mathscr{C})$ is a constant $\operatorname{Pro}$-object;
        \item For every $\infty$-category $\mathscr{D}$ and every functor $F: \mathscr{C} \to \mathscr{D}$ which preserves finite limits, the limit of $p$ is preserved by $F$.
    \end{enumerate}
\end{enumerate}
\end{lem}
\begin{example}\cite[Example 3.11]{MathewGalois}. Let $\mathscr{C}$ be an $\infty$-category with finite limits and colimits.
\begin{enumerate}[(i)]
    \item If $M_\bullet$ is a split simplicial object of $\mathscr{C}$, then the system of its $n$-skeleta $\operatorname{sk}_n (M_\bullet) := \operatorname{colim}_{[m] \in \Delta_{\leqslant n}^\mathsf{op}} M_m$ forms a constant $\operatorname{Ind}$-object with $\operatorname{colim}_n \operatorname{sk}_n (M_\bullet) = \operatorname{colim}_{[m] \in \Delta^\mathsf{op}} M_\bullet$.  
    \item If $M^\bullet$ is a split cosimplicial object of $\mathscr{C}$, then the system of its partial totalizations  $\operatorname{Tot}_n (M^\bullet) := \operatorname{lim}_{[m] \in \Delta_{\leqslant n}} M^m$ forms a constant $\operatorname{Pro}$-object with $\operatorname{lim}_n \operatorname{Tot}_n (M^\bullet) = \operatorname{lim}_{[m] \in \Delta} M^\bullet$.  
\end{enumerate}
\end{example}
\begin{proof}[Proof of Lemma \ref{lem:ncdescendable}] This is essentially the same as \cite[Proposition 2.6.3]{mann_p-adic_2022}. We check the hypotheses of the Barr--Beck--Lurie theorem. Let $M^\bullet \in \operatorname{LMod}_A$ be a $f^*$-split cosimplicial object. In particular the $\operatorname{Tot}$-tower of $f_*f^*M^\bullet = B \otimes_A M^\bullet$ is a constant Pro-object. It follows (using that $\mathscr{V}$ is stable) that the $\operatorname{Tot}$-tower of $N \otimes_A M^\bullet$ is constant for every $N \in \langle B \rangle$. By taking $N = A \in \langle B \rangle$ one deduces that the $\operatorname{Tot}$-tower of $M^\bullet$ is a constant Pro-object. In this case it is clear that $f^*$ commutes with the totalization of $M^\bullet$: as the relevant categories are stable, $f^*$ preserves finite limits. 

It remains to check that $f^*$ is conservative. Suppose that $f^*M \simeq 0$, then $f_*f^*M = B \otimes_A M \simeq 0$ and hence $N \otimes_A M \simeq 0$ for every $N \in \langle B \rangle$. Taking $N = A \in \langle B \rangle $ then implies that $M \simeq 0$. 
\end{proof}
Now let $A^\bullet$ be an augmented cosimplicial object in $\operatorname{Alg}(\mathscr{V})$. We obtain a functor
\begin{equation}
    \begin{aligned}
    N(\Delta_+) \to \mathsf{Cat}_\infty : [n] \mapsto \operatorname{LMod}_{A^n},
    \end{aligned}
\end{equation}
which sends every morphism $[m] \to [n]$ in $\Delta_+$ to the corresponding pullback functor $\operatorname{LMod}_{A^m} \to \operatorname{LMod}_{A^n}$ defined as in the preceding section. In the next definition we will use the following notation. For each $[n] \in  \Delta_+$ we denote by $d^0$ the injective morphism $d^0: [n] \hookrightarrow [0]\star [n]\simeq  [n+1]$ which omits $0$ from its image. For any morphism $\alpha: [m] \to [n]$ in $\Delta_+$ there is an obvious cosimplicial morphism $\alpha^\prime : [m+1] \to [n+1]$ such that $\alpha^\prime d^0 = d^0 \alpha$.  
\begin{defn}\label{defn:Beckchevalley}
We say that $A^\bullet$ \emph{satisfies the Beck-Chevalley condition} if for every map $\alpha: [m] \to [n]$ in $\Delta_+$, the natural morphism
\begin{equation}
    A^n \otimes_{A^m} A^{m+1} \to A^{n+1}
\end{equation}
is an equivalence in $_{A^n}\operatorname{BMod}_{A^{m+1}}$. Here, the algebra morphisms $A^m \to A^{m+1}$ and $A^n \to A^{n+1}$ are induced by $d^0$. The morphisms $A^m \to A^n$ and $A^{m+1} \to A^{n+1}$ are induced by $\alpha$ and $\alpha^\prime$.  
\end{defn}
\begin{lem}\label{lem:NCdescent2}
Suppose that $A^\bullet$ satisfies the Beck-Chevalley condition and $A^{-1} \to A^0$ is descendable. Then the canonical morphism
\begin{equation}
    \operatorname{LMod}_{A^{-1}} \to \underset{[n] \in \Delta}{\operatorname{lim}}\operatorname{LMod}_{A^n}
\end{equation}
is an equivalence of $\infty$-categories. 
\end{lem}
\begin{proof}
Follows from Lemma \ref{lem:ncdescendable} above and \cite[Corollary 4.7.5.3]{HigherAlgebra}. 
\end{proof}
\begin{example}
Suppose that $A \to B$ is a morphism of \emph{commutative} algebra objects in $\mathscr{V}$. Then $A^{\otimes_B (\bullet +1)}$ is an augmented cosimplicial algebra object which satisfies the Beck--Chevalley condition. It is obvious that $A \to B$ is descendable in the sense of Definition \ref{defn:descendableNC} if and only if $A \to B$ is descendable in the sense of Mathew \cite[\S 3.3]{MathewGalois}. Hence Lemma \ref{lem:NCdescent2} recovers Mathew's theorem \cite[Proposition 3.22]{MathewGalois}.
\end{example}
\section{Barr--Beck--Lurie in families}\label{subsec:BarrBeckFamilies}
In this section we present a generalization of the result of \cite[Proposition 4.4.5]{HausgengInftyOperads} which is adapted to our setting.
\begin{prop}\label{prop:BarrBeckFamilies}
Given a diagram 
\begin{equation}
\begin{tikzcd}[cramped]
	{\mathscr{C}} && {\mathscr{D}} \\
	& {\mathscr{B}}
	\arrow["U", from=1-1, to=1-3]
	\arrow["p"', from=1-1, to=2-2]
	\arrow["r", from=1-3, to=2-2]
\end{tikzcd}
\end{equation}
in $\mathsf{Cat}_\infty$ such that:
\begin{enumerate}[(i)]
    \item $p$ and $r$ are coCartesian fibrations and $U$ preserves coCartesian edges;
    \item $U$ has a left adjoint $F: \mathscr{D} \to \mathscr{C}$ such that $pF \simeq r$;
    \item The adjunction $F \dashv U$ restricts in each fiber to an adjunction $F_b \dashv U_b$. For all $b \in \mathscr{B}$, the functor $U_b$ is conservative, and $\mathscr{C}_b$ admits colimits of $U_b$-split simplicial objects,  which $U_b$ preserves. 

    \item For any edge $e: b \to b^\prime$ in $\mathscr{B}$, the coCartesian covariant transport $e_!: \mathscr{C}_{b} \to \mathscr{C}_{b^\prime}$ preserves geometric realizations of $U_{b}$-split simplicial objects.
\end{enumerate}
Then, the adjunction $F \dashv U$ is monadic. 
\end{prop}
\begin{rmk}
In view of the Barr--Beck--Lurie theorem, condition (iii) in Proposition \ref{prop:BarrBeckFamilies} is equivalent to:
\begin{enumerate}[(i)]
    \item[(iii)${}^\prime$] The adjunction $F \dashv U$ restricts in each fiber to a monadic adjunction $F_b \dashv U_b$.
\end{enumerate}
\end{rmk}
\begin{proof}[Proof of Proposition \ref{prop:BarrBeckFamilies}.]
We verify the conditions of the Barr--Beck--Lurie theorem \cite[Theorem 4.7.3.5]{HigherAlgebra}.

First we show that $U$ is conservative. We can argue in exactly the same way as \cite[Proposition 4.4.5]{HausgengInftyOperads}. Suppose that $f: c \to c^\prime$ is a morphism in $\mathscr{C}$ such that $Uf$ is an equivalence in $\mathscr{D}$. Then $e:=qUf \simeq pf$ is an equivalence in $\mathscr{B}$. One can factor $f$ as $c \xrightarrow[]{\varphi} e_!c \xrightarrow[]{f^\prime} c^\prime$ where $\varphi$ is a coCartesian lift of $e$ and $f^\prime$ is a morphism in the fiber $\mathscr{C}_{b^\prime}$ above $b^\prime := p(c^\prime)$. Since $\varphi$ is coCartesian lift of an equivalence, it is an equivalence. Because of the fiberwise monadicity assumption (iii), $f^\prime$ is an equivalence. Therefore $f$ is an equivalence and $U$ is conservative. 

Now let us show that $\mathscr{C}$ admits and $U$ preserves colimits of $U$-split simplicial objects. Let $q: \Delta^\mathsf{op} \to \mathscr{C}$ be a $U$-split simplicial object, so that $Uq$ extends to a diagram $\widetilde{Uq}: \Delta^\mathsf{op}_{-\infty} \to \mathscr{D}$. Let $f: \Delta^\mathsf{op}_{-\infty} \to \mathscr{B}$ be the underlying diagram in $\mathscr{B}$. There is a morphism 
\begin{equation}
    \Delta^1  \times \Delta_{-\infty}^\mathsf{op}  \to \Delta_{-\infty}^\mathsf{op}
\end{equation}
which is the identity on $ \{0\} \times \Delta_{-\infty}^\mathsf{op}$ and carries $\{1\} \times \Delta_{-\infty}^\mathsf{op} $ to $[-1] \in \Delta_{-\infty}^\mathsf{op}$. It sends each horizontal morphism $\{0\} \times 
 [n]\to \{1\}  \times [n] $ to the unique morphism $[n] \to [-1]$.
Consider the composite 
\begin{equation}
   P: \Delta^1  \times \Delta_{-\infty}^\mathsf{op}  \to \Delta_{-\infty}^\mathsf{op} \xrightarrow[]{f} \mathscr{B}.
    \end{equation}
Now we will take a coCartesian lifts, using the exponentiation for coCartesian fibrations \cite[\href{https://kerodon.net/tag/01VG}{Tag 01VG}]{kerodon}. 
\begin{itemize}
    \item[$\star$] Let $Q$ be a coCartesian lift of $\left.P\right|_{\Delta^1 \times  \Delta^\mathsf{op}}$ to $\mathscr{C}$. Then $Q$ is a natural transformation between $q$ and a morphism $q^\prime: \Delta^\mathsf{op} \to \mathscr{C}_b$, where $b$ is the image under $f$ of $[-1] \in \Delta_{-\infty}^\mathsf{op} $.
    \item[$\star$] Let $\widetilde{UQ}$ be a coCartesian lift of $P$ to $\mathscr{D}$. Then $\widetilde{UQ}$ is a natural transformation between $\widetilde{Uq}$ and a morphism $\widetilde{Uq^\prime}: \Delta_{-\infty}^\mathsf{op} \to \mathscr{C}_b$.
\end{itemize}
These natural transformations $Q$ and $\widetilde{UQ}$ are uniquely characterised by the property that their components are coCartesian edges \cite[\href{https://kerodon.net/tag/01VG}{Tag 01VG}]{kerodon}. Because of the assumption (i) that $U$ preserves coCartesian edges, this unicity implies that $UQ \simeq \left.\widetilde{UQ}\right|_{\Delta^1 \times \Delta^\mathsf{op}}$. In particular $U q^\prime: \Delta^\mathsf{op} \to \mathscr{C}_b$ extends to the split simplicial object $\widetilde{Uq^\prime}: \Delta^\mathsf{op}_{-\infty} \to \mathscr{C}_b$. By the fiberwise monadicity assumption (iii), this implies that $q^\prime$ extends to a colimit diagram $\overline{q}^\prime: (\Delta^\mathsf{op})^{\triangleright} \to \mathscr{C}_b$ such that $U\overline{q}^\prime$ is also a colimit diagram. By assumption (iv) and \cite[Proposition 4.3.1.10]{HigherToposTheory} it then follows that $\overline{q}^\prime$ (resp. $U \overline{q}^\prime$), when regarded as a diagram in $\mathscr{C}$ (resp. $\mathscr{D}$), is a $p$-colimit diagram (resp. $r$-colimit diagram). Now we can argue as in \cite[Corollary 4.3.1.11]{HigherToposTheory}. We have a commutative diagram
\begin{equation}
\begin{tikzcd}[column sep=large]
	{(\Delta^1 \times \Delta^\mathsf{op})\coprod_{\{1\} \times \Delta^\mathsf{op}}(\{1\}\times (\Delta^\mathsf{op})^\triangleright)} & {\mathscr{C}} \\
	{(\Delta^1 \times \Delta^\mathsf{op} )^\triangleright} & {\mathscr{B}}
	\arrow["{(Q, \overline{q}^\prime)}", from=1-1, to=1-2]
	\arrow[hook, from=1-1, to=2-1]
	\arrow["p", from=1-2, to=2-2]
	\arrow["s"{description}, dashed, from=2-1, to=1-2]
	\arrow["{(\left.f\right|_{(\Delta^\mathsf{op})^\triangleright})\circ\pi}"', from=2-1, to=2-2]
\end{tikzcd}
\end{equation}
in which $\pi: (\Delta^1 \times \Delta^\mathsf{op
})^{\triangleright} \to (\Delta^\mathsf{op})^\triangleright = \Delta^\mathsf{op}_+ \subseteq \Delta^\mathsf{op}_{-\infty}$ denotes the morphism which is the identity on $\{0\} \times \Delta^\mathsf{op}$ and which carries $(\{1\} \times \Delta^\mathsf{op})^\triangleright$ to the cone point. Because the left map is an inner fibration there exists a lift $s$ as indicated by the dashed arrow. Consider now the map $\Delta^1 \times (\Delta^\mathsf{op})^\triangleright \to (\Delta^1 \times \Delta^\mathsf{op})^\triangleright$ which is the identity on $\Delta^1 \times \Delta^\mathsf{op}$ and carries the other vertices of $ \Delta^1 \times (\Delta^\mathsf{op})^\triangleright $ to the cone point. Let $\overline{Q}$ denote the composition
\begin{equation}
    \Delta^1 \times (\Delta^\mathsf{op})^\triangleright \to (\Delta^1 \times \Delta^\mathsf{op})^\triangleright \xrightarrow[]{s} \mathscr{C}
\end{equation}
and define $\overline{q}:= \left.\overline{Q}\right|_{\{0\}\times (\Delta^\mathsf{op})^\triangleright}$. Then $\overline{Q}$ is a natural transformation from $\overline{q}$ to $\overline{q}^\prime$ which is componentwise coCartesian. Then \cite[Proposition 4.3.1.9]{HigherToposTheory} implies that $\overline{q}$ is a $p$-colimit diagram which fits into the diagram
\begin{equation}
\begin{tikzcd}[cramped,column sep=large]
	{\Delta^\mathsf{op}} & {\mathscr{C}} \\
	{(\Delta^\mathsf{op})^{\triangleright}} & {\mathscr{B}}
	\arrow["q", from=1-1, to=1-2]
	\arrow[hook, from=1-1, to=2-1]
	\arrow["p", from=1-2, to=2-2]
	\arrow["{\overline{q}}"{description}, from=2-1, to=1-2]
	\arrow["{\left.f\right|_{(\Delta^\mathsf{op})^\triangleright}}"', from=2-1, to=2-2]
\end{tikzcd}
\end{equation} 
By assumption (i), $U\overline{Q}$ is a natural transformation from  $U\overline{q}$ to $U\overline{q}^\prime$ which is componentwise coCartesian. Hence \cite[Proposition 4.3.1.9]{HigherToposTheory} implies that $U\overline{q}$ is an $r$-colimit diagram. The underlying diagram $\left.f\right|_{(\Delta^\mathsf{op})^\triangleright}$ of $\overline{q}$ in $\mathscr{B}$ extends to the split simplicial diagram $f$ and hence admits a colimit in $\mathscr{B}$. Hence \cite[Proposition 4.3.1.5(2)]{HigherToposTheory} implies that $\overline{q}$ and $U\overline{q}$ are colimit diagrams in $\mathscr{C}$ and $\mathscr{D}$ respectively. Hence $\mathscr{C}$ admits and $U$ preserves geometric realizations of $U$-split simplicial objects.
\end{proof}
\begin{cor}\label{cor:BarrBeckFamiliesCor}
Let $\mathscr{B}$ be an $\infty$-category and let $\eta : F \to G$ be a natural transformation of functors $F, G : \mathscr{B} \to \mathsf{Cat}_\infty$. Assume that:
\begin{enumerate}[(i)]
    \item For each $b \in \mathscr{B}$ the functor $\eta_b : F(b) \to G(b)$ admits a left adjoint $\nu_b$;
    \item For each $b \in \mathscr{B}$, the functor $\eta_b$ is conservative, and $F(b)$ admits colimits of $\eta_b$-split simplicial objects, which $\eta_b$ preserves;
    \item For any edge $e: b \to b^\prime$ in $\mathscr{B}$, the functor $F(e): F(b) \to F(b^\prime)$ preserves geometric realizations of $\eta_b$-split simplicial objects. 
\end{enumerate}
Then there exists a functor $H: \mathscr{B} \to \mathsf{Cat}_\infty$ equipped with a natural equivalence $\lambda: F \xrightarrow[]{\sim} H$ and a natural transformation $\mu: H \to G$ such that:
\begin{enumerate}[(i)]
    \item There is an equivalence $\mu \circ \lambda \simeq \eta$;
    \item Set $T_b := \eta_b \nu_b$. Then for each $b \in \mathscr{B}$ one has $H(b) = \operatorname{Mod}_{T_b}F(b)$ and 
    \begin{equation}
        F(b) \xrightarrow[]{\simeq \lambda(b)} H(b) \xrightarrow[]{\mu(b)} G(b)
    \end{equation}
    identifies with the factorization of $\eta(b)$ as the \emph{comparison functor} followed by the \emph{forgetful functor}. 
\end{enumerate}
\end{cor}
\begin{proof}
This follows from Proposition \ref{prop:BarrBeckFamilies} and straightening. Strictly speaking, an application of (the dual of) \cite[Proposition 7.3.2.6]{HigherAlgebra} is required. 
\end{proof}
\bibliography{references}

\begin{thebibliography}{BBKK24}

\bibitem[AW19]{Dcap1}
Konstantin Ardakov and Simon Wadsley.
\newblock {$\wideparen{\mathcal{D}}$}-modules on rigid analytic spaces {I}.
\newblock {\em J. Reine Angew. Math.}, 747:221--275, 2019.

\bibitem[BBB16]{BambozziDaggerBanach}
Federico Bambozzi and Oren Ben-Bassat.
\newblock Dagger geometry as {B}anach algebraic geometry.
\newblock {\em J. Number Theory}, 162:391--462, 2016.

\bibitem[BBK17]{BBKNonArch}
Oren Ben-Bassat and Kobi Kremnizer.
\newblock Non-archimedean analytic geometry as relative algebraic geometry.
\newblock {\em Ann. Fac. Sci. Toulouse Math. (6)}, 26(1), 2017.

\bibitem[BBKK24]{DAnG}
Oren Ben-Bassat, Jack Kelly, and Kobi Kremnizer.
\newblock A {Perspective} on the {Foundations} of {Derived} {Analytic}
  {Geometry}, May 2024.
\newblock arXiv:2405.07936 [math].

\bibitem[Bod21]{bode_six_2021}
Andreas Bode.
\newblock Six operations for {$\wideparen{\mathcal D}$}-modules on rigid
  analytic spaces, November 2021.
\newblock arXiv:2110.09398 [math].

\bibitem[Cam24]{Camargo_deRham}
Juan Esteban~Rodríguez Camargo.
\newblock The analytic de {Rham} stack in rigid geometry, January 2024.
\newblock arXiv:2401.07738 [math].

\bibitem[GHK22]{HausgengInftyOperads}
David Gepner, Rune Haugseng, and Joachim Kock.
\newblock {$\infty$}-operads as analytic monads.
\newblock {\em Int. Math. Res. Not. IMRN}, (16):12516--12624, 2022.

\bibitem[GR14]{MR3264953}
Dennis Gaitsgory and Nick Rozenblyum.
\newblock Crystals and $\mathcal{D}$-modules.
\newblock {\em Pure Appl. Math. Q.}, 10(1):57--154, 2014.

\bibitem[HHLN23]{HausgengLax}
Rune Haugseng, Fabian Hebestreit, Sil Linskens, and Joost Nuiten.
\newblock Lax monoidal adjunctions, two-variable fibrations and the calculus of
  mates.
\newblock {\em Proc. Lond. Math. Soc. (3)}, 127(4):889--957, 2023.

\bibitem[HM24]{heyer_6-functor_2024}
Claudius Heyer and Lucas Mann.
\newblock 6-{Functor} {Formalisms} and {Smooth} {Representations}, October
  2024.
\newblock arXiv:2410.13038 [math].

\bibitem[HN70]{HogbeNlendThesis}
Henri Hogbe-Nlend.
\newblock Completion, tenseurs et nuclearite en bornologie.
\newblock {\em J. Math. Pures Appl. (9)}, 49:193--288, 1970.

\bibitem[Kel24]{kelly_homotopy_2021}
Jack Kelly.
\newblock Homotopy in exact categories.
\newblock {\em Mem. Amer. Math. Soc.}, 298(1490):v+160, 2024.

\bibitem[Lur09]{HigherToposTheory}
Jacob Lurie.
\newblock {\em Higher topos theory}, volume 170 of {\em Annals of Mathematics
  Studies}.
\newblock Princeton University Press, Princeton, NJ, 2009.

\bibitem[Lur17]{HigherAlgebra}
Jacob Lurie.
\newblock Higher {A}lgebra.
\newblock Available at \url{https://www.math.ias.edu/~lurie/papers/HA.pdf},
  2017.

\bibitem[Lur18]{kerodon}
Jacob Lurie.
\newblock Kerodon.
\newblock \url{https://kerodon.net}, 2018.

\bibitem[Man22]{mann_p-adic_2022}
Lucas Mann.
\newblock A $p$-{adic} 6-{functor} {formalism} in {rigid}-{analytic}
  {geometry}, June 2022.
\newblock arXiv:2206.02022 [math].

\bibitem[Mat16]{MathewGalois}
Akhil Mathew.
\newblock The {G}alois group of a stable homotopy theory.
\newblock {\em Adv. Math.}, 291:403--541, 2016.

\bibitem[PS00]{ProsmansHomological}
Fabienne Prosmans and Jean-Pierre Schneiders.
\newblock A homological study of bornological spaces.
\newblock Laboratoire analyse, géométrie et applications, Unité mixte de
  recherche, Institut Galilee, Université Paris, CNRS, 2000.

\bibitem[Sch99]{schneiders_quasi-abelian_1999}
Jean-Pierre Schneiders.
\newblock Quasi-abelian categories and sheaves.
\newblock {\em Mémoires de la Société Mathématique de France},
  76:III1--VI140, 1999.
\newblock Publisher: Société mathématique de France.

\bibitem[Sch22]{ScholzeSixFunctors}
Peter Scholze.
\newblock {S}ix-{F}unctor {F}ormalisms.
\newblock Available at
  \url{https://people.mpim-bonn.mpg.de/scholze/SixFunctors.pdf}, 2022.

\bibitem[Soo23]{soor_quasicoherent_2023}
Arun Soor.
\newblock Quasicoherent sheaves for dagger analytic geometry, November 2023.
\newblock arXiv:2311.03101 [math].

\bibitem[Soo24]{soor_six-functor_2024}
Arun Soor.
\newblock A six-functor formalism for quasi-coherent sheaves and crystals on
  rigid-analytic varieties, September 2024.
\newblock arXiv:2409.07592 [math].

\bibitem[To{\"e}06]{ToenChampsAffines}
Bertrand To{\"e}n.
\newblock Champs affines.
\newblock {\em Selecta Math. (N.S.)}, 12(1):39--135, 2006.

\end{thebibliography}
\bibliographystyle{alpha}
\end{document}